\documentclass[reqno, 11pt]{amsart}

\usepackage[english]{babel}
\usepackage{amsmath}
\usepackage{amsthm}
\usepackage{amsfonts}
\usepackage{amssymb}
\usepackage{anysize}
\usepackage{hyperref}
\usepackage{appendix}
\usepackage{mathtools}
\usepackage{graphicx}
\usepackage{xcolor}
\usepackage{amsmath,esint,bigints, amssymb, mathrsfs, amsthm,mathbbol,upgreek,exscale}
\usepackage{psfrag}
\usepackage{epstopdf}
\usepackage{graphicx}
\usepackage{color}
\usepackage{transparent}

\newtheorem{defi}{Definition}[section]
\newtheorem{lem}[defi]{Lemma}
\newtheorem{theo}[defi]{Theorem}
\newtheorem{cor}[defi]{Corollary}
\newtheorem{pro}[defi]{Proposition}

\newtheorem{rem}[defi]{Remark}

\DeclareMathOperator{\N}{\mathbb{N}}
\DeclareMathOperator{\R}{\mathbb{R}}
\DeclareMathOperator{\C}{\mathbb{C}}
\DeclareMathOperator{\T}{\mathbb{T}}

\DeclareMathOperator{\Z}{\mathbb{Z}}

\allowdisplaybreaks 

\title[]{Self-similar spirals for the generalized surface quasi-geostrophic equations}

\author[C. Garc\'ia]{Claudia Garc\'ia}
\address{Departament de Matem\`atiques i Inform\`atica, Universitat de Barcelona, Gran Via de les Corts Catalanes 585, 08007 Barcelona, Spain}
\email{claudiagarcia@ub.edu}

\author[J. G\'omez-Serrano]{Javier G\'omez-Serrano}
\address{Department of Mathematics, Brown University, Kassar House, 151 Thayer St., Providence, RI 02912, USA  \vskip 0.1cm
Departament de Matem\`atiques i Inform\`atica, Universitat de Barcelona, Gran Via de les Corts Catalanes 585, 08007 Barcelona, Spain \vskip 0.1cm
Centre de Recerca Matem\`atica, Edifici C, Campus Bellaterra, 08193 Bellaterra, Spain
}
\email{javier\_gomez\_serrano@brown.edu, jgomezserrano@ub.edu}

{\thanks{C.G. and J.G.-S. have been partially supported by the European Research Council ERC-StG-852741 (CAPA) and the MICINN (Spain) research grant number PID2021--125021NA--I00. C. G. has also been partially supported by the MINECO--Feder (Spain) research grant number RTI2018--098850--B--I00, and the Junta de Andaluc\'ia (Spain) Project FQM 954. This material is based upon work supported by the National Science Foundation under Grant No. DMS--1929284 while C.G. and J.G.-S. were in residence at the Institute for Computational and Experimental Research in Mathematics in Providence, RI, during the program ``Hamiltonian Methods in Dispersive and Wave Evolution Equations''. This work is supported by the Spanish
State Research Agency, through the Severo Ochoa and Mar\'ia de Maeztu Program for Centers and Units of Excellence in R\&D (CEX2020-001084-M)}}

\begin{document}

\date{\today}

\begin{abstract}
In this paper we construct a large class of non-trivial (non-radial) self-similar solutions of the generalized surface quasi-geostrophic equation (gSQG). To the best of our knowledge, this is the first rigorous construction of any self-similar solution for these equations. The solutions are of spiral type, locally integrable, and may have mixed sign. Moreover, they bear some resemblance with the finite time singularity scenario numerically proposed by Scott and Dritschel \cite{Scott-Dritschel:self-similar-sqg-patch} in the SQG patch setting.
\end{abstract}

\maketitle

\tableofcontents

\section{Introduction}

The generalized surface quasi-geostrophic equations describing the evolution of the potential temperature $\theta$ read as

\begin{equation}\label{gSQG}
\left\{
\begin{array}{ll}
	\theta_t+(u\cdot\nabla)\theta=0, & \textnormal{in } [0,+\infty)\times\R^2,\\
	u=-\nabla^\perp(-\Delta)^{-1+\frac{\gamma}{2}}\theta, &\textnormal{in } [0,+\infty)\times \R^2,\\
	\theta(0,x)=\theta_0(x), &\textnormal{with } x\in\R^2.
\end{array}
\right.
\end{equation}
In this system $u$ refers to the velocity field, $\nabla^\perp=(-\partial_2,\partial_1)$ and $\gamma\in[0,2)$. The velocity field is linked to the potential temperature $\theta$ via the operator $-(-\Delta)^{-1+\frac{\gamma}{2}}$ agreeing with $u=\nabla^\perp\psi$ and
\begin{equation}\label{psi-theta}
\psi(t,x)=-(-\Delta)^{-1+\frac{\gamma}{2}}\theta(x)=\int_{\R^2}K_\gamma(x-y)\theta(t,y)dy,
\end{equation}
where
$$
K_\gamma(x)=\left\{\begin{array}{ll}
\frac{1}{2\pi}\log|x|, &\gamma=0,\\
-\frac{C_\gamma}{2\pi}\frac{1}{|x|^\gamma},& \gamma\in(0,2),
\end{array}\right.
$$
and $C_\gamma=\frac{\Gamma(\gamma/2)}{2^{1-\gamma}\Gamma\left(\frac{1-\gamma}{2}\right)}$.
The case $\gamma=0$ corresponds to the Euler equations and $\gamma=1$ to the surface quasi-geostrophic (SQG) model. In this work we will be interested in the case $\gamma\in(0,1)$. 

The SQG equation $(\gamma = 1)$ models the evolution of the temperature from a general quasi-geostrophic system for atmospheric and oceanic flows (see  \cite{Constantin-Majda-Tabak:formation-fronts-qg,Held-Pierrehumbert-Garner-Swanson:sqg-dynamics,Pedlosky:geophysical,Majda-Bertozzi:vorticity-incompressible-flow} for more details).
Constantin--Majda--Tabak pointed out the analogy between the SQG equation and the incompressible 3D Euler equations \cite{Constantin-Majda-Tabak:formation-fronts-qg} and carried out the first numerical and analytical study of the equation. The gSQG or (SQG)$_\gamma$  model \eqref{gSQG} was proposed by C\'ordoba--Fontelos--Mancho--Rodrigo in \cite{Cordoba-Fontelos-Mancho-Rodrigo:evidence-singularities-contour-dynamics} as an interpolation between Euler and surface quasi--geostrophic equations. Yet, almost 30 years after \cite{Constantin-Majda-Tabak:formation-fronts-qg}, the global existence vs finite time singularity question remains open, even in the case $\gamma > 0$. As a first step towards the finite time singularity direction, this paper aims to construct self-similar solutions, which will have infinite energy. We leave for future work the transformation of the aforementioned self-similar solutions into finite energy ones that may develop a finite time singularity.

\subsection{Main result}

Our main task is to construct spiral solutions that are perturbations of a stationary solution. Indeed, we will perturb appropriately the following radial one
\begin{equation}\label{trivial-radial}
\theta(t,x)=\theta_0(x)=|x|^{-\frac{1}{\mu}},
\end{equation}
for $\mu>0$. The singularity of the stationary solution \eqref{trivial-radial} depends on the parameter $\mu$. However, we need some condition on $\mu$ to ensure the existence of the associated stream function, that is, the integrability of the singular operator \eqref{psi-theta}. In particular, assuming that $\theta$ agrees with \eqref{trivial-radial}, one needs that 
\begin{equation}\label{mu-condition}
\frac12<\mu<\frac{1}{2-\gamma},
\end{equation}
to have that $\psi(x)=C_0|x|^{2-\gamma-\frac{1}{\mu}}$, where $C_0$ is defined in \eqref{C0}.

In this work, we will be interested in self-similar solutions. That means
\begin{equation}\label{self-similar-0-intro}
\theta(t,x)=\lambda^{1+\alpha-\gamma}\theta(\lambda^{1+\alpha}t,\lambda x), \quad \forall\lambda>0.
\end{equation}
By taking $\lambda=1/t^\frac{1}{1+\alpha}$, the problem reduces to find a self-similar profile $\hat{\theta}$ such that
\begin{equation}\label{theta-self-similar-intro}
\theta(t,x)=\frac{1}{t^{\frac{1+\alpha-\gamma}{1+\alpha}}}\hat{\theta}\left(\frac{x}{t^{\frac{1}{1+\alpha}}}\right).
\end{equation}
Moreover, in order to have that the trivial singular radial solution \eqref{trivial-radial} is self-similar, one needs the following condition on $\mu$ and $\alpha$:
\begin{equation}\label{mu-intro}
\frac{1}{\mu}=1+\alpha-\gamma,
\end{equation}
which, together with \eqref{mu-condition}, implies that $\alpha\in(1,1+\gamma)$. 

Our main theorem reads as follows.
\begin{theo}\label{th-intro}
Let $\gamma\in(0,1)$ and $\alpha\in(1,1+\gamma)$. Then, there exists $\varepsilon>0$ such that for any $\frac{2\pi}{m}$-periodic $\tilde{\Omega}\in B_{L^p([0,2\pi])}(1,\varepsilon)$ and $p>\frac{1}{1-\gamma}$, the initial condition
$$
\theta_0(r e^{i\vartheta})=r^{-(1+\alpha-\gamma)}\tilde{\Omega}(\vartheta),
$$
describes a self-similar solution of the type \eqref{theta-self-similar-intro} for $\textnormal{(SQG)}_\gamma$, for any $m\geq 1$.
\end{theo}

We refer to Figure \ref{Fig-1} to illustrate the kind of self-similar profiles obtained in Theorem \ref{th-intro}. Let us remark that comparing to the Euler case studied in \cite{Bressan-Murray:self-similar-euler, Elling:algebraic-spiral-solutions-2d-euler, Elling:self-similar-euler-mixed-sign-vorticity}, condition \eqref{mu-condition} gives us that $\theta_0\in L^1_{\textnormal{loc}}$. More precisely, we can prove that the initial vorticity $\theta_0$ belongs to $L^q_{\textnormal{loc}}$, with $q<\min\{p,2\mu\}$, provided that $\tilde{\Omega}\in L^p$.

\begin{center}
\begin{figure}[htbp]
\centering
\def\svgwidth{1\textwidth}
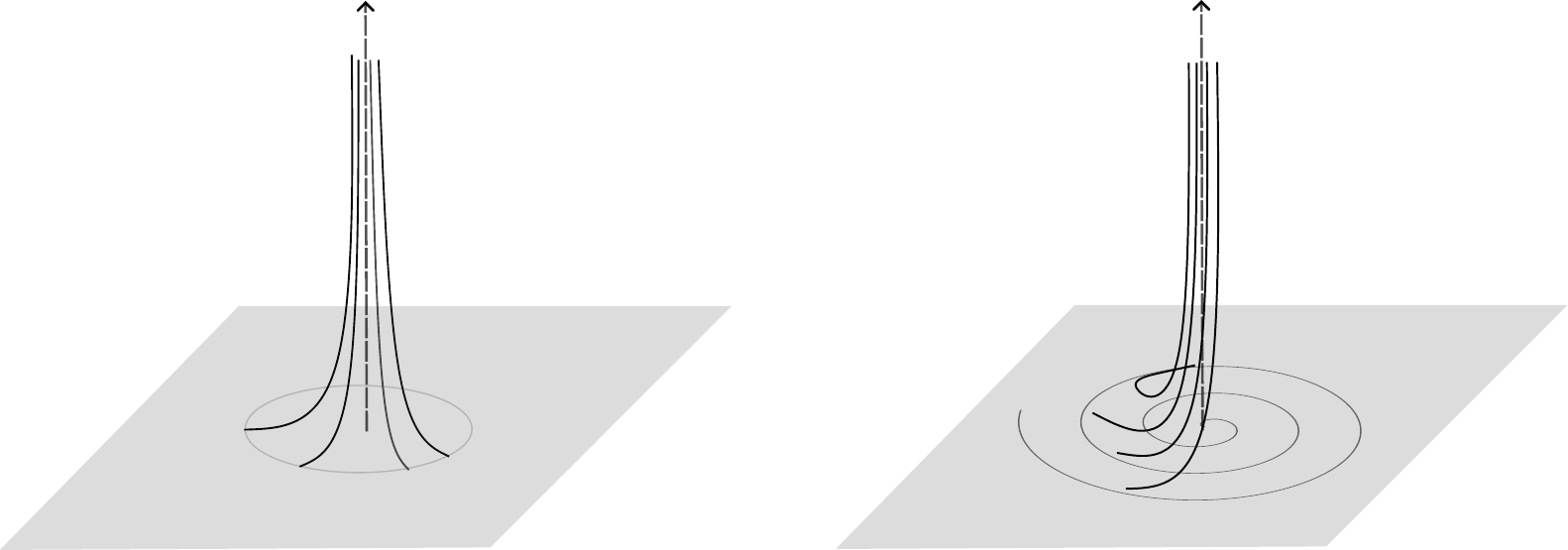
\caption{Left: the trivial self-similar radial solution \eqref{trivial-radial}.
Right: the perturbed solution of Theorem \ref{th-intro}. }
\label{Fig-1}
\end{figure}
\end{center}


We now summarize the previous literature related to this problem. We will first discuss the 2D Euler equation, then discuss the SQG and gSQG equations.

\subsection{Spiral solutions for 2D Euler}

The study of spiraling solutions for the 2D Euler equation has a long history due to its connections to turbulence and non-uniqueness, which starts from the seminal work of Kaden \cite{Kaden:spiral} and Pullin \cite{Pullin:spiral-vortex-sheet} in the context of vortex sheets. These vortex spiral solutions exhibit non-uniqueness phenomena. Elling constructed in  \cite{Elling:algebraic-spiral-solutions-2d-euler} spiral solutions close to constant in a high symmetry class, and later upgraded them to sign-changing solutions \cite{Elling:self-similar-euler-mixed-sign-vorticity}. Lately, Cieslak--Kokocki--Ozanski \cite{Cieslak-Kokocki-Ozanski:well-posedness-logarithmic-spiral-vortex-sheets,Cieslak-Kokocki-Ozanski:nonsymmetric-logarithmic-spiral-vortex-sheets} have proved local existence for logarithmic spiral vortex sheets. See also \cite{Fontelos-delaHoz:singularities-water-waves} where spirals are reported in the context of water waves.

Bressan, in a series of papers together with Shen and Murray \cite{Bressan-Murray:self-similar-euler,Bressan-Shen:a-posteriori-error-self-similar-euler} has laid out a very ambitious program to prove non-uniqueness by constructing an initial data that can be continued as either one spiral or two spirals. The idea developed in these works amounts to proving some a priori estimates and establishing the existence of some of the building blocks of the program, as well as providing compelling numerical evidence of it.

We remark that in our case, in contrast to the previous literature, we are able to achieve solutions in any symmetry class due to the study of the full linear operator, as opposed to: 1. Either a Neumann series in the symmetry class, where the contribution of one of the pieces of the linear operator dominates over the other whenever the solution is sufficiently symmetric \cite{Elling:algebraic-spiral-solutions-2d-euler,Elling:self-similar-euler-mixed-sign-vorticity}, or 2. The case of an inner domain where the problem has less constraints \cite{Bressan-Murray:self-similar-euler}). Moreover, let us emphasize one fundamental difference between the Euler case and the gSQG equations: in \cite{Bressan-Murray:self-similar-euler,Elling:algebraic-spiral-solutions-2d-euler, Elling:self-similar-euler-mixed-sign-vorticity}  the authors really exploit the fact that the relation between the stream function and the vorticity is local, which does not hold for the gSQG equations, and hence new techniques must be developed here.


In \cite{Elgindi-Jeong:singular-vortex-patches-i}, Elgindi-Jeong obtained global well-posedness for a class of patches with boundary singular at the origin which includes spirals. Later, in \cite{Elgindi-Jeong:singular-vortex-patches-ii}, the authors constructed a patch solution whose boundary forms filaments which spirals around the center of the patch, as time goes to infinity, which was suggested numerically by \cite{Dritschel:repeated-filamentation-2d-vorticity}.

\subsection{Self-similar solutions, non-uniqueness and finite time singularities for SQG and gSQG}

Resnick, in his thesis \cite{Resnick:phd-thesis-sqg-chicago}, showed global existence of weak solutions in $L^{2}$ using  the oddness of the Riesz transform to obtain an extra cancellation. Marchand \cite{Marchand:existence-regularity-weak-solutions-sqg} extended Resnick's result to the class of initial data belonging to $L^{p}$ with $p > 4/3$. 

The question of non-uniqueness for weak solutions of SQG is still a very challenging problem which has received a lot of attention lately. Starting from \cite{Azzam-Bedrossian:bmo-uniqueness-active-scalar-equations,Isett-Vicol:holder-continuous-active-scalar} up to the breakthrough by Buckmaster--Shkoller--Vicol \cite{Buckmaster-Shkoller-Vicol:nonuniqueness-sqg}, alternative newer proofs by Isett and Ma \cite{Isett-Ma:non-uniqueness-sqg} and the study of the stationary problem by Cheng--Kwon--Li \cite{Cheng-Kwon-Li:non-uniqueness-steady-state-sqg}. It is likely that the scenario crafted by Bressan and his collaborators for 2D Euler also holds in the gSQG case.

The problem of whether the SQG and gSQG system presents finite time singularities or there is global existence is one of the most outstanding problems in mathematical fluid mechanics. Kiselev and Nazarov \cite{Kiselev-Nazarov:simple-energy-pump-sqg} constructed solutions for which there was norm inflation, and Friedlander and Shvydkoy \cite{Friedlander-Shvydkoy:unstable-spectrum-sqg} showed the existence of unstable eigenvalues of the spectrum. We also refer to the construction of Castro and C\'ordoba of singular solutions with infinite energy in \cite{Castro-Cordoba:infinite-energy-sqg}.

Aside from the analytical work, there have been many numerical results in this framework. Constantin--Majda--Tabak \cite{Constantin-Majda-Tabak:formation-fronts-qg} reported a possible singularity in the form of a hyperbolic saddle closing in finite time. In contrast, Ohkitani and Yamada \cite{Ohkitani-Yamada:inviscid-limit-sqg} and Constantin--Nie--Schorghofer \cite{Constantin-Nie-Schorghofer:nonsingular-sqg-flow} suggested that the growth was double exponential instead. C\'ordoba, in his thesis \cite{Cordoba:nonexistence-hyperbolic-blowup-qg} bounded the growth by a quadruple exponential (see also later work by C\'ordoba and Fefferman on a double exponential bound \cite{Cordoba-Fefferman:growth-solutions-qg-2d-euler} and numerical simulations by Deng--Hou--Li--Yu \cite{Deng-Hou-Li-Yu:non-blowup-2d-sqg}), thus settling the question. 20 years later, using improved algorithms and more computational resources, Constantin et al. \cite{Constantin-Lai-Sharma-Tseng-Wu:new-numerics-sqg} recalculated the hyperbolic saddle scenario, finding no evidence of blowup. Scott, in \cite{Scott:scenario-singularity-quasigeostrophic}, starting from elliptical configurations, proposed a candidate that develops filamentation and after a few cascades, blowup of $\nabla \theta$. This is the only scenario that still stands up to date in the smooth setting.

A simpler setting for the SQG and gSQG problem is the so-called {\it patch setting}, where the scalar takes the form of a characteristic function of a time-dependent domain with smooth boundary. Rodrigo (in a $C^\infty$ space) \cite{Rodrigo:evolution-sharp-fronts-qg} and Gancedo \cite{Gancedo:existence-alpha-patch-sobolev} and Constantin et al \cite{Chae-Constantin-Cordoba-Gancedo-Wu:gsqg-singular-velocities} (in a Sobolev space) proved local existence for the case $0 < \gamma \leq 1$ and $1 < \gamma$ respectively.
Some 20 years ago, C\'ordoba--Fontelos--Mancho--Rodrigo \cite{Cordoba-Fontelos-Mancho-Rodrigo:evidence-singularities-contour-dynamics} found numerically strong evidence of finite time singularity formation in the form of a self-intersecting patch where at the point of intersection the curvature blows up (the curve has to lose regularity at the self-intersection point due to Gancedo--Strain \cite{Gancedo-Strain:absence-splash-muskat-SQG}). They conducted experiments for different values of $\gamma \in (0,1]$ and reported blowup for all of them, see also \cite{Mancho:numerical-studies-self-similar-alpha-patch}. Moreover, they numerically checked that the singularity was of asymptotically self-similar type. After them, Scott--Dritschel performed a very careful study using ellipses as initial condition: On the one hand, in \cite{Scott-Dritschel:self-similar-sqg}, after taking an ellipse with aspect ratio $a = 0.16$, a self-similar cascade of filamentation is reported, where the filament width collapses to zero in a finite time. On the other hand, in \cite{Scott-Dritschel:self-similar-sqg-patch}, a more detailed analysis is done (for initial aspect ratios $a = 0.094, 0.132, 0.180$), and the boundary is split into three scales: an inner (faster) scale that develops a self-similar corner, a middle scale that develops a self-similar spiral, and the outermost region. We note that there is some resemblance between the middle scale of this analysis and the solutions constructed in this paper which are also self-similar spirals.

Finally, we also mention the formation of singularities in other geometries (in the presence of a boundary) by Kiselev--Ryzhik--Yao--Zlato\v{s} \cite{Kiselev-Ryzhik-Yao-Zlatos:singularity-alpha-patch-boundary} (for the case $0 < \gamma < \frac{1}{12}$) and Gancedo--Patel \cite{Gancedo-Patel:local-existence-blowup-gsqg} (for the case $0 < \gamma < \frac13$). We refer also the reader to \cite{Ao-Davila-delPino-Musso-Wei:travelling-rotating-solutions-gsqg,Castro-Cordoba-GomezSerrano:existence-regularity-vstates-gsqg,Castro-Cordoba-GomezSerrano:global-smooth-solutions-sqg,Cordoba-GomezSerrano-Ionescu:global-generalized-sqg-patch,delaHoz-Hassainia-Hmidi:doubly-connected-vstates-gsqg,Garcia:Karman-vortex-street,Garcia:vortex-patch-choreography,GomezSerrano:stationary-patches,Gravejat-Smets:travelling-waves-smooth-sqg,Hassainia-Hmidi:v-states-generalized-sqg,Hassainia-Hmidi-Masmoudi:kam-gsqg,Hmidi-Mateu:existence-corotating-counter-rotating,Hunter-Shu-Zhang:global-model-front-sqg,Hunter-Shu-Zhang:global-gsqg,Nahmod-Pavlovic-Staffilani-Totz:global-invariant-measures-gsqg,Renault:relative-equlibria-holes-sqg} for other results concerning global solutions for the gSQG equations.

Regarding negative results of self-similar solutions (i.e. whenever the only possible solution is the trivial one) Cannone--Xue \cite{Cannone-Xue:self-similar-solutions-sqg} and Chae \cite{Chae:nonexistence-selfsimilar-euler-3d-sqg} proved that there are no locally self-similar solutions under suitable technical assumptions, one of them being integrable. Although our solutions decay at infinity, they do not decay sufficiently fast to be in $L^q$ with $q>2$ as needed in \cite{Cannone-Xue:self-similar-solutions-sqg}.

\subsection{Main ideas of the proof}

Let us give the main ideas of the proof. By inserting the ansatz \eqref{theta-self-similar-intro} in the gSQG equations, we find the equivalent equation for $\hat{\theta}$:
\begin{equation}\label{eq-self-intro}
\left\{\frac{1+\alpha-\gamma}{1+\alpha}-\frac{2}{1+\alpha}\right\}\hat{\theta}(z)-\nabla\cdot\left\{\left(\hat{u}(z)-\frac{z}{1+\alpha}\right)\hat{\theta}(z)\right\}=0,
\end{equation}
where the vector field
\begin{equation}\label{pseudo-velocity-intro}
q(z):=\hat{u}(z)-\frac{z}{1+\alpha}=\nabla^\perp \hat{\psi}(z)-\frac{z}{1+\alpha},
\end{equation}
is called the pseudo-velocity, and its integral curves are the pseudo-streamlines. Hence, in order to find nontrivial self-similar solutions one needs to solve \eqref{self-similar-0-intro}.

Motivated by numerous works in the literature about the existence of nontrivial steady vortex patches in different active scalar equations \cite{Burbea:motions-vortex-patches, Davila-DelPino-Musso-Wei:travelling-helices-vortex-filament,delaHoz-Hassainia-Hmidi:doubly-connected-vstates-gsqg,Garcia:Karman-vortex-street,Garcia:vortex-patch-choreography, Garcia-Hmidi-Mateu:time-periodic-doubly-3d-qg,Garcia-Hmidi-Mateu:time-periodic-3d-qg,Garcia-Hmidi-Soler:non-uniform-vstates-euler, GomezSerrano:stationary-patches,Hassainia-Hmidi:v-states-generalized-sqg,Hmidi-Mateu:existence-corotating-counter-rotating, Hmidi-Mateu-Verdera:rotating-vortex-patch, Renault:relative-equlibria-holes-sqg}, we will solve \eqref{eq-self-intro} via perturbative arguments by means of the infinite dimensional Implicit Function theorem. That means that we will look for nontrivial self-similar solutions which are {\it close} to the trivial stationary solution \eqref{trivial-radial}. However, perturbative arguments can not be directly applied to \eqref{eq-self-intro} since the linearized operator at a trivial radial solution \eqref{trivial-radial} is not Fredholm. That comes from the fact that the solution is singular at $0$. However the kind of singularity at the linear level of the terms with the radial and angular derivatives is different, which complicates the choice of the appropriate weighted space. That problem also appeared when perturbing radial functions in the Euler and gSQG equations \cite{Castro-Cordoba-GomezSerrano:uniformly-rotating-smooth-euler, Castro-Cordoba-GomezSerrano:analytic-vstates-ellipses, Garcia-Hmidi-Soler:non-uniform-vstates-euler}. The authors overcame such issue by either performing an appropriate change of variables by means of the level sets or by working in a suitable class of solutions.

In this work, we will tackle the lack of Fredholmness of the linarized operator when working directly with \eqref{eq-self-intro} in the physical variables, by using the {\it adapted coordiantes} introduced by Elling \cite{Elling:algebraic-spiral-solutions-2d-euler}. Hence, we will transform the polar coordiantes $(r,\vartheta)$ into some new ones $(\beta,\phi)$. Such transformation is a nonlinear change of coordinates which is implicit in the sense that it depends on the solution itself. Here $\beta\in(0,\infty)$ is the angle along a particular pseudo-streamline. Moreover $\beta=0$ will correspond to the spatial infinity while $\beta=+\infty$ approaches the spiral center. On the other hand, $\phi$ is the polar angle at $r\rightarrow +\infty$.  In order to define $(\beta,\phi)$ we assume that $\partial_\beta$ is tangential to its pseudo-streamlines. Define the angle $\phi$ as follows
$$
\vartheta=\beta+\phi, \quad \phi\in\T.
$$
We will fix $\beta$ such that the pseudo-velocity \eqref{pseudo-velocity-intro} has vanishing $\phi$-component. Defining
$$
\Theta(\beta,\phi):=\hat{\theta}(z(\beta,\phi)), \quad \Psi(\beta,\phi):=\hat{\psi}(z(\beta,\phi)),
$$
we achieve that 
\begin{align*}
\left(\nabla^\perp \hat{\psi}-\frac{z}{1+\alpha}\right)\cdot \nabla_z\hat{\theta}=&\left((\beta_{z_2}\phi_{z_1}-\beta_{z_1}\phi_{z_2})\Psi_\phi-\frac{1}{1+\alpha}(z_1\beta_{z_1}+z_2\beta_{z_2}),\right.\\
&\left.(\beta_{z_1}\phi_{z_2}-\beta_{z_2}\phi_{z_1})\Psi_\beta-\frac{1}{1+\alpha}(z_1\phi_{z_1}+z_2\phi_{z_2}) \right)\cdot \nabla_{\beta,\phi}\Theta.
\end{align*}
Hence, we fix $\beta$ such that $\phi$-component of the pseudo-velocity is 0, that is:
\begin{equation}\label{beta-1-intro}
(\beta_{z_1}\phi_{z_2}-\beta_{z_2}\phi_{z_1})\Psi_\beta-\frac{1}{1+\alpha}(z_1\phi_{z_1}+z_2\phi_{z_2})=0,
\end{equation}
which, after some computations agrees with
$$
r=(-(1+\alpha)\Psi_\beta)^\frac12.
$$
\begin{center}
\begin{figure}[htbp]
\centering
\def\svgwidth{0.7\textwidth}
\begingroup%
  \makeatletter%
  \providecommand\color[2][]{%
    \errmessage{(Inkscape) Color is used for the text in Inkscape, but the package 'color.sty' is not loaded}%
    \renewcommand\color[2][]{}%
  }%
  \providecommand\transparent[1]{%
    \errmessage{(Inkscape) Transparency is used (non-zero) for the text in Inkscape, but the package 'transparent.sty' is not loaded}%
    \renewcommand\transparent[1]{}%
  }%
  \providecommand\rotatebox[2]{#2}%
  \newcommand*\fsize{\dimexpr\f@size pt\relax}%
  \newcommand*\lineheight[1]{\fontsize{\fsize}{#1\fsize}\selectfont}%
  \ifx\svgwidth\undefined%
    \setlength{\unitlength}{314.53202044bp}%
    \ifx\svgscale\undefined%
      \relax%
    \else%
      \setlength{\unitlength}{\unitlength * \real{\svgscale}}%
    \fi%
  \else%
    \setlength{\unitlength}{\svgwidth}%
  \fi%
  \global\let\svgwidth\undefined%
  \global\let\svgscale\undefined%
  \makeatother%
  \begin{picture}(1,0.45741636)%
    \lineheight{1}%
    \setlength\tabcolsep{0pt}%
    \put(0,0){\includegraphics[width=\unitlength,page=1]{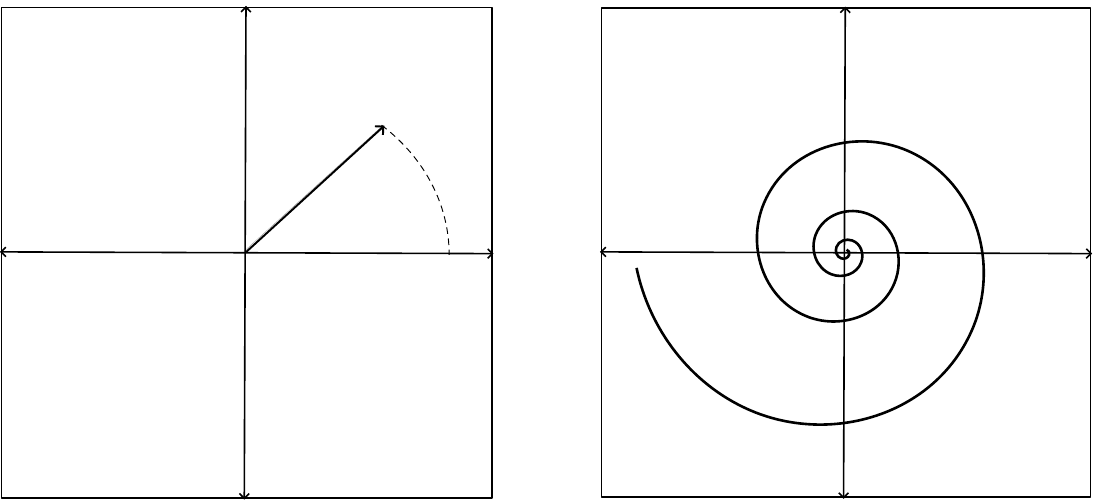}}%
    \put(0.24596154,0.30914828){\color[rgb]{0.01568627,0.01568627,0.01568627}\transparent{0.97647101}\makebox(0,0)[lt]{\lineheight{1.25}\smash{\begin{tabular}[t]{l}$r$\end{tabular}}}}%
    \put(0.01449349,0.42436342){\color[rgb]{0.01568627,0.01568627,0.01568627}\transparent{0.97647101}\makebox(0,0)[lt]{\lineheight{1.25}\smash{\begin{tabular}[t]{l}$(z_1,z_2)$\end{tabular}}}}%
    \put(0.56450747,0.42214408){\color[rgb]{0.01568627,0.01568627,0.01568627}\transparent{0.97647101}\makebox(0,0)[lt]{\lineheight{1.25}\smash{\begin{tabular}[t]{l}$(z_1,z_2)$\end{tabular}}}}%
    \put(0.39900823,0.30008661){\color[rgb]{0.01568627,0.01568627,0.01568627}\transparent{0.97647101}\makebox(0,0)[lt]{\lineheight{1.25}\smash{\begin{tabular}[t]{l}$\vartheta$\end{tabular}}}}%
  \end{picture}%
\endgroup%

\caption{Comparison between the polar coordinates in the $(z_1,z_2)$ plane with the $(\beta,\phi)$ adapted coordinates.}
\label{Fig-2}
\end{figure}
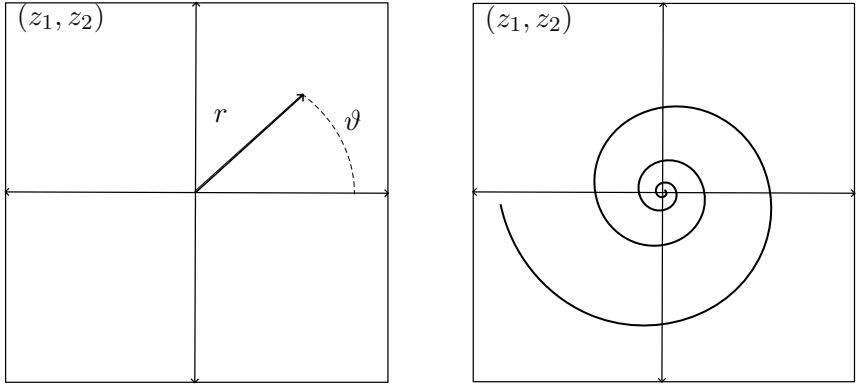
\end{center}

We refer to Figure \ref{Fig-2} to illustrate the  adapted coordinates. {\it Miraculously}, when assuming \eqref{beta-1-intro} we have that 
\begin{align}\label{vorticity-eq-intro}
\left(\nabla^\perp\hat{\psi}-\frac{z}{1+\alpha}\right)&\cdot\nabla_z\hat{\theta}=\left\{(\beta_{z_2}\phi_{z_1}-\beta_{z_1}\phi_{z_2})\Psi_\phi-\frac{1}{1+\alpha}(z_1\beta_{z_1}+z_2\beta_{z_2})\right\}\Theta_\beta,
\end{align}
where only $\partial_\beta$ survives there. That makes that the vorticity \eqref{eq-self-intro} written in the adapted coordinates depend only on the $\beta$-derivative of $\Theta$. Thanks to that, a trivial solution to \eqref{eq-self-intro} is found as
\begin{equation}\label{trivial-sol-1-intro}
\Theta(\beta,\phi)=(\Psi_\varphi)^{-\frac{1+\alpha-\gamma}{2}}\Omega(\phi)=(\Psi_\varphi)^{-\frac{1/\mu}{2}}\Omega(\phi),
\end{equation}
for any $\Omega$. However, $\Theta$ depends on the stream function $\Psi$ and it remains to check the elliptic relation \eqref{psi-theta} between these two functions. Written in the adapted coordinates and using
$$
\Psi(\beta,\phi)=-C_0^{\frac{2}{1+\alpha}}(\alpha-1)^{\frac{1-\alpha}{1+\alpha}}\beta^\frac{\alpha-1}{1+\alpha}(1+f(\beta,\phi)), \quad \Omega(\phi)=(1+\alpha)^{-\frac{1}{2\mu}}\tilde{\Omega}(\phi),
$$
the relation amounts to
$$
\tilde{F}(f,\tilde{\Omega})(\beta,\phi)=0, \quad (\beta,\phi)\in[0,\infty)\times\T,
$$
where
\begin{align}\label{Ftilde-intro}
\nonumber&\tilde{F}(f,\tilde{\Omega})(\beta,\phi)=1+f(\beta,\phi)\\
\nonumber&-(-1)^\frac{-1}{2\mu}\frac{C_\gamma (\alpha-1)^\frac{\alpha-1}{2}}{4\pi C_0(1+\alpha)^{\frac{\alpha-1}{2}}}\beta^{-\frac{\alpha-1}{1+\alpha}}\\
&\times \int_0^{2\pi}\int_0^\infty\frac{b^{-\frac{\gamma}{1+\alpha}}b^{\frac{-2}{1+\alpha}}\left\{-\frac{\alpha-1}{1+\alpha}+\partial_{\overline{\varphi}}f(b,\Phi)\right\}^{-\frac{1}{2\mu}}\left\{\frac{2(\alpha-1)}{(1+\alpha)^2}+\partial_{\overline{\beta\varphi}}f(b,\Phi)\right\}\tilde{\Omega}(\Phi)dbd\Phi}{D(f)(\beta,\phi,b,\Phi)^\frac{\gamma}{2}},
\end{align}
and
\begin{align}\label{denominator-0-intro}
D(f)(\beta,\phi,b,\Phi):=&\tilde{\Psi}_\beta+\tilde{\Psi}_b-2\tilde{\Psi}_\beta^\frac12\tilde{\Psi}_b^\frac12\cos(\phi+\beta-\Phi-b)\\
=&\beta^{\frac{-2}{1+\alpha}}\left\{\frac{\alpha-1}{1+\alpha}+\partial_{\overline{\beta}}f(\beta,\phi)\right\}+b^{\frac{-2}{1+\alpha}}\left\{\frac{\alpha-1}{1+\alpha}+\partial_{\overline{b}}f(b,\Phi)\right\}\nonumber\\
&-2\beta^{\frac{-1}{1+\alpha}}b^{\frac{-1}{1+\alpha}}\left\{\frac{\alpha-1}{1+\alpha}+\partial_{\overline{\beta}}f(\beta,\phi)\right\}^\frac12\left\{\frac{\alpha-1}{1+\alpha}+\partial_{\overline{b}}f(b,\Phi)\right\}^\frac12\cos(\phi+\beta-\Phi-b)\nonumber\\
=&\left(\beta^{\frac{-1}{1+\alpha}}\left\{\frac{\alpha-1}{1+\alpha}+\partial_{\overline{\beta}}f(\beta,\phi)\right\}^\frac12-b^{\frac{-1}{1+\alpha}}\left\{\frac{\alpha-1}{1+\alpha}+\partial_{\overline{b}}f(b,\Phi)\right\}^\frac12\right)^2\nonumber\\
&+2\beta^{\frac{-1}{1+\alpha}}b^{\frac{-1}{1+\alpha}}\left\{\frac{\alpha-1}{1+\alpha}+\partial_{\overline{\beta}}f(\beta,\phi)\right\}^\frac12\left\{\frac{\alpha-1}{1+\alpha}+\partial_{\overline{b}}f(b,\Phi)\right\}^\frac12\nonumber\\
&\times (1-\cos(\phi+\beta-\Phi-b)).\nonumber
\end{align}
There, $\partial_\varphi:=\partial_\phi-\partial_\beta$ and notice that the previous expression does not depend on $\partial_\phi$ which is crucial for our analysis. Indeed, the choice of the adapted coordinates is the key point of this work due to two facts: first we can solve the vorticity equation with \eqref{trivial-sol-1-intro}, and second the nonlinear equation only depends on one derivative as opposite to \eqref{eq-self-intro}.

Then, the proof of the main theorem relies on the application of the infinite dimensional Implicit Function theorem to $\tilde{F}$ in some appropriate spaces. The choice of the function spaces is a key point in this work since we should add some appropriate decay for $f$ as $\beta$ approaches $0$ and $\infty$. Let us point out that the funcional $\tilde{F}$ differs to the formulation used in \cite{Bressan-Murray:self-similar-euler, Elling:algebraic-spiral-solutions-2d-euler, Elling:self-similar-euler-mixed-sign-vorticity} when working with the Euler equations since there the relation between the vorticity and the streamfunction is local, as opposite to the gSQG equations. 

In order to apply the Implicit Function theorem to $\tilde{F}$, we should check that the linear operator around the trivial solution is an isomorphism. In particular, we shall check that its kernel is trivial. To that aim, we apply the Mellin transform to the kernel equation as it was motivated by some works on coagulation and fragmentation models (see \cite{Breschi-Fontelos:self-similar-fragmentation,Escobedo-Velazquez:fundamental-solution-coagulation}) to obtain a recurrence equation with singular coefficients. Later, we relate such recurrence equation with the moment problem concluding that the function in the kernel is singular and indeed does not belong to our function spaces. Finally, after showing that the Implicit Function theorem can be applied to the formulation with the adapted coordiantes, we shall prove the invertibility of the adapted coordinates when we are {\it close} to the trivial solution to come back to the physical variables.

\subsection{Organization of the paper}
The paper is organized in the following way: Section \ref{sec-selfsimilar} sets the problem up and writes the equations in the adapted coordinates. Section \ref{sec-functional} defines the functional spaces in which we will work on and performs estimates at the nonlinear level. On the contrary, in Section \ref{sec-linearized} we perform regularity estimates of the linear operator. Section \ref{sec-spectral} is devoted to the spectral study and in particular the invertibility of the linear operator. Finally Sections \ref{sec-mfold} and \ref{sec-recovering} comprise the reduction to $m$-fold symmetry classes and the translation of Theorem \ref{theo-1} into standard coordinates. We also include two Appendices: Appendix \ref{ap-radial} concerns basic results on radial solutions of the fractional Laplacian, whereas Appendix \ref{ap-special} deals with basic properties of some special functions needed throughout the text.

\section{Self-similar solutions}\label{sec-selfsimilar}
In this section, we aim to provide the definition of a self-similar solution to the gSQG equations together with the formulation that characterizes them. Note that solutions to the gSQG equations \eqref{gSQG} are invariant under the following scaling: for $\alpha>-1$,
$$
\theta(t,x)\mapsto \theta_\lambda(t,x)=\lambda^{1+\alpha-\gamma}\theta(\lambda^{1+\alpha}t,\lambda x); \quad u(t,x)\mapsto u_\lambda (t,x)=\lambda^\alpha u(\lambda^{1+\alpha}t,\lambda x), \quad \forall\lambda>0.
$$
We will say that $\theta$ solution of \eqref{gSQG} is self-similar if $\theta=\theta_\lambda$, for any $\lambda>0$. That is,
\begin{equation}\label{self-similar-0}
\theta(t,x)=\lambda^{1+\alpha-\gamma}\theta(\lambda^{1+\alpha}t,\lambda x), \quad \forall\lambda>0.
\end{equation}
Hence, in order to have that the trivial solution \eqref{trivial-radial} is self-similar in the above way, we have the following condition for the singularity parameter $\mu$ depending on $\alpha$ and $\gamma$:
\begin{equation}\label{mu}
\frac{1}{\mu}=1+\alpha-\gamma.
\end{equation}
By \eqref{mu-condition} and the relation \eqref{mu} we find that 
\begin{equation}\label{alpha-condition}
1<\alpha<1+\gamma.
\end{equation}
Now, by taking $\lambda=1/t^{\frac{1}{1+\alpha}}$ in \eqref{self-similar-0}, our aim is to look for self-similar solutions of the form:
\begin{equation}\label{theta-self-similar}
\theta(t,x)=\frac{1}{t^{\frac{1+\alpha-\gamma}{1+\alpha}}}\hat{\theta}\left(\frac{x}{t^{\frac{1}{1+\alpha}}}\right),
\end{equation}
for some self-similar profile $\hat{\theta}$. 

In the following proposition, we show the expression of $u$ with the self-similar ansatz.
\begin{pro}
If $\theta$ takes the form \eqref{theta-self-similar}, then
$$
u(t,x)=\frac{1}{t^{\frac{\alpha}{1+\alpha}}}\hat{u}\left(\frac{x}{t^{\frac{1}{1+\alpha}}}\right),
$$
where $\hat{u}=-\nabla^\perp(-\Delta)^{-1+\frac{\gamma}{2}}\hat{\theta}$. Moreover, \eqref{gSQG} is equivalent to
\begin{equation}\label{eq-self-sim}
\frac{1+\alpha-\gamma}{1+\alpha}\hat{\theta}(z)+\frac{1}{1+\alpha}z\cdot \nabla \hat{\theta}(z)-\hat{u}\cdot\nabla\hat{\theta}(z)=0,
\end{equation}
where $z=\frac{x}{t^{\frac{1}{1+\alpha}}}$. In a similar way, let us define $\hat{\psi}$ as 
$$
\hat{\psi}(z)=-(-\Delta)^{-1+\frac{\gamma}{2}}\hat{\theta}(z).
$$
\end{pro}
\begin{proof}
First, let us compute $u$ assuming \eqref{theta-self-similar}:
\begin{align*}
u(t,x)=&-\nabla^\perp(-\Delta)^{-1+\frac{\gamma}{2}}\theta(t,x)\\
=&-\frac{C_\gamma \gamma}{4\pi}\int_{\R^2}\frac{(x-y)^\perp}{|x-y|^{\gamma+2}}\theta(t,y)dA(y)\\
=&\frac{C_\gamma \gamma}{4\pi}\frac{1}{t^\frac{1+\alpha-\gamma}{1+\alpha}}\int_{\R^2}\frac{(x-y)^\perp}{|x-y|^{\gamma+2}}\hat{\theta}\left(\frac{y}{t^{\frac{1}{1+\alpha}}}\right)dA(y)\\
=&\frac{C_\gamma \gamma}{4\pi}\frac{1}{t^\frac{1+\alpha-\gamma}{1+\alpha}}\frac{t^{\frac{1}{1+\alpha}}}{t^{\frac{\gamma+2}{1+\alpha}}}\int_{\R^2}\frac{\left(\frac{x}{t^{\frac{1}{1+\alpha}}}-\frac{y}{t^{\frac{1}{1+\alpha}}}\right)^\perp}{\left|\frac{x}{t^{\frac{1}{1+\alpha}}}-\frac{y}{t^{\frac{1}{1+\alpha}}}\right|^{\gamma+2}}\hat{\theta}\left(\frac{y}{t^{\frac{1}{1+\alpha}}}\right)dA(y)\\
=&\frac{C_\gamma \gamma}{4\pi}\frac{1}{t^\frac{\alpha}{1+\alpha}}\int_{\R^2}\frac{\left(\frac{x}{t^{\frac{1}{1+\alpha}}}-y\right)^\perp}{\left|\frac{x}{t^{\frac{1}{1+\alpha}}}-y\right|^{\gamma+2}}\hat{\theta}(y)dA(y)\\
=&\frac{1}{t^{\frac{\alpha}{1+\alpha}}}\hat{u}\left(\frac{x}{t^{\frac{1}{1+\alpha}}}\right),
\end{align*}
where $\hat{u}=-\nabla^\perp(-\Delta)^{-1+\frac{\gamma}{2}}\hat{\theta}$.

Let us now check \eqref{eq-self-sim}. Note that
\begin{align*}
\partial_t \theta(t,x)=&-{\frac{1+\alpha-\gamma}{1+\alpha}}\frac{1}{t^{\frac{1+\alpha-\gamma}{1+\alpha}+1}}\hat{\theta}\left(\frac{x}{t^{\frac{1}{1+\alpha}}}\right)-\frac{1}{1+\alpha}\frac{1}{t^{\frac{1+\alpha-\gamma}{1+\alpha}+1}}\frac{x}{t^{\frac{1}{1+\alpha}}}\cdot (\nabla \hat{\theta})\left(\frac{x}{t^{\frac{1}{1+\alpha}}}\right)\\
=&-\frac{1}{t^{\frac{1+\alpha-\gamma}{1+\alpha}+1}}\left\{\frac{1+\alpha-\gamma}{1+\alpha}\hat{\theta}(z)+\frac{1}{1+\alpha}z\cdot\nabla\hat{\theta}(z)\right\},\\
\nabla\theta(t,x)=&\frac{1}{t^{\frac{1+\alpha-\gamma}{1+\alpha}}}\frac{1}{t^\frac{1}{1+\alpha}}(\nabla\hat{\theta})\left(\frac{x}{t^{\frac{1}{1+\alpha}}}\right)\\
=&\frac{1}{t^{\frac{1+\alpha-\gamma}{1+\alpha}}}\frac{1}{t^\frac{1}{1+\alpha}}\nabla\hat{\theta}(z).
\end{align*}
Then,
\begin{align*}
\partial_t \theta(t,x)+(u(t,x)\cdot \nabla)\theta(t,x)=-\frac{1}{t^{\frac{1+\alpha-\gamma}{1+\alpha}+1}}\left\{\frac{1+\alpha-\gamma}{1+\alpha}\hat{\theta}(z)+\frac{1}{1+\alpha}z\cdot \nabla\hat{\theta}(z)-\hat{u}(z)\cdot \nabla \hat{\theta}(z)\right\},
\end{align*}
yielding to the equivalent equation \eqref{eq-self-sim}:
$$
\frac{1+\alpha-\gamma}{1+\alpha}\hat{\theta}(z)+\frac{1}{1+\alpha}z\cdot \nabla\hat{\theta}(z)-\hat{u}(z)\cdot \nabla \hat{\theta}(z)=0.
$$
Note that since $\nabla\cdot \hat{u}=0$, the previous equation can be written as
$$
\left\{\frac{1+\alpha-\gamma}{1+\alpha}-\frac{2}{1+\alpha}\right\}\hat{\theta}(z)-\nabla\cdot\left\{\left(\hat{u}(z)-\frac{z}{1+\alpha}\right)\hat{\theta}(z)\right\}=0.
$$
\end{proof}
The vector field 
\begin{equation}\label{pseudo-velocity}
q(z):=\hat{u}(z)-\frac{z}{1+\alpha}=\nabla^\perp \hat{\psi}(z)-\frac{z}{1+\alpha},
\end{equation}
is called the pseudo-velocity, and its integral curves are the pseudo-streamlines.

\subsection{Adapted coordinates}\label{sec-adaptive-coord}
Our aim is to perturb the stationary radial solution \eqref{trivial-radial} to find non trivial self similar solutions with a locally integrable singularity at the origin. However, perturbation techniques, for instance the Implicit Function theorem or Crandall-Rabinowitz theorem, of radial solutions to gSQG or Euler equations seem to be difficult due to the very special structure of the equation. Roughly speaking, the nonlinear operator and the linear one behave very differently in terms of the singularity at the origin. That makes difficult the choice of the appropriate function spaces and leads to main difficulties in the Fredholm structure of the linearized operator. A similar issue appears when dealing with uniformly rotating solutions, which could be solved in \cite{Castro-Cordoba-GomezSerrano:analytic-vstates-ellipses,Castro-Cordoba-GomezSerrano:uniformly-rotating-smooth-euler, Garcia-Hmidi-Soler:non-uniform-vstates-euler} by restricting the class of solutions or doing an appropriate change of variables by using the level sets of the vorticity.

Although here we are not dealing with rotating solutions, our equation \eqref{eq-self-sim} has a very similar structure close to the radial solution \eqref{trivial-radial}. Moreover, we will have to deal with extra difficulties due to the singularity of \eqref{trivial-radial} at $0$ and its behavior at $+\infty$. More precisely, we are now working in an unbounded domain and we need to use some appropriate weights at $0$ and $+\infty$. In order to overcome the problematic with the Fredholm structure of the linearized operator, working with some adapted coordinates instead of the usual polar coordinates seems to be more convenient in this case.

Let $(r,\vartheta)$ be the polar coordinates with $z=re^{i\vartheta}$. Here, we would like to find some adapted coordinates $(\beta,\phi)$ such that the pseudo-velocity \eqref{pseudo-velocity} has vanishing $\phi$-component.

First, define $\phi$ as
\begin{equation}
\vartheta=\beta+\phi, \quad\phi\in\T.
\end{equation}
In the following proposition we find the nonlinear expression of $\beta$ depending on $(r,\vartheta)$ in order to have vanishing $\phi$-component of the pseudo-velocity. We shall see that such nonlinear expression is indeed an implicit function that depends on the solution itself. Denote
$$
\Theta(\beta,\phi):=\hat{\theta}(z(\beta,\phi)), \quad \Psi(\beta,\phi):=\hat{\psi}(z(\beta,\phi)).
$$

\begin{pro}\label{prop-coordinates}
If 
$
r=(-(1+\alpha)\Psi_\beta)^\frac12$ and  $\vartheta=\beta+\phi
$, then the $\phi$-component of the pseudo-velocity, that is, $\nabla^\perp\hat{\psi}-\frac{z}{1+\alpha}$, vanishes.
\end{pro}
\begin{proof}
Let us compute the term $\left(\nabla^\perp\hat{\psi}-\frac{z}{1+\alpha}\right)\cdot\nabla_z\hat{\theta}$  in \eqref{eq-self-sim} in terms of $(\beta,\phi)$:

\begin{align*}
\left(\nabla^\perp\hat{\psi}-\frac{z}{1+\alpha}\right)&\cdot\nabla_z\hat{\theta}=\left(-\hat{\psi}_{z_2}-\frac{z_1}{1+\alpha}\right)\hat{\theta}_{z_1}+\left(\hat{\psi}_{z_1}-\frac{z_2}{1+\alpha}\right)\hat{\theta}_{z_2}\\
=&\left\{-\Psi_\beta\beta_{z_2}-\Psi_\phi\phi_{z_2-\frac{z_1}{1+\alpha}}\right\}\left\{\Theta_\beta\beta_{z_1}+\Theta_\phi\phi_{z_1}\right\}\\
&+\left\{\Psi_\beta\beta_{z_1}+\Psi_\phi\phi_{z_1-\frac{z_2}{1+\alpha}}\right\}\left\{\Theta_\beta\beta_{z_2}+\Theta_\phi\phi_{z_2}\right\}\\
=&\left\{(\beta_{z_1}\phi_{z_2}-\beta_{z_2}\phi_{z_1})\Psi_\beta-\frac{1}{1+\alpha}(z_1\phi_{z_1}+z_2\phi_{z_2})\right\}\Theta_\phi\\
&+\left\{(\beta_{z_2}\phi_{z_1}-\beta_{z_1}\phi_{z_2})\Psi_\phi-\frac{1}{1+\alpha}(z_1\beta_{z_1}+z_2\beta_{z_2})\right\}\Theta_\beta\\
=&\left((\beta_{z_2}\phi_{z_1}-\beta_{z_1}\phi_{z_2})\Psi_\phi-\frac{1}{1+\alpha}(z_1\beta_{z_1}+z_2\beta_{z_2}),\right.\\
&\left.(\beta_{z_1}\phi_{z_2}-\beta_{z_2}\phi_{z_1})\Psi_\beta-\frac{1}{1+\alpha}(z_1\phi_{z_1}+z_2\phi_{z_2}) \right)\cdot \nabla_{\beta,\phi}\Theta.
\end{align*}
We want to fix $\beta$ such that the $\phi$-component of the pseudo-velocity is 0, that is, $\beta$ will be defined through the following nonlinear equation
\begin{equation}\label{beta-1}
(\beta_{z_1}\phi_{z_2}-\beta_{z_2}\phi_{z_1})\Psi_\beta-\frac{1}{1+\alpha}(z_1\phi_{z_1}+z_2\phi_{z_2})=0.
\end{equation}
Note that assuming \eqref{beta-1}, we have that 
\begin{align*}
\left(\nabla^\perp\hat{\psi}-\frac{z}{1+\alpha}\right)&\cdot\nabla_z\hat{\theta}=\left\{(\beta_{z_2}\phi_{z_1}-\beta_{z_1}\phi_{z_2})\Psi_\phi-\frac{1}{1+\alpha}(z_1\beta_{z_1}+z_2\beta_{z_2})\right\}\Theta_\beta,
\end{align*}
where only $\partial_\beta$ survives there, which will be the crucial in order to find some trivial solution to the vorticity equation.

We can also compute the Jacobian of the transformation as follows
\begin{equation}\label{jacobian}
\left[\begin{array}{cc}
z_{1,\beta}&z_{1,\phi}\\ z_{2,\beta}&z_{2,\phi}
\end{array}
\right]=
\left[\begin{array}{cc}
\beta_{z_1}&\beta_{z_2}\\ \phi_{z_1}&\phi_{z_2}
\end{array}
\right]^{-1}=
\frac{1}{\beta_{z_1}\phi_{z_2}-\beta_{z_2}\phi_{z_1}}\left[\begin{array}{cc}
\phi_{z_2}&-\beta_{z_2}\\ -\phi_{z_1}&\beta_{z_1}
\end{array}
\right].
\end{equation}
and using \eqref{beta-1}--\eqref{jacobian}, we find
$$
\Psi_\beta+\frac{1}{1+\alpha}z\times z_\beta=\Psi_\beta+\frac{1}{1+\alpha}(z_1z_{2,\beta}-z_2z_{1,\beta})=0.
$$
By virtue of the variables $\vartheta(\beta,\phi)=\beta+\phi$ and $z=z(r,\vartheta)$, we find that
$$
z\times z_\beta=|z|^2,
$$
which amounts to
$$
\Psi_\beta=-\frac{1}{(1+\alpha)}|z|^2.
$$
Hence, the map $(\beta,\phi)\mapsto(r,\vartheta)$ takes the form
$$
r=(-(1+\alpha)\Psi_\beta)^\frac12,\quad \vartheta=\beta+\phi.
$$
That concludes the proof of this proposition.
\end{proof}

Let us remark again that the change of variables of the previous proposition is a nonlinear change of variables that depends on the solution, indeed, it depends on the stream function. Later, in Section \ref{sec-invertibility-change} we will check that we can come back to the physical variables and thus the change of variables is invertible. Moreover, from now on define
\begin{equation}\label{partialvarphi}
\partial_\varphi=\partial_\phi-\partial_\beta.
\end{equation}
In the following, let us compute the jacobian of the transformation.

\begin{pro}\label{prop-jacobian}
The jacobian of the transformation {$(\beta,\phi)\mapsto (z_1,z_2)$} is given by
$$
|J|=\frac{1+\alpha}{2}\Psi_{\beta\varphi},
$$
where $\partial_\varphi=\partial_\phi-\partial_\beta$.
\end{pro}
\begin{proof}

Note that the matrix of the transformation is given by
\begin{equation*}
J=\left[\begin{array}{cc}
z_{1,\beta}&z_{1,\phi}\\ z_{2,\beta}&z_{2,\phi}
\end{array}
\right].
\end{equation*}
From Proposition \ref{prop-coordinates} we have that
$$
r=(-(1+\alpha)\Psi_\beta)^\frac12,\quad \vartheta=\beta+\phi,
$$
and using such expression, we can compute the following
\begin{align*}
\partial_\phi z_1=&\partial_r z_1\partial_\phi r+\partial_\vartheta z_1\partial_\phi\theta\\
=&-\frac12\cos(\vartheta)(-(1+\alpha)\Psi_\beta)^{-\frac12}(1+\alpha)\partial^2_{\phi\beta}\Psi-r\sin(\vartheta),\\
\partial_\beta z_1=&-\frac12\cos(\vartheta)(-(1+\alpha)\Psi_\beta)^{-\frac12}(1+\alpha)\partial^2_{\beta\beta}\Psi-r\sin(\vartheta),\\
\partial_\phi z_2=&-\frac12\sin(\vartheta)(-(1+\alpha)\Psi_\beta)^{-\frac12}(1+\alpha)\partial^2_{\phi\beta}\Psi+r\cos(\vartheta),\\
\partial_\beta z_2=&-\frac12\sin(\vartheta)(-(1+\alpha)\Psi_\beta)^{-\frac12}(1+\alpha)\partial^2_{\beta\beta}\Psi+r\cos(\vartheta).
\end{align*}
Then, after some straightforward computations, we finally find that
\begin{align*}
|J|=&z_{1,\beta}z_{2,\phi}-z_{1,\phi}z_{2,\beta}=\frac{1+\alpha}{2}\Psi_{\beta\varphi}.
\end{align*}
\end{proof}

\subsection{The equation in the new coordinates}
In Section \ref{sec-adaptive-coord} we have defined the adapted coordinates that will be used through this work. Recall that we have chosen such coordinates in order to have that the $\phi$ component of the pseudo-velocity is vanishing. In this way, we find
$$
\left(\nabla^\perp\hat{\psi}-\frac{z}{1+\alpha}\right)\cdot\nabla\hat{\theta}=\left\{(\beta_{z_2}\phi_{z_1}-\beta_{z_1}\phi_{z_2})\Psi_\phi-\frac{1}{1+\alpha}(z_1\beta_{z_1}+z_2\beta_{z_2})\right\}\Theta_\beta,
$$
and then the vorticity equation \eqref{eq-self-sim} agrees with
\begin{equation}\label{eq-self-sim-2}
\frac{1+\alpha-\gamma}{1+\alpha}\Theta(\beta,\phi)=\left\{(\beta_{z_2}\phi_{z_1}-\beta_{z_1}\phi_{z_2})\Psi_\phi-\frac{1}{1+\alpha}(z_1\beta_{z_1}+z_2\beta_{z_2})\right\}\Theta_\beta.
\end{equation}
The very special structure of \eqref{eq-self-sim-2} involving only a $\beta$ derivative of $\Theta$ gives us a trivial solution. That will be discussed in the following result.

\begin{pro}\label{prop-eq-theta}
Equation \eqref{eq-self-sim-2} agrees with
\begin{equation}\label{eq-self-sim-3}
(1+\alpha-\gamma)\Theta(\beta,\phi)\Psi_{\beta\varphi}(\beta,\phi)+2\Psi_{\varphi}(\beta,\phi)\Theta_\beta(\beta,\phi)=0,
\end{equation}
which has the following trivial solution:
\begin{equation}\label{trivial-sol-1}
\Theta(\beta,\phi)=(\Psi_\varphi)^{-\frac{1+\alpha-\gamma}{2}}\Omega(\phi)=(\Psi_\varphi)^{-\frac{1/\mu}{2}}\Omega(\phi),
\end{equation}
for any $\Omega$.
\end{pro}
\begin{proof}
Let us start from \eqref{eq-self-sim-2} taking into account that
$$
\beta_{z_2}\phi_{z_1}-\beta_{z_1}\phi_{z_2}=-|J^{-1}|=-\frac{1}{|J|}=-\frac{2}{(1+\alpha)\Psi_{\beta\varphi}},
$$
that is, \eqref{eq-self-sim-2} agrees with
\begin{equation}\label{eq-self-sim-4}
\frac{1+\alpha-\gamma}{1+\alpha}\Theta(\beta,\phi)\Psi_{\beta\varphi}+\left\{\frac{2}{1+\alpha}\Psi_\phi+\frac{\Psi_{\beta\varphi}}{1+\alpha}(z_1\beta_{z_1}+z_2\beta_{z_2})\right\}\Theta_\beta=0.
\end{equation}
Let us now compute $z_1\beta_{z_1}+z_2\beta_{z_2}=(z_1, z_2)\cdot \nabla_z \beta$. Note that
\begin{equation}\label{jacobian-inverse}
J^{-1}=\left[\begin{array}{cc}
\beta_{z_1}&\beta_{z_2}\\ \phi_{z_1}&\phi_{z_2}
\end{array}
\right]=\frac{1}{|J|}
\left[\begin{array}{cc}
z_{2,\phi}&-z_{1,\phi}\\ -z_{2,\beta}&z_{1,\beta}
\end{array}
\right].
\end{equation}
Then:
\begin{align*}
\beta_{z_1}=&\frac{2}{(1+\alpha)\Psi_{\beta\varphi}}\left\{-\frac{\sin(\vartheta)(1+\alpha)}{2r}\partial^2_{\phi\beta}\Psi+r\cos(\vartheta)\right\},\\
\beta_{z_2}=&-\frac{2}{(1+\alpha)\Psi_{\beta\varphi}}\left\{-\frac{\cos(\vartheta)(1+\alpha)}{2r}\partial^2_{\phi\beta}\Psi-r\sin(\vartheta)\right\}.
\end{align*}
After some straightforward computations we find
$$
z_1\beta_{z_1}+z_2\beta_{z_2}=\frac{2r^2}{(1+\alpha)\Psi_{\beta\varphi}}=-2\frac{\Psi_\beta}{\Psi_{\beta\varphi}}.
$$
Putting everything together, one achieve that \eqref{eq-self-sim-4} is equivalent to \eqref{eq-self-sim-3}, that is:
\begin{align*}
(1+\alpha-\gamma)\Theta(\beta,\phi)\Psi_{\beta\varphi}(\beta,\phi)+2\Psi_{\varphi}(\beta,\phi)\Theta_\beta(\beta,\phi)=0.
\end{align*}
Note that it has the following trivial solution:
$$
\Theta(\beta,\phi)=(\Psi_\varphi)^{-\frac{1+\alpha-\gamma}{2}}\Omega(\phi),
$$
for any function $\Omega$.
\end{proof}
Via the adapted coordinates, we have simplified the nonlinear equation \eqref{eq-self-sim} into \eqref{eq-self-sim-3} where a trivial solution is found by \eqref{trivial-sol-1}.  However, it does not take into account the relation between $\hat{\theta}$ and $\hat{\psi}$ via the singular integral operator \eqref{psi-theta}, and this will be our aim from now on. Note that
\begin{equation}\label{elliptic-eq}
\hat{\psi}(z)=-\frac{C_\gamma}{2\pi}\int_{\R^2}\frac{\hat{\theta}(y)}{|x-y|^\gamma}dA(y).
\end{equation}
By writing \eqref{elliptic-eq} with the adapted coordinates with the help of the previous computation on the jacobian of the transformation in Proposition \ref{prop-jacobian} together with the trivial solution from in Proposition \ref{prop-eq-theta}, we find that it agrees with
\begin{equation}\label{elliptic-eq-2}
\Psi(\beta,\phi)=\frac{C_\gamma (1+\alpha)^{1-\gamma/2}}{4\pi}\int_0^{2\pi}\int_0^\infty\frac{(\Psi_\varphi(b,\Phi))^{-\frac{1+\alpha-\gamma}{2}}\Omega(\Phi)\Psi_{b\varphi}(b,\Phi)dbd\Phi}{\{-\Psi_\beta-\Psi_b+2\Psi_\beta^\frac12\Psi_b^\frac12\cos(\phi+\beta-\Phi-b)\}^\frac{\gamma}{2}}.
\end{equation}
In the previous expression we need to use that $r=0$ corresponds to $\beta=+\infty$ and $r=+\infty$ agrees with $\beta=0$ as we shall see later, which is the reason we have a change of sign in the integral \eqref{elliptic-eq-2}. Hence, our aim is to find functions $\Psi$ and $\Omega$ such that
\begin{align}\label{equation}
F(\Psi,\Omega)&(\beta,\phi):=\Psi(\beta,\phi)\nonumber\\
&-\frac{C_\gamma (1+\alpha)^{1-\gamma/2}}{4\pi}\int_0^{2\pi}\int_0^\infty\frac{(\Psi_\varphi(b,\Phi))^{-\frac{1}{2\mu}}\Omega(\Phi)\Psi_{b\varphi}(b,\Phi)dbd\Phi}{\{-\Psi_\beta-\Psi_b+2\Psi_\beta^\frac12\Psi_b^\frac12\cos(\phi+\beta-\Phi-b)\}^\frac{\gamma}{2}}=0,
\end{align}
for any $\beta\in[0,+\infty)$ and $\phi\in\T$. In the previous expression we have used that $\frac{1}{\mu}=1+\alpha-\gamma$ coming from \eqref{mu}.

\subsection{Trivial solution}
The problem reduces to study the nonlinear equation \eqref{equation} which is nothing but the relation between $\Omega$ and $\Psi$. In this section, we look for trivial solutions of such equation. That trivial solutions comes from the fact that any radial function $\theta$ provides a stationary solution of \eqref{gSQG}, where the stream function is also radial. Here, we will use the computations done in Appendix \ref{ap-radial} for the fractional Laplacian.

\begin{pro}\label{prop-trivial}
If
$$
\Omega^0=\frac{1}{(1+\alpha)^\frac{1}{2\mu}},
$$
and
\begin{equation}\label{Psi0beta}
\Psi^0(\beta)=-C_0^{\frac{2}{1+\alpha}}(\alpha-1)^{\frac{1-\alpha}{1+\alpha}}\beta^{-\frac{1-\alpha}{1+\alpha}}.
\end{equation}
then $F(\Psi^0,\Omega^0)(\beta,\phi)=0$, for any $\beta\in[0,\infty)$ and $\phi\in\T$, where $C_0$ is defined in \eqref{C0}. Moreover, the adapted coordinates $(r,\vartheta)\mapsto (\beta,\phi)$ at the trivial solution reads as follows
 $$\beta=C_0(\alpha-1)r^{-1-\alpha}, \quad \phi=\vartheta-C_0(\alpha-1)r^{-1-\alpha}.$$ 
\end{pro}

\begin{proof}
The trivial solution that we will construct comes from the stationary one 
$$
\theta(x)=|x|^{-\frac{1}{\mu}},
$$
for $\frac12<\mu<\frac{1}{2-\gamma}$.  From Corollary \ref{cor-radial} we have that $\psi$ is also radial and is given by
$$
\psi(x)=-C_0|x|^{2-\gamma-\frac{1}{\mu}}=-C_0|x|^{1-\alpha},
$$
by using \eqref{mu} and where $C_0$ is defined in \eqref{C0}. Using \eqref{theta-self-similar}, we have that $\hat{\theta}(z)=|z|^{-\frac{1}{\mu}}$ and $\hat{\psi}(z)=-C_0|z|^{1-\alpha}$. Let us now obtain such solution in terms of $(\beta,\phi)$.

First, let us find the explicit relation between $\beta$ and $r$ via the nonlinear change of variables in Proposition \ref{prop-coordinates}. Indeed we have that $\Psi_\beta=-\frac{r^2}{1+\alpha}$ and then
\begin{align*}
\partial_r \hat{\psi}(z)=C_0(\alpha-1)r^{-\alpha}=\partial_r \beta \partial_\beta\Psi=-\partial_r\beta \frac{r^2}{1+\alpha}.
\end{align*}
Hence
$$
\partial_r\beta=-C_0(1+\alpha)(\alpha-1)r^{-2-\alpha},
$$
implying
\begin{equation}\label{rel-r-beta}
\beta=C_0(\alpha-1)r^{-1-\alpha},
\end{equation}
which amounts to
$$
r=\left(C_0(\alpha-1)\right)^{\frac{1}{1+\alpha}}\beta^{-\frac{1}{1+\alpha}}.
$$
Since $\alpha>1$ and $r\in[0,\infty)$ we get that $\beta\in[0,\infty)$. In this way, we can compute $\Psi$ depending on $\beta$ as
$$
\Psi^0(\beta)=-C_0^{\frac{2}{1+\alpha}}(\alpha-1)^{\frac{1-\alpha}{1+\alpha}}\beta^{-\frac{1-\alpha}{1+\alpha}}.
$$
Let us now compute $\Omega$. Note that
$$
\Theta(\beta)=\left(\frac{1}{C_0(\alpha-1)}\right)^{\frac{1/\mu}{1+\alpha}}\beta^{\frac{1/\mu}{1+\alpha}},
$$
and using the relation \eqref{trivial-sol-1}, we find
$$
\Omega^0=\frac{1}{(1+\alpha)^\frac{1}{2\mu}}.
$$
That concludes the proof.
\end{proof}

\begin{rem}
The solution that we aim to construct are perturbations of the singular radial stationary solution $\theta(x)=|x|^{-\frac{1}{\mu}}$. The associated initial stream function is also radial and can be computed via Corollary \ref{cor-radial} as $\psi(x)=-C_0|x|^{2-\gamma-\frac{1}{\mu}}=-C_0|x|^{1-\alpha}$. Note that both the potential temperature $\theta$ and the associated stream function have a singularity when $|x|=0$ since $\alpha>1$ via \eqref{alpha-condition}. However, the change of variables $(r,\vartheta)\mapsto (\beta,\phi)$ at the trivial solution gives us $\beta=C_0(1-\alpha)r^{-1-\alpha}$. Hence, $\theta$ and $\psi$ written in terms of $\beta$ are not anymore singular at $\beta=0$, but as $\beta\rightarrow +\infty$.
\end{rem}

\subsection{Formulation with the appropriate scaling}\label{sec-scaling}

Our aim is to construct solutions which are perturbations of the trivial one \eqref{Psi0beta}. Since $\alpha>1$ one has that $\Psi^0$ in \eqref{Psi0beta} is vanishing at $0$ and diverges at $+\infty$. We shall use similar scaling for the perturbations.
First, we shall simplify the constant in $\Psi$ and $\Omega$ by working with $\tilde{\Psi}$ and $\tilde{\Omega}$ instead:
$$
\Psi(\beta,\phi)=-C_0^{\frac{2}{1+\alpha}}(\alpha-1)^{\frac{1-\alpha}{1+\alpha}}\tilde{\Psi}(\beta,\phi),
$$
and
$$
\Omega(\phi)=(1+\alpha)^{-\frac{1}{2\mu}}\tilde{\Omega}(\phi).
$$
In that way, \eqref{equation} reads as
\begin{align*}
&\frac{F(\Psi,\Omega)(\beta,\phi)}{-C_0^{\frac{2}{1+\alpha}}(\alpha-1)^{\frac{1-\alpha}{1+\alpha}}}=\tilde{\Psi}(\beta,\phi)\\
&-(-1)^{-\frac{1}{2\mu}+\frac{\gamma}{2}}\frac{C_\gamma (1+\alpha)^{-\frac{\alpha-1}{2}}(\alpha-1)^\frac{\alpha-1}{2}}{4\pi C_0}\int_0^{2\pi}\int_0^\infty\frac{(\tilde{\Psi}_\varphi(b,\Phi))^{-\frac{1}{2\mu}}\tilde{\Omega}(\Phi)\tilde{\Psi}_{b\varphi}(b,\Phi)dbd\Phi}{\{-\tilde{\Psi}_\beta-\tilde{\Psi}_b+2\tilde{\Psi}_\beta^\frac12\tilde{\Psi}_b^\frac12\cos(\phi+\beta-\Phi-b)\}^\frac{\gamma}{2}}=0.
\end{align*}
Hence, now take $f$ as
$$
\tilde{\Psi}(\beta,\phi)=\beta^{\frac{\alpha-1}{1+\alpha}}\left(1+f(\beta,\phi)\right),
$$
and we divide all the equation by $\beta^{\frac{\alpha-1}{1+\alpha}}$ getting the equivalence formulation
\begin{align*}
&\tilde{F}(f,\tilde{\Omega})(\beta,\phi):=\beta^{-\frac{\alpha-1}{1+\alpha}}\frac{F(\Psi,\Omega)(\beta,\phi)}{-C_0^{\frac{2}{1+\alpha}}(\alpha-1)^{\frac{1-\alpha}{1+\alpha}}}=1+f(\beta,\phi)\\
&-(-1)^{-\frac{1}{2\mu}+\frac{\gamma}{2}}\frac{C_\gamma (1+\alpha)^{-\frac{\alpha-1}{2}}(\alpha-1)^\frac{\alpha-1}{2}}{4\pi C_0}\beta^{-\frac{\alpha-1}{1+\alpha}}\\
&\times\int_0^{2\pi}\int_0^\infty\frac{(\tilde{\Psi}_\varphi(b,\Phi))^{-\frac{1}{2\mu}}\tilde{\Omega}(\Phi)\tilde{\Psi}_{b\varphi}(b,\Phi)dbd\Phi}{\{-\tilde{\Psi}_\beta-\tilde{\Psi}_b+2\tilde{\Psi}_\beta^\frac12\tilde{\Psi}_b^\frac12\cos(\phi+\beta-\Phi-b)\}^\frac{\gamma}{2}}=0.
\end{align*}
We can compute $\partial_\varphi\tilde{\Psi}$, $\partial_\beta\tilde{\Psi}$ and $\partial_{\beta\varphi}\tilde{\Psi}$ in terms of $f$ as follows
\begin{align}
\partial_\beta \tilde{\Psi}(\beta,\phi)=&\beta^{\frac{-2}{1+\alpha}}\left\{\frac{\alpha-1}{1+\alpha}+\partial_{\overline{\beta}}f(\beta,\phi)\right\},\nonumber\\
\partial_\varphi \tilde{\Psi}(\beta,\phi)=&\beta^{\frac{-2}{1+\alpha}}\left\{-\frac{\alpha-1}{1+\alpha}+\partial_{\overline{\varphi}}f(\beta,\phi)\right\},\nonumber\\
\partial_{\beta\varphi} \tilde{\Psi}(\beta,\phi)=&\beta^{-1}\beta^{\frac{-2}{1+\alpha}}\left\{\frac{2(\alpha-1)}{(1+\alpha)^2}+\partial_{\overline{\beta\varphi}}f\right\},\label{Psibetapsi}
\end{align}
where the linear operators $\partial_{\overline{\beta}},\partial_{\overline{\varphi}}$ and $\partial_{\overline{\beta\varphi}}$ reads as
\begin{align}
\partial_{\overline{\beta}}f(\beta,\phi):=&\frac{\alpha-1}{1+\alpha}f(\beta,\phi)+\beta f_\beta(\beta,\phi),\label{lin-op-1}\\
\partial_{\overline{\varphi}}f(\beta,\phi):=&-\frac{\alpha-1}{1+\alpha}f(\beta,\phi)+\beta f_\varphi(\beta,\phi),\label{lin-op-2}\\
\partial_{\overline{\beta\varphi}}f:=&\frac{\alpha-1}{1+\alpha}f+\partial_{\overline{\beta}}\partial_{\overline{\varphi}}f\nonumber\\
=&\frac{2(\alpha-1)}{(1+\alpha)^2}f+\frac{\alpha-1}{1+\alpha}\beta(f_\varphi-f_\beta)+\beta^2 f_{\beta\varphi}.\label{lin-op-3}
\end{align}
Hence, $\tilde{F}$ reads as follows
\begin{align}\label{Ftilde}
\nonumber&\tilde{F}(f,\tilde{\Omega})(\beta,\phi)=1+f(\beta,\phi)\\
\nonumber&-(-1)^\frac{-1}{2\mu}\frac{C_\gamma (\alpha-1)^\frac{\alpha-1}{2}}{4\pi C_0(1+\alpha)^{\frac{\alpha-1}{2}}}\beta^{-\frac{\alpha-1}{1+\alpha}}\\
&\times \int_0^{2\pi}\int_0^\infty\frac{b^{-\frac{\gamma}{1+\alpha}}b^{\frac{-2}{1+\alpha}}\left\{-\frac{\alpha-1}{1+\alpha}+\partial_{\overline{\varphi}}f(b,\Phi)\right\}^{-\frac{1}{2\mu}}\left\{\frac{2(\alpha-1)}{(1+\alpha)^2}+\partial_{\overline{\beta\varphi}}f(b,\Phi)\right\}\tilde{\Omega}(\Phi)dbd\Phi}{D(f)(\beta,\phi,b,\Phi)^\frac{\gamma}{2}},
\end{align}
where
\begin{align}\label{denominator-0}
D(f)(\beta,\phi,b,\Phi):=&\tilde{\Psi}_\beta+\tilde{\Psi}_b-2\tilde{\Psi}_\beta^\frac12\tilde{\Psi}_b^\frac12\cos(\phi+\beta-\Phi-b)\\
=&\beta^{\frac{-2}{1+\alpha}}\left\{\frac{\alpha-1}{1+\alpha}+\partial_{\overline{\beta}}f(\beta,\phi)\right\}+b^{\frac{-2}{1+\alpha}}\left\{\frac{\alpha-1}{1+\alpha}+\partial_{\overline{b}}f(b,\Phi)\right\}\nonumber\\
&-2\beta^{\frac{-1}{1+\alpha}}b^{\frac{-1}{1+\alpha}}\left\{\frac{\alpha-1}{1+\alpha}+\partial_{\overline{\beta}}f(\beta,\phi)\right\}^\frac12\left\{\frac{\alpha-1}{1+\alpha}+\partial_{\overline{b}}f(b,\Phi)\right\}^\frac12\cos(\phi+\beta-\Phi-b)\nonumber\\
=&\left(\beta^{\frac{-1}{1+\alpha}}\left\{\frac{\alpha-1}{1+\alpha}+\partial_{\overline{\beta}}f(\beta,\phi)\right\}^\frac12-b^{\frac{-1}{1+\alpha}}\left\{\frac{\alpha-1}{1+\alpha}+\partial_{\overline{b}}f(b,\Phi)\right\}^\frac12\right)^2\nonumber\\
&+2\beta^{\frac{-1}{1+\alpha}}b^{\frac{-1}{1+\alpha}}\left\{\frac{\alpha-1}{1+\alpha}+\partial_{\overline{\beta}}f(\beta,\phi)\right\}^\frac12\left\{\frac{\alpha-1}{1+\alpha}+\partial_{\overline{b}}f(b,\Phi)\right\}^\frac12\nonumber\\
&\times (1-\cos(\phi+\beta-\Phi-b)).\nonumber
\end{align}

Let us remark that in the expression of $\tilde{F}$ no derivatives in $\phi$ appear. Indeed, it only depends on $\partial_\beta$ and $\partial_\varphi$. In that way, we will not assume any boundedness of $\partial_\phi$ in the function spaces. Moreover, since we are working the an unbounded space $[0,\infty)$ and the trivial solution is singular at $\beta=0$, we shall insert some weights in the function spaces as we will see in the following section.

\section{Functional regularity}\label{sec-functional}
\subsection{Function spaces}
Let us describe the function spaces for the scaled perturbation $f$. Indeed, we will take that $f$ is bounded and has some appropriate decay at $0$ and $+\infty$. We need such a decay for $f$ in order to tackle the singularity at $\beta=0$ and $\beta\rightarrow +\infty$ in the nonlinear equation.

For $\sigma>0$ (which will be fixed later), let us define the weighted norm
\begin{equation}\label{norm}
||f||_{L^\infty_\sigma}:=\sup_{\beta\in[0,+\infty),\phi\in[0,2\pi]} \frac{(1+\beta)^{2\sigma}}{\beta^\sigma}|f(\beta,\phi)|.
\end{equation}
We will assume that $f$ is H\"older continuous in the variable $\beta$ and then define
\begin{equation}\label{norm-holder}
||f||_{C^\delta_\beta}:=\sup_{\phi\in[0,2\pi]}\sup_{|\beta-b|<1}|\beta+b|^\delta\frac{|f(\beta,\phi)-f(b,\phi)|}{|\beta-b|^\delta},
\end{equation}
and Lipschitz continuous in the periodic variable $\phi$:
\begin{equation}\label{norm-lipschitz}
||f||_{\textnormal{Lip}_\phi}:=\sup_{\beta\in[0,+\infty)}\sup_{|\phi-\Phi|<1}\frac{|f(\beta,\phi)-f(\beta,\Phi)|}{\sin((\Phi-\phi)/2)}.
\end{equation}
Hence, let us define
\begin{equation}\label{X-0}
\mathcal{X}^{\sigma,\delta}:=\left\{f:(\beta,\phi)\in[0,+\infty)\times[0,2\pi]\rightarrow\C, \quad  ||f||_{\mathcal{X}^{\sigma,\delta}}:=||f||_{L^\infty_\sigma}+||f||_{C^\delta_\beta}<+\infty \right\}.
\end{equation}
Note that the functions  $f$ lying in $\mathcal{X}^{\sigma,\delta}$ satisfy
$$
 f\sim_0 \beta^{\sigma}, \quad f\sim_\infty \beta^{-\sigma},
$$
which will be crucial to work with the singularity at $0$ and $+\infty$.

Now, we can define the function spaces for $f$:
\begin{equation}\label{X1}
{X}_1^{\sigma,\delta}:=\{f,\quad f, \beta f_\beta,\beta f_\varphi, \beta^2 f_{\beta\varphi}\in\mathcal{X}^{\sigma,\delta}, ||f||_{\textnormal{Lip}_\phi}+||\beta f_\beta||_{\textnormal{Lip}_\phi}<+\infty, f(\beta,\phi)=\sum_{k\in\Z}f_k(\beta)e^{ik\phi}\},
\end{equation}
with the norm
\begin{equation}\label{norm-X1}
||f||_{X_1^{\sigma,\delta}}:=||f||_{\mathcal{X}^{\sigma,\delta}}+||\beta f_\beta||_{\mathcal{X}^{\sigma,\delta}}+||\beta f_\varphi||_{\mathcal{X}^{\sigma,\delta}}+||\beta^2 f_{\varphi\beta}||_{\mathcal{X}^{\sigma,\delta}}+||f||_{\textnormal{Lip}_\phi}+||\beta f_\beta||_{\textnormal{Lip}_\phi}.
\end{equation}
Notice that the expression of the nonlinear operator \eqref{Ftilde} only involves derivatives of the type $\partial_\varphi$ and $\partial_\beta$, and not $\partial_\phi$. For that reason, we only consider such derivatives in the function space for the perturbations \eqref{X1} and we will take into account that $\partial_\phi=\partial_\varphi+\partial_\beta$. However, we shall ask for Lipschitz continuity in the $\phi$ variable in order to tackle the singularity of the denominator in \eqref{Ftilde}.

Moreover, let us denote by $B_{X_1^{\sigma,\delta}}(\epsilon)$  the ball in $X_1^{\sigma,\delta}$ centered at 0 and radius $\epsilon$. Let us remark that $1\notin X_1^{\sigma,\delta}$ (this is the associated trivial solution after the scaling in \eqref{Ftilde}) because it does not have the same decay at $0$ and $+\infty$. Hence, we will have to subtract the contribution from the trivial solution from the nonlinear operator to obtain its well posedness.

\begin{rem}
For $f\in X_1^{\sigma,\delta}$, one has that $\partial_{\overline{\beta}}f,\partial_{\overline{\varphi}}f, \partial_{\overline{\beta\varphi}}f \in \mathcal{X}^{\sigma,\delta}$, where those linear operators are defined in \eqref{lin-op-1}--\eqref{lin-op-3}.
\end{rem}

We will assume that $\Omega$ lies in the Lebesgue space $L^p$ for some $p$ that will be fixed later. Then, define
\begin{equation}\label{X2}
X_2^p:=\{f\in L^p(0,2\pi), \quad f(\phi)=1+\sum_{0\neq k\in\Z} f_k e^{ik\phi}\}.
\end{equation}

\begin{rem}
Our goal is to prove self-similar solutions for a large family of perturbations of the initial vorticity. Hence, we need to assume that $\Omega$ lies in a large space. Due to the regularity of the nonlinear function $\tilde{F}$, we realized that the minimum condition to have that all the analysis works is that $\Omega\in L^p$, for some appropriate $p$ that will depend on $\gamma$.
\end{rem}
Finally, we set the range space:
\begin{equation}\label{Y}
Y^{\sigma,\delta}:=\{f,\quad f, \beta f_\varphi, \in\mathcal{X}^{\sigma,\delta}, ||f||_{\textnormal{Lip}_\phi}<+\infty\},
\end{equation}
with the norm
$$
||f||_{Y^{\sigma,\delta}}:=||f||_{\mathcal{X}^{\sigma,\delta}}+||\beta f_\varphi|_{\mathcal{X}^{\sigma,\delta}}+||f||_{\textnormal{Lip}_\phi}.
$$
From now on, let us denote by $B_{X}(c,r)$ the ball in $X$ centered at $c$ of radius $r$.

\begin{rem}\label{rem-fourier-coef}
Note that any function $f\in X_1^{\sigma,\delta}$ can be decompose in Fourier series in $\phi$, that is, there exists $f_k(\beta)$ such that
$$
f(\beta,\phi)=\sum_{k\in\Z}f_k(\beta) e^{ik\phi}.
$$
Although we are looking for a real function $f$, which will be the perturbation of the initial stream function, we will consider it initially as a complex function. Later, we can just include in the function spaces that $f$ is real.

Moreover, we can translate the decay properties of $f$ into the coefficients $f_k$. In that way, we find that $f_k\in\mathcal{X}_0$ where 
\begin{align}
\mathcal{X}_0:=&\left\{f,\quad f,\beta f'\in\mathcal{X}^{\sigma,\delta}\right\}. \label{X01}
\end{align}
However, let us remark that functions with Fourier coefficients in $\mathcal{X}_0$ are not necessary in $X_1^{\sigma,\delta}$. We can do the same with $f\in Y^{\sigma,\delta}$ having that the associated Fourier coefficients lie in the following space:
\begin{align}
\mathcal{Y}_0:=&\left\{f,\quad f\in\mathcal{X}^{\sigma,\delta}\right\}. \label{X00}
\end{align}
These spaces will be used in order to check that the kernel of the linearized operator is trivial.
\end{rem}

\begin{rem}
The function spaces used in \cite{Bressan-Murray:self-similar-euler, Elling:algebraic-spiral-solutions-2d-euler, Elling:self-similar-euler-mixed-sign-vorticity} for the Euler equations can not be used here. This is due to the different formulation of the nonlinear function $F$ used in these papers. Note that the relation between the stream function $\psi$ and the vorticity $\omega$ in the Euler equations is given via a local relation. There, there is no need to use the nonlocal expression via the singular integral, as we have to do for gSQG equations.
\end{rem}

\begin{rem}
We consider that $\Omega$ is of the form
$$
\Omega(\phi)=1+\sum_{0\neq k\in\Z}g_k e^{ik\phi},
$$
where $1$ is the trivial solution and then no perturbation of the mode $k=0$ is done there. That ensures us that we do not obtain more trivial solutions and that the solution is non radial.
\end{rem}

\begin{rem}
Function spaces with weights are commonly used when working in an unbounded space. Moreover, such weight must be consider when taking any derivative. That happens also in \cite{Elgindi-Jeong:symmetries-fluids}, where the authors needs to include some weight even for the H\"older norm. Moreover, note that in the case that $f$ only depends on the angle $\phi$ (and not on $\beta$), then these are the usual spaces $C^1, L^\infty$,...
\end{rem}

\subsection{Deformation of the Euclidean norm}

In order to study the regularity of the nonlinear function \eqref{Ftilde}, we shall work with the denominator given in \eqref{denominator-0}. Such denominator is nothing but the euclidean norm deformed by the adapted coordinates and the scaling. That happened also in \cite{Garcia-Hmidi-Mateu:time-periodic-doubly-3d-qg,Garcia-Hmidi-Mateu:time-periodic-3d-qg}, where the authors performed some cylindrical coordinates amounting to a deformation of the euclidean norm. Hence, estimating by below the denominator, i.e., 
\begin{align}\label{exp-denominator-2}
D(f)(\beta,\phi,b,\Phi):=&\tilde{\Psi}_\beta+\tilde{\Psi}_b-2\tilde{\Psi}_\beta^\frac12\tilde{\Psi}_b^\frac12\cos(\phi+\beta-\Phi-b)\\
=&\left(\beta^{\frac{-1}{1+\alpha}}\left\{\frac{\alpha-1}{1+\alpha}+\partial_{\overline{\beta}}f(\beta,\phi)\right\}^\frac12-b^{\frac{-1}{1+\alpha}}\left\{\frac{\alpha-1}{1+\alpha}+\partial_{\overline{b}}f(b,\Phi)\right\}^\frac12\right)^2\nonumber\\
&+2\beta^{\frac{-1}{1+\alpha}}b^{\frac{-1}{1+\alpha}}\left\{\frac{\alpha-1}{1+\alpha}+\partial_{\overline{\beta}}f(\beta,\phi)\right\}^\frac12\left\{\frac{\alpha-1}{1+\alpha}+\partial_{\overline{b}}f(b,\Phi)\right\}^\frac12\nonumber \\
&\times(1-\cos(\phi+\beta-\Phi-b)),\nonumber
\end{align}
is crucial to study the regularity of \eqref{Ftilde}. That is the goal of the following proposition.

\begin{pro}\label{prop-denominator}
There exists $\epsilon>0$ such that if $f\in B_{X_1^{\sigma,\delta}}(0,\epsilon)$ then
\begin{align}\label{den-estim-1}
D(f)(\beta,\phi,b,\Phi)\geq&C(\beta^{\frac{-2}{1+\alpha}}+b^{\frac{-2}{1+\alpha}})\left\{ \frac{|\beta-b|^2}{|\beta+b|^2}+\sin^2\left((\beta-b+\phi-\Phi)/2\right)\right\}.
\end{align}
\end{pro}
\begin{rem}\label{rem-denominator}
In particular, we find that $D(f)(\beta,\phi,b,\Phi)\geq C D(0)(\beta,\phi,b,\Phi)$. Moreover, we also get that $D(f)(\beta,\phi,b,\Phi)\leq C D(0)(\beta,\phi,b,\Phi)$.
\end{rem}

\begin{proof}

Take $f$ such that $f\in B_{X_1^{\sigma,\delta}}(\epsilon)$. In particular $\partial_{\overline{\beta}}f\in \mathcal{X}^{\sigma,\delta}$ and $||\partial_{\overline{\beta}}f||_{\textnormal{Lip}_\phi}<+\infty$.

Hence, we have that
$$
C^{-1}<\left\{\frac{\alpha-1}{1+\alpha}+\partial_{\overline{\beta}}f(\beta,\phi)\right\}<C,
$$
and
$$
C^{-1}<\left\{\frac{\alpha-1}{1+\alpha}+\partial_{\overline{b}}f(b,\Phi)\right\}<C.
$$
That, by symmetry, implies
\begin{align*}
D(f)(\beta,\phi,b,\Phi)\geq& \left(\beta^{\frac{-1}{1+\alpha}}\left\{\frac{\alpha-1}{1+\alpha}+\partial_{\overline{\beta}}f(\beta,\phi)\right\}^\frac12-b^{\frac{-1}{1+\alpha}}\left\{\frac{\alpha-1}{1+\alpha}+\partial_{\overline{b}}f(b,\Phi)\right\}^\frac12\right)^2\nonumber\\
&+C^{-1}(\beta^{\frac{-2}{1+\alpha}}+b^{\frac{-2}{1+\alpha}})(1-\cos(\phi+\beta-\Phi-b)),\nonumber
\end{align*}
for $C$ uniformly in $b,\beta,\phi,\Phi$. Let us then work with the first part. Note that
\begin{align*}
&\left|\beta^{\frac{-1}{1+\alpha}}\left\{\frac{\alpha-1}{1+\alpha}+\partial_{\overline{\beta}}f(\beta,\phi)\right\}^\frac12-b^{\frac{-1}{1+\alpha}}\left\{\frac{\alpha-1}{1+\alpha}+\partial_{\overline{b}}f(b,\Phi)\right\}^\frac12\right|\\
&=\left|\frac{\beta^{\frac{-2}{1+\alpha}}\left\{\frac{\alpha-1}{1+\alpha}+\partial_{\overline{\beta}}f(\beta,\phi)\right\}-b^{\frac{-2}{1+\alpha}}\left\{\frac{\alpha-1}{1+\alpha}+\partial_{\overline{b}}f(b,\Phi)\right\}}{\beta^{\frac{-1}{1+\alpha}}\left\{\frac{\alpha-1}{1+\alpha}+\partial_{\overline{\beta}}f(\beta,\phi)\right\}^\frac12+b^{\frac{-1}{1+\alpha}}\left\{\frac{\alpha-1}{1+\alpha}+\partial_{\overline{b}}f(b,\Phi)\right\}^\frac12}\right|\\
&\gtrsim \frac{\left|\beta^{\frac{-2}{1+\alpha}}\left\{\frac{\alpha-1}{1+\alpha}+\partial_{\overline{\beta}}f(\beta,\phi)\right\}-b^{\frac{-2}{1+\alpha}}\left\{\frac{\alpha-1}{1+\alpha}+\partial_{\overline{b}}f(b,\Phi)\right\}\right|}{\beta^{\frac{-1}{1+\alpha}}+b^{\frac{-1}{1+\alpha}}}.
\end{align*}
On the other hand, we have
\begin{align*}
&\left|\beta^{\frac{-2}{1+\alpha}}\left\{\frac{\alpha-1}{1+\alpha}+\partial_{\overline{\beta}}f(\beta,\phi)\right\}-b^{\frac{-2}{1+\alpha}}\left\{\frac{\alpha-1}{1+\alpha}+\partial_{\overline{b}}f(b,\Phi)\right\}\right|\gtrsim |\beta^{\frac{-2}{1+\alpha}}-b^{\frac{-2}{1+\alpha}}|\\
&-(\beta^{\frac{-2}{1+\alpha}}+b^{\frac{-2}{1+\alpha}})|\partial_{\overline{\beta}}f(\beta,\phi)-\partial_{\overline{b}}f(b,\Phi)|\\
\gtrsim& |\beta^{\frac{-2}{1+\alpha}}+b^{\frac{-2}{1+\alpha}}|\left\{\frac{|\beta-b|}{\beta+b}-|\partial_{\overline{\beta}}f(\beta,\phi)-\partial_{\overline{b}}f(b,\Phi)|\right\}.
\end{align*}
Then, we find
\begin{align*}
&\left|\beta^{\frac{-1}{1+\alpha}}\left\{\frac{\alpha-1}{1+\alpha}+\partial_{\overline{\beta}}f(\beta,\phi)\right\}^\frac12-b^{\frac{-1}{1+\alpha}}\left\{\frac{\alpha-1}{1+\alpha}+\partial_{\overline{b}}f(b,\Phi)\right\}^\frac12\right|\\
&\gtrsim |\beta^{\frac{-1}{1+\alpha}}+b^{\frac{-1}{1+\alpha}}|\left\{\frac{|\beta-b|}{\beta+b}-|\partial_{\overline{\beta}}f(\beta,\phi)-\partial_{\overline{b}}f(b,\Phi)|\right\},
\end{align*}
which implies
\begin{align*}
&\left|\beta^{\frac{-1}{1+\alpha}}\left\{\frac{\alpha-1}{1+\alpha}+\partial_{\overline{\beta}}f(\beta,\phi)\right\}^\frac12-b^{\frac{-1}{1+\alpha}}\left\{\frac{\alpha-1}{1+\alpha}+\partial_{\overline{b}}f(b,\Phi)\right\}^\frac12\right|^2\\
&\gtrsim |\beta^{\frac{-2}{1+\alpha}}+b^{\frac{-2}{1+\alpha}}|\left\{\frac{|\beta-b|^2}{|\beta+b|^2}-|\partial_{\overline{\beta}}f(\beta,\phi)-\partial_{\overline{b}}f(b,\Phi)|^2\right\}.
\end{align*}
That is,
\begin{align}\label{esti-den}
D(f)(\beta,\phi,b,\Phi)\gtrsim&(\beta^{\frac{-2}{1+\alpha}}+b^{\frac{-2}{1+\alpha}})\left\{ \frac{|\beta-b|^2}{|\beta+b|^2}+(1-\cos(\phi+\beta-\Phi-b))-|\partial_{\overline{\beta}}f(\beta,\phi)-\partial_{\overline{b}}f(b,\Phi)|^2\right\}.
\end{align}
It remains to work with 
$$
|\partial_{\overline{\beta}}f(\beta,\phi)-\partial_{\overline{b}}f(b,\Phi)|.
$$
Since we should bound this quantity by $\frac{|\beta-b|}{|\beta+b|}$ or $1-\cos(\phi+\beta-\Phi-b)$, we need to define the auxiliary function
$$
g(\beta,x)=\partial_{\overline{\beta}}f(\beta, x-\beta).
$$
In that way, taking $x=\beta+\phi$ and $y=b+\Phi$ one has
\begin{equation}\label{den-aux}
|\partial_{\overline{\beta}}f(\beta,\phi)-\partial_{\overline{b}}f(b,\Phi)|=|g(\beta,x)-g(b,y)|\leq |g(\beta,x)-g(b,x)|+|g(b,x)-g(b,y)|.
\end{equation}
For the first term we use the mean value theorem. Note that $\partial_\beta g$ behaves as $\partial_\varphi f$ and $\beta \partial_{\beta\varphi}f$, by the definition of $\partial_{\overline{\beta}}$. Then, from the definition of $X_1^{\sigma,\delta}$ in \eqref{X1} we have that $\beta \partial_\varphi g$ is bounded. In that way, one finds
$$
|g(\beta,x)-g(b,x)|\leq C||f||_{X_1^{\sigma,\delta}}\frac{|\beta-b|}{\beta+b}.
$$
For the second term, we use the Lipschitz continuity of $f$ and $\beta f_\beta$ (and then $\partial_{\overline{\beta}}f$) with respect to $\phi$ obtaining
$$
|g(b,x)-g(b,y)|\leq C||f||_{X_1^{\sigma,\delta}}|\sin((x-y)/2)|.
$$
By using that $x=\beta+\phi$ and $y=b+\Phi$ we find
$$
|\partial_{\overline{\beta}}f(\beta,\phi)-\partial_{\overline{b}}f(b,\Phi)|\leq C||f||_{X_1^{\sigma,\delta}}\left\{\frac{|\beta-b|}{\beta+b}+\sin^2\left((\beta-b+\phi-\Phi)/2\right)\right\}.
$$
Coming back to \eqref{esti-den} and using that $||f||_{X_1^{\sigma,\delta}}$ is small enough we achieve \eqref{den-estim-1}:
\begin{align*}
D(f)(\beta,\phi,b,\Phi)\geq&(\beta^{\frac{-2}{1+\alpha}}+b^{\frac{-2}{1+\alpha}})\left\{ \frac{|\beta-b|^2}{|\beta+b|^2}+\sin^2\left((\beta-b+\phi-\Phi)/2\right)\right\}.
\end{align*}
That concludes the proof.
\end{proof}

\subsection{Regularity of the nonlinear function}

In the following proposition, we deal with singular operators defined as 
\begin{equation}\label{H-operator}
\mathcal{H}_f(h_1,h_2)(\beta,\phi):=\beta^{-\frac{\alpha-1}{1+\alpha}}\int_0^{2\pi}\int_0^\infty \frac{b^{-\frac{\gamma}{1+\alpha}}b^{\frac{-2}{1+\alpha}}h_1(b,\Phi)h_2(\Phi)dbd\Phi}{D(f)(\beta,\phi,b,\Phi)^\frac{\gamma}{2}}.
\end{equation}
This will be crucial in order to prove later the persistence regularity of the nonlinear functional \eqref{equation}.

\begin{pro}\label{prop-operator-H}
Let $\gamma,\sigma\in(0,1)$ and $\alpha\in(1,2)$. Then, there exists $\epsilon<1$ such that if $f\in B_{X_1^{\sigma,\delta}} (0,\epsilon)$ and 
\begin{equation}\label{sigma-rel}
0<\sigma<\min\left\{\frac{\alpha-1}{1+\alpha},\frac{\gamma-\alpha+1}{1+\alpha}\right\},  \quad 1<\alpha<1+\gamma,
\end{equation}
then
\begin{align}\label{est-integral}
\left\|\mathcal{H}(h_1,h_2)\right\|_{Y^{\sigma,\delta}}\leq C||h_1||_{\mathcal{X}^{\sigma,\delta}}||h_2||_{X_2^p},
\end{align}
for any $p>\frac{1}{1-\gamma}$ and {$\delta<\min\{1-\gamma,\sigma\}$}. 
\end{pro}
\begin{proof}
The proof is very technical and will be divided in different steps. We will provide the key points and the general ideas.\\

\noindent\medskip $\bullet$ $\mathcal{H}_f\in L^\infty_\sigma$.

Using Proposition \ref{prop-denominator}, we can bound $\mathcal{H}$ as
\begin{align*}
|\mathcal{H}_f(h_1,h_2)(\beta,\phi)|\leq &C \beta^{-\frac{\alpha-1}{1+\alpha}}\beta^{\frac{\gamma}{1+\alpha}}\int_0^{2\pi}\int_0^\infty b^{\frac{-2}{1+\alpha}}\frac{|h_1(b,\phi+\beta-\eta-b)||h_2(\phi+\beta-\eta-b)|dbd\eta}{(\beta^\frac{\gamma}{1+\alpha}+b^\frac{\gamma}{1+\alpha})\left\{\frac{|\beta-b|^2}{|\beta+b|^2}+\sin^2\left(\frac{\eta}{2}\right)\right\}^\frac{\gamma}{2}}\\
\leq &C \beta^{-\frac{\alpha-1}{1+\alpha}}\beta^{\frac{\gamma}{1+\alpha}}\int_0^{2\pi}\int_0^\infty \frac{b^{\frac{-2}{1+\alpha}}b^\sigma/(1+b)^{2\sigma}|h_2(\phi+\beta-\eta-b)|}{(\beta^\frac{\gamma}{1+\alpha}+b^\frac{\gamma}{1+\alpha})\left\{\frac{|\beta-b|^2}{|\beta+b|^2}+\sin^2\left(\frac{\eta}{2}\right)\right\}^\frac{\gamma}{2}}\\
&\times \frac{(1+b)^{2\sigma}}{b^\sigma}|h_1(b,\phi+\beta-\eta-b)|dbd\eta\\
=&C \beta^\sigma\int_0^{2\pi}\int_0^\infty \frac{z^{\frac{-2}{1+\alpha}}z^\sigma/(1+\beta z)^{2\sigma}|h_2(\phi+\beta-\eta-\beta z)|}{(1+z^\frac{\gamma}{1+\alpha})\left\{\frac{|1-z|^2}{|1+z|^2}+\sin^2\left(\frac{\eta}{2}\right)\right\}^\frac{\gamma}{2}}\\
&\times \frac{(1+\beta z)^{2\sigma}}{(\beta z)^\sigma}|h_1(\beta z,\phi+\beta-\eta-\beta z)|dzd\eta,
\end{align*}
where we have done the change of variables $b=\beta z$ in the last equality. Then, we find
\begin{align*}
||\mathcal{H}(h_1,h_2)||_{L^\infty_\sigma} \leq & C|| h_1||_{L^\infty_\sigma}||h_2||_{L^1}\sup_{\beta\in[0,+\infty)}\int_0^\infty \frac{(1+\beta)^{2\sigma}}{(1+\beta z)^{2\sigma}}  \frac{z^{\frac{-2}{1+\alpha}}z^\sigma|1+z|^\gamma}{|1-z|^\gamma|1+z|^{\frac{\gamma}{1+\alpha}}}dz.
\end{align*}
Recall that $\alpha>1$ by \eqref{mu-intro}. In order to estimate the above integral, we split in cases. Take $\varepsilon>0$, then let us work with $\beta\in[0,\varepsilon)$ and $\beta\in[\varepsilon,+\infty)$ separately. First,
$$
\sup_{\beta\in[0,\varepsilon)}\int_0^\infty \frac{(1+\beta)^{2\sigma}}{(1+\beta z)^{2\sigma}}  \frac{z^{\frac{-2}{1+\alpha}}z^\sigma|1+z|^\gamma}{|1-z|^\gamma|1+z|^{\frac{\gamma}{1+\alpha}}}dz\leq C \sup_{\beta\in[0,\varepsilon)}\int_0^\infty    \frac{z^{\frac{-2}{1+\alpha}}z^\sigma|1+z|^\gamma}{|1-z|^\gamma|1+z|^{\frac{\gamma}{1+\alpha}}}dz\leq C,
$$
if 
\begin{equation}\label{sigma1}
0<\sigma<\frac{\gamma-\alpha+1}{1+\alpha}, \quad 1<\alpha<1+\gamma.
\end{equation}
In the other case:
$$
\sup_{\beta\in[\varepsilon,+\infty)}\int_0^\infty \frac{(1+\beta)^{2\sigma}}{(1+\beta z)^{2\sigma}}  \frac{z^{\frac{-2}{1+\alpha}}z^\sigma|1+z|^\gamma}{|1-z|^\gamma|1+z|^{\frac{\gamma}{1+\alpha}}}dz\leq C \sup_{\beta\in[\varepsilon,+\infty)}\int_0^\infty  \frac{1}{z^\sigma} \frac{z^{\frac{-2}{1+\alpha}}|1+z|^\gamma}{|1-z|^\gamma|1+z|^{\frac{\gamma}{1+\alpha}}}dz\leq C,
$$
if 
\begin{equation}\label{sigma2}
0<\sigma<\frac{\alpha-1}{1+\alpha}.
\end{equation}
Hence, we find that
\begin{equation}\label{sigma3}
0<\sigma<\min\left\{\frac{\alpha-1}{1+\alpha},\frac{\gamma-\alpha+1}{1+\alpha}\right\},  \quad 1<\alpha<1+\gamma.
\end{equation}

\noindent\medskip $\bullet$ $\mathcal{H}_f\in C^\delta_\beta$.

Let us now check the H\"older semi-norm. Take $\beta_1<\beta_2$ and then
\begin{equation}\label{Holder-betas}
\beta_2^{-\frac{\alpha-1}{1+\alpha}}<\beta_1^{-\frac{\alpha-1}{1+\alpha}}.
\end{equation}
Doing the difference of $\mathcal{H}_f$ at two points we arrive at
\begin{align}\label{H-holder-0}
&|\mathcal{H}_f(h_1,h_2)(\beta_1,\phi)-\mathcal{H}_f(h_1,h_2)(\beta_2,\phi)|=\Bigg|\beta_1^{-\frac{\alpha-1}{1+\alpha}}\int_0^{2\pi}\int_0^\infty \frac{b^{-\frac{\gamma}{1+\alpha}}b^{\frac{-2}{1+\alpha}}h_1(b,\Phi)h_2(\Phi)dbd\Phi}{D(f)(\beta_1,\phi,b,\Phi)^\frac{\gamma}{2}}\\
&-\beta_2^{-\frac{\alpha-1}{1+\alpha}}\int_0^{2\pi}\int_0^\infty \frac{b^{-\frac{\gamma}{1+\alpha}}b^{\frac{-2}{1+\alpha}}h_1(b,\Phi)h_2(\Phi)dbd\Phi}{D(f)(\beta_2,\phi,b,\Phi)^\frac{\gamma}{2}}\Bigg|\nonumber\\
\leq &C\frac{|\beta_1-\beta_2|^\delta}{|\beta_1+\beta_2|^\delta}\left(\beta_1^{-\frac{\alpha-1}{1+\alpha}}+\beta_2^{-\frac{\alpha-1}{1+\alpha}}\right){\Bigg|\int_0^{2\pi}\int_0^\infty \frac{b^{-\frac{\gamma}{1+\alpha}}b^{\frac{-2}{1+\alpha}}h_1(b,\Phi)h_2(\Phi)dbd\Phi}{D(f)(\beta_1,\phi,b,\Phi)^\frac{\gamma}{2}}\Bigg|}\nonumber\\
&+\beta_2^{-\frac{\alpha-1}{1+\alpha}}\Bigg|\int_0^{2\pi}\int_0^\infty \frac{b^{-\frac{\gamma}{1+\alpha}}b^{\frac{-2}{1+\alpha}}h_1(b,\Phi)h_2(\Phi)dbd\Phi}{D(f)(\beta_1,\phi,b,\Phi)^\frac{\gamma}{2}}-\int_0^{2\pi}\int_0^\infty \frac{b^{-\frac{\gamma}{1+\alpha}}b^{\frac{-2}{1+\alpha}}h_1(b,\Phi)h_2(\Phi)dbd\Phi}{D(f)(\beta_2,\phi,b,\Phi)^\frac{\gamma}{2}}\Bigg|.\nonumber
\end{align}
Due to \eqref{Holder-betas} and using the previous analysis to bound $\mathcal{H}_f$ in $L^\infty_\sigma$ we get that
$$
\left(\beta_1^{-\frac{\alpha-1}{1+\alpha}}+\beta_2^{-\frac{\alpha-1}{1+\alpha}}\right)\Bigg|\int_0^{2\pi}\int_0^\infty \frac{b^{-\frac{\gamma}{1+\alpha}}b^{\frac{-2}{1+\alpha}}h_1(b,\Phi)h_2(\Phi)dbd\Phi}{D(f)(\beta_1,\phi,b,\Phi)^\frac{\gamma}{2}}\Bigg|<C.
$$
Hence, we need to estimate the second term in \eqref{H-holder-0}. Due to its expression, we will need to work with the difference of the denominator at two points. Thus
\begin{align}\label{H-holder}
&D(f)(\beta_1,\phi,b,\Phi)-D(f)(\beta_2,\phi,b,\Phi)=\left(\beta_1^{\frac{-1}{1+\alpha}}\left\{\frac{\alpha-1}{1+\alpha}+\partial_{\overline{\beta}}f(\beta_1,\phi)\right\}^\frac12-b^{\frac{-1}{1+\alpha}}\left\{\frac{\alpha-1}{1+\alpha}+\partial_{\overline{b}}f(b,\Phi)\right\}^\frac12\right)^2\nonumber\\
&-\left(\beta_2^{\frac{-1}{1+\alpha}}\left\{\frac{\alpha-1}{1+\alpha}+\partial_{\overline{\beta}}f(\beta_2,\phi)\right\}^\frac12-b^{\frac{-1}{1+\alpha}}\left\{\frac{\alpha-1}{1+\alpha}+\partial_{\overline{b}}f(b,\Phi)\right\}^\frac12\right)^2\nonumber\\
&+2\beta_1^{\frac{-1}{1+\alpha}}b^{\frac{-1}{1+\alpha}}\left\{\frac{\alpha-1}{1+\alpha}+\partial_{\overline{\beta}}f(\beta_1,\phi)\right\}^\frac12\left\{\frac{\alpha-1}{1+\alpha}+\partial_{\overline{b}}f(b,\Phi)\right\}^\frac12(1-\cos(\phi+\beta_1-\Phi-b))\nonumber\\
&-2\beta_2^{\frac{-1}{1+\alpha}}b^{\frac{-1}{1+\alpha}}\left\{\frac{\alpha-1}{1+\alpha}+\partial_{\overline{\beta}}f(\beta_2,\phi)\right\}^\frac12\left\{\frac{\alpha-1}{1+\alpha}+\partial_{\overline{b}}f(b,\Phi)\right\}^\frac12(1-\cos(\phi+\beta_2-\Phi-b)).
\end{align}
Let us estimate it in different steps. First
\begin{align*}
&\Big|\left(\beta_1^{\frac{-1}{1+\alpha}}\left\{\frac{\alpha-1}{1+\alpha}+\partial_{\overline{\beta}}f(\beta_1,\phi)\right\}^\frac12-b^{\frac{-1}{1+\alpha}}\left\{\frac{\alpha-1}{1+\alpha}+\partial_{\overline{b}}f(b,\Phi)\right\}^\frac12\right)^2\\
&-\left(\beta_2^{\frac{-1}{1+\alpha}}\left\{\frac{\alpha-1}{1+\alpha}+\partial_{\overline{\beta}}f(\beta_2,\phi)\right\}^\frac12-b^{\frac{-1}{1+\alpha}}\left\{\frac{\alpha-1}{1+\alpha}+\partial_{\overline{b}}f(b,\Phi)\right\}^\frac12\right)^2\Big|\\
\leq &\left|\beta_1^{\frac{-1}{1+\alpha}}\left\{\frac{\alpha-1}{1+\alpha}+\partial_{\overline{\beta}}f(\beta_1,\phi)\right\}^\frac12-\beta_2^{\frac{-1}{1+\alpha}}\left\{\frac{\alpha-1}{1+\alpha}+\partial_{\overline{\beta}}f(\beta_2,\Phi)\right\}^\frac12\right|\\
&\times\Big|\left(\beta_1^{\frac{-1}{1+\alpha}}\left\{\frac{\alpha-1}{1+\alpha}+\partial_{\overline{\beta}}f(\beta_1,\phi)\right\}^\frac12-b^{\frac{-1}{1+\alpha}}\left\{\frac{\alpha-1}{1+\alpha}+\partial_{\overline{b}}f(b,\Phi)\right\}^\frac12\right)\\
&+\left(\beta_2^{\frac{-1}{1+\alpha}}\left\{\frac{\alpha-1}{1+\alpha}+\partial_{\overline{\beta}}f(\beta_2,\phi)\right\}^\frac12-b^{\frac{-1}{1+\alpha}}\left\{\frac{\alpha-1}{1+\alpha}+\partial_{\overline{b}}f(b,\Phi)\right\}^\frac12\right)\Big|,
\end{align*}
where
\begin{align*}
&\left|\beta_1^{\frac{-1}{1+\alpha}}\left\{\frac{\alpha-1}{1+\alpha}+\partial_{\overline{\beta}}f(\beta_1,\phi)\right\}^\frac12-\beta_2^{\frac{-1}{1+\alpha}}\left\{\frac{\alpha-1}{1+\alpha}+\partial_{\overline{\beta}}f(\beta_2,\Phi)\right\}^\frac12\right|\\
\leq & C\left|\frac{\beta_1^{\frac{-2}{1+\alpha}}\left\{\frac{\alpha-1}{1+\alpha}+\partial_{\overline{\beta}}f(\beta_1,\phi)\right\}-\beta_2^{\frac{-2}{1+\alpha}}\left\{\frac{\alpha-1}{1+\alpha}+\partial_{\overline{\beta}}f(\beta_2,\Phi)\right\}}{\beta_1^{\frac{-1}{1+\alpha}}+\beta_2^{\frac{-1}{1+\alpha}}}\right|\\
\leq & C\left|\frac{(\beta_1^{\frac{-2}{1+\alpha}}-\beta_2^{\frac{-2}{1+\alpha}})\left\{\frac{\alpha-1}{1+\alpha}+\partial_{\overline{\beta}}f(\beta_1,\phi)\right\}+\beta_2^{\frac{-2}{1+\alpha}}\left\{\partial_{\overline{\beta}}f(\beta_1,\Phi)-\partial_{\overline{\beta}}f(\beta_2,\Phi)\right\}}{\beta_1^{\frac{-1}{1+\alpha}}+\beta_2^{\frac{-1}{1+\alpha}}}\right|\\
\leq &C\frac{\beta_1^{\frac{-2}{1+\alpha}}+\beta_2^{\frac{-2}{1+\alpha}}}{\beta_1^{\frac{-1}{1+\alpha}}+\beta_2^{\frac{-1}{1+\alpha}}}\frac{|\beta_1-\beta_2|^\delta}{|\beta_1+\beta_2|^\delta}\\
\leq &C\left\{\beta_1^{\frac{-1}{1+\alpha}}+\beta_2^{\frac{-1}{1+\alpha}}\right\}\frac{|\beta_1-\beta_2|^\delta}{|\beta_1+\beta_2|^\delta}.
\end{align*}
Hence
\begin{align*}
&\Big|\left(\beta_1^{\frac{-1}{1+\alpha}}\left\{\frac{\alpha-1}{1+\alpha}+\partial_{\overline{\beta}}f(\beta_1,\phi)\right\}^\frac12-b^{\frac{-1}{1+\alpha}}\left\{\frac{\alpha-1}{1+\alpha}+\partial_{\overline{b}}f(b,\Phi)\right\}^\frac12\right)^2\\
&-\left(\beta_2^{\frac{-1}{1+\alpha}}\left\{\frac{\alpha-1}{1+\alpha}+\partial_{\overline{\beta}}f(\beta_2,\phi)\right\}^\frac12-b^{\frac{-1}{1+\alpha}}\left\{\frac{\alpha-1}{1+\alpha}+\partial_{\overline{b}}f(b,\Phi)\right\}^\frac12\right)^2\Big|\\
\leq& C\Big|\left(\beta_1^{\frac{-1}{1+\alpha}}\left\{\frac{\alpha-1}{1+\alpha}+\partial_{\overline{\beta}}f(\beta_1,\phi)\right\}^\frac12-b^{\frac{-1}{1+\alpha}}\left\{\frac{\alpha-1}{1+\alpha}+\partial_{\overline{b}}f(b,\Phi)\right\}^\frac12\right)\\
&+\left(\beta_2^{\frac{-1}{1+\alpha}}\left\{\frac{\alpha-1}{1+\alpha}+\partial_{\overline{\beta}}f(\beta_2,\phi)\right\}^\frac12-b^{\frac{-1}{1+\alpha}}\left\{\frac{\alpha-1}{1+\alpha}+\partial_{\overline{b}}f(b,\Phi)\right\}^\frac12\right)\Big|\\
&\times \left\{\beta_1^{\frac{-1}{1+\alpha}}+\beta_2^{\frac{-1}{1+\alpha}}\right\}\frac{|\beta_1-\beta_2|^\delta}{|\beta_1+\beta_2|^\delta}\\
\leq &C(D(f)(\beta_1,\phi,b,\Phi)^\frac12+D(f)(\beta_2,\phi,b,\Phi)^\frac12)\left\{\beta_1^{\frac{-1}{1+\alpha}}+\beta_2^{\frac{-1}{1+\alpha}}\right\}\frac{|\beta_1-\beta_2|^\delta}{|\beta_1+\beta_2|^\delta}.
\end{align*}
Now let us estimate the second part of \eqref{H-holder} using similar ideas
\begin{align*}
&\Big|\beta_1^{\frac{-1}{1+\alpha}}b^{\frac{-1}{1+\alpha}}\left\{\frac{\alpha-1}{1+\alpha}+\partial_{\overline{\beta}}f(\beta_1,\phi)\right\}^\frac12\left\{\frac{\alpha-1}{1+\alpha}+\partial_{\overline{b}}f(b,\Phi)\right\}^\frac12(1-\cos(\phi+\beta_1-\Phi-b))\nonumber\\
&-\beta_2^{\frac{-1}{1+\alpha}}b^{\frac{-1}{1+\alpha}}\left\{\frac{\alpha-1}{1+\alpha}+\partial_{\overline{\beta}}f(\beta_2,\phi)\right\}^\frac12\left\{\frac{\alpha-1}{1+\alpha}+\partial_{\overline{b}}f(b,\Phi)\right\}^\frac12(1-\cos(\phi+\beta_2-\Phi-b))\Big|\\
\leq &C\frac{|\beta_1-\beta_2|^\delta}{|\beta_1+\beta_2|^\delta}(\beta_1^\frac{-1}{1+\alpha}+\beta_2^\frac{-1}{1+\alpha})b^\frac{-1}{1+\alpha}\left\{(\sin^2((\phi+\beta_1-\Phi-b)/2))+(\sin^2((\phi+\beta_2-\Phi-b)/2))\right\}\\
&+C|\sin^\delta((\beta_1-\beta_2)/2)|(\beta_1^\frac{-1}{1+\alpha}+\beta_2^\frac{-1}{1+\alpha})b^\frac{-1}{1+\alpha}\\
&\times \left\{(\sin^{2-\delta}((\phi+\beta_1-\Phi-b)/2))+(\sin^{2-\delta}((\phi+\beta_2-\Phi-b)/2))\right\}\\
\leq &C\frac{|\beta_1-\beta_2|^\delta}{|\beta_1+\beta_2|^\delta}(\beta_1^\frac{-1}{1+\alpha}+\beta_2^\frac{-1}{1+\alpha})b^\frac{-1}{1+\alpha}\left\{(\sin^2((\phi+\beta_1-\Phi-b)/2))+(\sin^2((\phi+\beta_2-\Phi-b)/2))\right\}\\
&+C\frac{|\beta_1-\beta_2|^\delta}{|\beta_1+\beta_2|^\delta}(\beta_1^{\frac{-1}{1+\alpha}+\delta}+\beta_2^{\frac{-1}{1+\alpha}+\delta})b^\frac{-1}{1+\alpha}\left\{(\sin^{2-\delta}((\phi+\beta_1-\Phi-b)/2))+(\sin^{2-\delta}((\phi+\beta_2-\Phi-b)/2))\right\}\\
\leq &C \frac{|\beta_1-\beta_2|^\delta}{|\beta_1+\beta_2|^\delta}(\beta_1^\frac{-1}{1+\alpha}+\beta_2^\frac{-1}{1+\alpha})(D(f)(\beta_1,\phi,b,\Phi)^\frac12+D(f)(\beta_2,\phi,b,\Phi)^\frac12)\\
&+C\frac{|\beta_1-\beta_2|^\delta}{|\beta_1+\beta_2|^\delta}(\beta_1^{\frac{-1}{1+\alpha}+\delta}+\beta_2^{\frac{-1}{1+\alpha}+\delta})\left\{(\sin^{1-\delta}((\phi+\beta_1-\Phi-b)/2))+(\sin^{1-\delta}((\phi+\beta_2-\Phi-b)/2))\right\}\\
&\times (D(f)(\beta_1,\phi,b,\Phi)^\frac12+D(f)(\beta_2,\phi,b,\Phi)^\frac12).
\end{align*}
Then, at the end we get
\begin{align}\label{H-holder2}
|D(f)(\beta_2,\phi,b,\Phi)-D(f)(\beta_1,\phi,b,\Phi)|\leq& C \frac{|\beta_1-\beta_2|^\delta}{|\beta_1+\beta_2|^\delta}(\beta_1^\frac{-1}{1+\alpha}+\beta_2^\frac{-1}{1+\alpha}+\beta_1^{\frac{-1}{1+\alpha}+\delta}+\beta_2^{\frac{-1}{1+\alpha}+\delta})\nonumber\\
&\times(D(f)(\beta_1,\phi,b,\Phi)^\frac12+D(f)(\beta_2,\phi,b,\Phi)^\frac12).
\end{align}
Hence
\begin{align*}
&\left|\frac{1}{D(f)(\beta_1,\phi,b,\Phi)^\frac{\gamma}{2}}-\frac{1}{D(f)(\beta_2,\phi,b,\Phi)^\frac{\gamma}{2}}\right|\leq \frac{D(f)(\beta_2,\phi,b,\Phi)^\frac{\gamma}{2}-D(f)(\beta_1,\phi,b,\Phi)^\frac{\gamma}{2}}{D(f)(\beta_1,\phi,b,\Phi)^\frac{\gamma}{2} D(f)(\beta_2,\phi,b,\Phi)^\frac{\gamma}{2}}\\
\lesssim& \frac{|D(f)(\beta_2,\phi,b,\Phi)-D(f)(\beta_1,\phi,b,\Phi)|}{D(f)(\beta_1,\phi,b,\Phi)^\frac{\gamma}{2}D(f)(\beta_2,\phi,b,\Phi)^\frac{\gamma}{2}}\frac{1}{(D(f)(\beta_2,\phi,b,\Phi)^\frac{2-\gamma}{2}+D(f)(\beta_1,\phi,b,\Phi)^\frac{2-\gamma}{2})},
\end{align*}
and using \eqref{H-holder2} we find
\begin{align}\label{Dfbetabeta2}
&\left|\frac{1}{D(f)(\beta_1,\phi,b,\Phi)^\frac{\gamma}{2}}-\frac{1}{D(f)(\beta_2,\phi,b,\Phi)^\frac{\gamma}{2}}\right|\\
\leq &C \frac{|\beta_1-\beta_2|^\delta}{|\beta_1+\beta_2|^\delta}\frac{\beta_1^\frac{-1}{1+\alpha}+\beta_2^\frac{-1}{1+\alpha}+\beta_1^{\frac{-1}{1+\alpha}+\delta}+\beta_2^{\frac{-1}{1+\alpha}+\delta}}{D(f)(\beta_1,\phi,b,\Phi)^\frac12 D(f)(\beta_2,\phi,b,\Phi)^\frac{\gamma}{2}+D(f)(\beta_2,\phi,b,\Phi)^\frac12 D(f)(\beta_1,\phi,b,\Phi)^\frac{\gamma}{2}}.\nonumber
\end{align}
Let us come back to \eqref{H-holder-0} to achieve
\begin{align*}
&\left|\mathcal{H}_f(h_1,h_2)(\beta_1,\phi)-\mathcal{H}_f(h_1,h_2)(\beta_2,\phi)\right|\leq C\frac{|\beta_1-\beta_2|^\delta}{|\beta_1+\beta_2|^\delta}\\
&+C\frac{|\beta_1-\beta_2|^\delta}{|\beta_1+\beta_2|^\delta}\beta_2^{-\frac{\alpha-1}{1+\alpha}}\int_0^{2\pi}\int_0^\infty b^{-\frac{\gamma}{1+\alpha}}b^{\frac{-2}{1+\alpha}}|h_1(b,\Phi)||h_2(\Phi)|dbd\Phi\\
&\times  \frac{\beta_1^\frac{-1}{1+\alpha}+\beta_2^\frac{-1}{1+\alpha}+\beta_1^{\frac{-1}{1+\alpha}+\delta}+\beta_2^{\frac{-1}{1+\alpha}+\delta}}{D(f)(\beta_1,\phi,b,\Phi)^\frac12 D(f)(\beta_2,\phi,b,\Phi)^\frac{\gamma}{2}+D(f)(\beta_2,\phi,b,\Phi)^\frac12 D(f)(\beta_1,\phi,b,\Phi)^\frac{\gamma}{2}} dbd\Phi.
\end{align*}
By virtue of Proposition \ref{prop-denominator} we can bound the denominator as
{\footnotesize{\begin{align*}
&\left|\mathcal{H}_f(h_1,h_2)(\beta_1,\phi)-\mathcal{H}_f(h_1,h_2)(\beta_2,\phi)\right|\leq C\frac{|\beta_1-\beta_2|^\delta}{|\beta_1+\beta_2|^\delta}\\
&+C\frac{|\beta_1-\beta_2|^\delta}{|\beta_1+\beta_2|^\delta}\int_0^{2\pi}\int_0^\infty b^{-\frac{\gamma}{1+\alpha}}b^{\frac{-2}{1+\alpha}}|h_1(b,\Phi)||h_2(\Phi)|\\
&\times  \Bigg(\frac{\beta_2^{-\frac{\alpha-1}{1+\alpha}}\beta_1^\frac{-1}{1+\alpha}}{(\beta_1^\frac{-1}{1+\alpha}+b^\frac{-1}{1+\alpha})(\beta_2^\frac{-\gamma}{1+\alpha}+b^\frac{-\gamma}{1+\alpha}) \left\{\frac{|\beta_1-b|^2}{|\beta_1+b|^2}+\sin^2((\beta_1+\phi-b-\Phi)/2)\right\}^\frac12\left\{\frac{|\beta_2-b|^2}{|\beta_2+b|^2}+\sin^2((\beta_2+\phi-b-\Phi)/2)\right\}^\frac{\gamma}{2}}\\
&+\frac{\beta_2^{-\frac{\alpha-1}{1+\alpha}}\beta_2^\frac{-1}{1+\alpha}}{(\beta_2^\frac{-1}{1+\alpha}+b^\frac{-1}{1+\alpha})(\beta_1^\frac{-\gamma}{1+\alpha}+b^\frac{-\gamma}{1+\alpha}) \left\{\frac{|\beta_2-b|^2}{|\beta_2+b|^2}+\sin^2((\beta_2+\phi-b-\Phi)/2)\right\}^\frac12\left\{\frac{|\beta_1-b|^2}{|\beta_1+b|^2}+\sin^2((\beta_1+\phi-b-\Phi)/2)\right\}^\frac{\gamma}{2}}\\
&+\frac{\beta_2^{-\frac{\alpha-1}{1+\alpha}}\beta_1^\frac{-1}{1+\alpha}\beta_1^\delta}{(\beta_1^\frac{-1}{1+\alpha}+b^\frac{-1}{1+\alpha})(\beta_2^\frac{-\gamma}{1+\alpha}+b^\frac{-\gamma}{1+\alpha}) \left\{\frac{|\beta_1-b|^2}{|\beta_1+b|^2}+\sin^2((\beta_1+\phi-b-\Phi)/2)\right\}^\frac12\left\{\frac{|\beta_2-b|^2}{|\beta_2+b|^2}+\sin^2((\beta_2+\phi-b-\Phi)/2)\right\}^\frac{\gamma}{2}}\\
&+\frac{\beta_2^{-\frac{\alpha-1}{1+\alpha}}\beta_2^\frac{-1}{1+\alpha}\beta_2^\delta}{(\beta_2^\frac{-1}{1+\alpha}+b^\frac{-1}{1+\alpha})(\beta_1^\frac{-\gamma}{1+\alpha}+b^\frac{-\gamma}{1+\alpha}) \left\{\frac{|\beta_2-b|^2}{|\beta_2+b|^2}+\sin^2((\beta_2+\phi-b-\Phi)/2)\right\}^\frac12\left\{\frac{|\beta_1-b|^2}{|\beta_1+b|^2}+\sin^2((\beta_1+\phi-b-\Phi)/2)\right\}^\frac{\gamma}{2}}\Bigg) dbd\Phi.
\end{align*}
}}
Let us show the bound for the third integral, and the others will follow similarly. Using now that $h_1\in \mathcal{X}^{\sigma,\delta}$ and then
{\footnotesize{\begin{align*}
&\int_0^{2\pi}\int_0^\infty \frac{\beta_2^{-\frac{\alpha-1}{1+\alpha}}\beta_1^{\frac{-1}{1+\alpha}+\delta}b^{-\frac{\gamma}{1+\alpha}}b^{\frac{-2}{1+\alpha}}|h_1(b,\Phi)||h_2(\Phi)|dbd\Phi}{(\beta_1^\frac{-1}{1+\alpha}+b^\frac{-1}{1+\alpha})(\beta_2^\frac{-\gamma}{1+\alpha}+b^\frac{-\gamma}{1+\alpha}) \left\{\frac{|\beta_1-b|^2}{|\beta_1+b|^2}+\sin^2((\beta_1+\phi-b-\Phi)/2)\right\}^\frac12\left\{\frac{|\beta_2-b|^2}{|\beta_2+b|^2}+\sin^2((\beta_2+\phi-b-\Phi)/2)\right\}^\frac{\gamma}{2}}\\
&\leq C||h_1||_{L^\infty_{\sigma}} \beta_1^\delta\int_0^{2\pi}\int_0^\infty \frac{\beta_2^{-\frac{\alpha-1}{1+\alpha}}b^{-\frac{\gamma}{1+\alpha}}b^{\frac{-2}{1+\alpha}}b^\sigma/(1+b)^{2\sigma}|h_2(\Phi)|dbd\Phi}{(\beta_2^\frac{-\gamma}{1+\alpha}+b^\frac{-\gamma}{1+\alpha}) \left\{\frac{|\beta_1-b|^2}{|\beta_1+b|^2}+\sin^2((\beta_1+\phi-b-\Phi)/2)\right\}^\frac12\left\{\frac{|\beta_2-b|^2}{|\beta_2+b|^2}+\sin^2((\beta_2+\phi-b-\Phi)/2)\right\}^\frac{\gamma}{2}}\\
&\leq C||h_1||_{L^\infty_{\sigma}}\beta_1^\delta \int_0^{2\pi}\int_0^\infty \frac{\beta_2^\frac{\gamma}{1+\alpha}\beta_2^{-\frac{\alpha-1}{1+\alpha}}b^{\frac{-2}{1+\alpha}}b^\sigma/(1+b)^{2\sigma} |h_2(\Phi)|dbd\Phi}{(\beta_2^\frac{\gamma}{1+\alpha}+b^\frac{\gamma}{1+\alpha}) \frac{|\beta_1-b|^{1-\nu}}{|\beta_1+b|^{1-\nu}}|\sin((\beta_1+\phi-b-\Phi)/2)|^\nu\frac{|\beta_2-b|^\gamma}{|\beta_2+b|^\gamma}}
\end{align*}
}}
for $\nu\in(0,1)$ with 
\begin{equation}\label{nu-gamma}
{1-\nu+\gamma\in(0,1).}
\end{equation} Hence
{\footnotesize{\begin{align*}
&\int_0^{2\pi}\int_0^\infty \frac{\beta_2^{-\frac{\alpha-1}{1+\alpha}}\beta_1^{\frac{-1}{1+\alpha}+\delta}b^{-\frac{\gamma}{1+\alpha}}b^{\frac{-2}{1+\alpha}}|h_1(b,\Phi)||h_2(\Phi)|dbd\Phi}{(\beta_1^\frac{-1}{1+\alpha}+b^\frac{-1}{1+\alpha})(\beta_2^\frac{-\gamma}{1+\alpha}+b^\frac{-\gamma}{1+\alpha}) \left\{\frac{|\beta_1-b|^2}{|\beta_1+b|^2}+\sin^2((\beta_1+\phi-b-\Phi)/2)\right\}^\frac12\left\{\frac{|\beta_2-b|^2}{|\beta_2+b|^2}+\sin^2((\beta_2+\phi-b-\Phi)/2)\right\}^\frac{\gamma}{2}}\\
&\leq C||h_1||_{L^\infty_{\sigma}}\beta_1^\delta \int_0^\infty \frac{\beta_2^\frac{\gamma}{1+\alpha}\beta_2^{-\frac{\alpha-1}{1+\alpha}}b^{\frac{-2}{1+\alpha}}b^\sigma/(1+b)^{2\sigma} }{(\beta_2^\frac{\gamma}{1+\alpha}+b^\frac{\gamma}{1+\alpha}) \frac{|\beta_1-b|^{1-\nu}}{|\beta_1+b|^{1-\nu}}\frac{|\beta_2-b|^\gamma}{|\beta_2+b|^\gamma}}\int_0^{2\pi} |\sin^{-\nu}((\beta_2+\phi-b-\Phi)/2)||h_2(\Phi)|d\Phi db
\end{align*}
}}
Note that by interpolation we get
\begin{align}\label{H-interp}
\int_0^{2\pi}|\sin^{-\nu}((\beta_2+\phi-b-\Phi)/2)||h_2(\Phi)|d\Phi\leq& C\left(\int_0^{2\pi}|h_2(\Phi)|^p d\Phi\right)^\frac{1}{p}\left(\int_0^{2\pi}|\sin^{-\nu}((\Phi)/2)|^q d\Phi\right)^\frac{1}{q}\nonumber\\
\leq& C||h_2||_{X_2^p},
\end{align}
provided that {$\nu q<1$ and that $\frac{1}{p}+\frac{1}{q}=1$. By using \eqref{nu-gamma}, we can choose any $p$ such that 
\begin{equation}\label{condition-p}
\frac{1}{1-\gamma}<p.
\end{equation}
} Finally, using the change of variables $b=\beta_2 z$ we get:
{\footnotesize{\begin{align}\label{Hholder-2}
&\int_0^{2\pi}\int_0^\infty \frac{\beta_2^{-\frac{\alpha-1}{1+\alpha}}\beta_1^{\frac{-1}{1+\alpha}+\delta}b^{-\frac{\gamma}{1+\alpha}}b^{\frac{-2}{1+\alpha}}|h_1(b,\Phi)||h_2(\Phi)|dbd\Phi}{(\beta_1^\frac{-1}{1+\alpha}+b^\frac{-1}{1+\alpha})(\beta_2^\frac{-\gamma}{1+\alpha}+b^\frac{-\gamma}{1+\alpha}) \left\{\frac{|\beta_1-b|^2}{|\beta_1+b|^2}+\sin^2((\beta_1+\phi-b-\Phi)/2)\right\}^\frac12\left\{\frac{|\beta_2-b|^2}{|\beta_2+b|^2}+\sin^2((\beta_2+\phi-b-\Phi)/2)\right\}^\frac{\gamma}{2}}\nonumber\\
&\leq C||h_1||_{L^\infty_{\sigma}} ||h_2||_{X_2^p}\beta_1^\delta\int_0^\infty \frac{\beta_2^\frac{\gamma}{1+\alpha}\beta_2^{-\frac{\alpha-1}{1+\alpha}}b^{\frac{-2}{1+\alpha}}b^\sigma/(1+b)^{2\sigma} }{(\beta_2^\frac{\gamma}{1+\alpha}+b^\frac{\gamma}{1+\alpha}) \frac{|\beta_1-b|^{1-\nu}}{|\beta_1+b|^{1-\nu}}\frac{|\beta_2-b|^\gamma}{|\beta_2+b|^\gamma}}db\\
&\leq C||h_1||_{L^\infty_{\sigma}} ||h_2||_{X_2^p}\beta_2^\delta\int_0^\infty \frac{z^{\frac{-2}{1+\alpha}}(z\beta_2)^\sigma/(1+z\beta_2)^{2\sigma} }{(1+z^\frac{\gamma}{1+\alpha}) \frac{|\beta_1-z\beta_2|^{1-\nu}}{|\beta_1+z\beta_2|^{1-\nu}}\frac{|1-z|^\gamma}{|1+z|^\gamma}}dz\nonumber\\
&\leq C ||h_1||_{L^\infty_{\sigma}} ||h_2||_{X_2^p},\nonumber
\end{align}
}}
where we have used that $\beta_1<\beta_2$, and which is bounded provided that $\sigma$ verifies \eqref{sigma1} and \eqref{sigma2} and
$$
\delta<\sigma,
$$
in order to get the boundedness as $\beta_2$ goes to $\infty$. Hence, putting everything together we find that
$$
|\mathcal{H}_f(\beta_1,\phi,b,\Phi)-\mathcal{H}_f(\beta_2,\phi,b,\Phi)|\leq C ||h_1||_{L^\infty_{\sigma}} ||h_2||_{X_2^p}\frac{|\beta_1-\beta_2|^\delta}{|\beta_1+\beta_2|^\delta},
$$
implying that $H_f\in C^\delta_\beta$.

\noindent\medskip $\bullet$ $\mathcal{H}_f\in \textnormal{Lip}_\phi$.

We shall prove that
\begin{align}\label{H-lips}
|\mathcal{H}_f(\beta,\phi_1,b,\Phi)-\mathcal{H}_f(\beta,\phi_2,b,\Phi)|\leq C|\sin((\phi_1-\phi_2)/2)|.
\end{align}
For that we need to estimate the difference of the denominator at two points:
\begin{align*}
&D(f)(\beta,\phi_1,b,\Phi)-D(f)(\beta,\phi_2,b,\Phi)=\left(\beta^{\frac{-1}{1+\alpha}}\left\{\frac{\alpha-1}{1+\alpha}+\partial_{\overline{\beta}}f(\beta,\phi_1)\right\}^\frac12-b^{\frac{-1}{1+\alpha}}\left\{\frac{\alpha-1}{1+\alpha}+\partial_{\overline{b}}f(b,\Phi)\right\}^\frac12\right)^2\nonumber\\
&-\left(\beta^{\frac{-1}{1+\alpha}}\left\{\frac{\alpha-1}{1+\alpha}+\partial_{\overline{\beta}}f(\beta,\phi_2)\right\}^\frac12-b^{\frac{-1}{1+\alpha}}\left\{\frac{\alpha-1}{1+\alpha}+\partial_{\overline{b}}f(b,\Phi)\right\}^\frac12\right)^2\nonumber\\
&+2\beta^{\frac{-1}{1+\alpha}}b^{\frac{-1}{1+\alpha}}\left\{\frac{\alpha-1}{1+\alpha}+\partial_{\overline{\beta}}f(\beta,\phi_1)\right\}^\frac12\left\{\frac{\alpha-1}{1+\alpha}+\partial_{\overline{b}}f(b,\Phi)\right\}^\frac12(1-\cos(\phi_1+\beta-\Phi-b))\nonumber\\
&-2\beta^{\frac{-1}{1+\alpha}}b^{\frac{-1}{1+\alpha}}\left\{\frac{\alpha-1}{1+\alpha}+\partial_{\overline{\beta}}f(\beta,\phi_2)\right\}^\frac12\left\{\frac{\alpha-1}{1+\alpha}+\partial_{\overline{b}}f(b,\Phi)\right\}^\frac12(1-\cos(\phi_2+\beta-\Phi-b)).
\end{align*}
Hence we can achieve that
\begin{align}\label{H-phi-2}
&|D(f)(\beta,\phi_1,b,\Phi)-D(f)(\beta,\phi_2,b,\Phi)|\leq C |\sin((\phi_1-\phi_2)/2)|{(\beta^{\frac{-1}{1+\alpha}}+b^\frac{-1}{1+\alpha})}\nonumber\\
&\times(D(f)(\beta,\phi_1,b,\Phi)^\frac12+D(f)(\beta,\phi_2,b,\Phi)^\frac12).
\end{align}
Hence
\begin{align*}
&\left|\frac{1}{D(f)(\beta,\phi_1,b,\Phi)^\frac{\gamma}{2}}-\frac{1}{D(f)(\beta,\phi_2,b,\Phi)^\frac{\gamma}{2}}\right|\leq \frac{D(f)(\beta,\phi_1,b,\Phi)^\frac{\gamma}{2}-D(f)(\beta,\phi_2,b,\Phi)^\frac{\gamma}{2}}{D(f)(\beta,\phi_1,b,\Phi)^\frac{\gamma}{2} D(f)(\beta,\phi_2,b,\Phi)^\frac{\gamma}{2}}\\
\leq& \frac{|D(f)(\beta,\phi_1,b,\Phi)-D(f)(\beta,\phi_2,b,\Phi)|}{D(f)(\beta,\phi_1,b,\Phi)^\frac{\gamma}{2}D(f)(\beta,\phi_2,b,\Phi)^\frac{\gamma}{2}}\frac{1}{D(f)(\beta,\phi_2,b,\Phi)^\frac{2-\gamma}{2}+D(f)(\beta,\phi_1,b,\Phi)^\frac{2-\gamma}{2}},
\end{align*}
and using \eqref{H-phi-2} we find
\begin{align*}
&\left|\frac{1}{D(f)(\beta,\phi_1,b,\Phi)^\frac{\gamma}{2}}-\frac{1}{D(f)(\beta,\phi_2,b,\Phi)^\frac{\gamma}{2}}\right|\\
\leq &C|\sin((\phi_1-\phi_2)/2)|\frac{{(\beta^{\frac{-1}{1+\alpha}}+b^\frac{-1}{1+\alpha})}}{D(f)(\beta,\phi_1,b,\Phi)^\frac12 D(f)(\beta,\phi_2,b,\Phi)^\frac{\gamma}{2}+D(f)(\beta,\phi_2,b,\Phi)^\frac12 D(f)(\beta,\phi_1,b,\Phi)^\frac{\gamma}{2}}.
\end{align*}
In that way, we obtain that
\begin{align*}
&|\mathcal{H}_f(\beta,\phi_1,b,\Phi)-\mathcal{H}_f(\beta,\phi_2,b,\Phi)|\\
\leq& C |\sin((\phi_1-\phi_2)/2)|\beta^{-\frac{\alpha-1}{1+\alpha}}\\
&\times \int_0^{2\pi}\int_0^\infty \frac{b^{-\frac{\gamma}{1+\alpha}}b^\frac{-2}{1+\alpha}{(\beta^{\frac{-1}{1+\alpha}}+b^\frac{-1}{1+\alpha})}|h_1(b,\Phi)|h_2(\Phi)|db d\Phi}{D(f)(\beta,\phi_1,b,\Phi)^\frac12 D(f)(\beta,\phi_2,b,\Phi)^\frac{\gamma}{2}+D(f)(\beta,\phi_2,b,\Phi)^\frac12 D(f)(\beta,\phi_1,b,\Phi)^\frac{\gamma}{2}}\\
\leq& C |\sin((\phi_1-\phi_2)/2)|\beta^{-\frac{\alpha-1}{1+\alpha}}\int_0^{2\pi}\int_0^\infty \frac{b^{-\frac{\gamma}{1+\alpha}}b^\frac{-2}{1+\alpha}{(\beta^{\frac{-1}{1+\alpha}}+b^\frac{-1}{1+\alpha})}|h_1(b,\Phi)|h_2(\Phi)|db d\Phi}{D(f)(\beta,\phi_1,b,\Phi)^\frac12 D(f)(\beta,\phi_2,b,\Phi)^\frac{\gamma}{2}}\\
\leq& C |\sin((\phi_1-\phi_2)/2)|\beta^{-\frac{\alpha-1}{1+\alpha}}\\
&\times\int_0^{2\pi}\int_0^\infty \frac{b^{-\frac{\gamma}{1+\alpha}}b^\frac{-2}{1+\alpha}{(\beta^{\frac{-1}{1+\alpha}}+b^\frac{-1}{1+\alpha})}|h_1(b,\Phi)|h_2(\Phi)|db d\Phi}{\left\{\beta^{\frac{-\gamma}{1+\alpha}}+b^\frac{-\gamma}{1+\alpha}\right\}\left\{\beta^{\frac{-1}{1+\alpha}}+b^\frac{-1}{1+\alpha}\right\}\frac{|\beta-b|^\gamma}{|\beta+b|^\gamma}\left\{\frac{|\beta-b|^2}{|\beta+b|^2}+\sin^2((\beta-b+\phi_1-\Phi)/2)\right\}^\frac12}\\
\leq& C |\sin((\phi_1-\phi_2)/2)|\beta^{-\frac{\alpha-1}{1+\alpha}}\int_0^{2\pi}\int_0^\infty \frac{b^{-\frac{\gamma}{1+\alpha}}{b^\frac{-2}{1+\alpha}}|h_1(b,\Phi)||h_2(\Phi)|db d\Phi}{\left\{\beta^{\frac{-\gamma}{1+\alpha}}+b^\frac{-\gamma}{1+\alpha}\right\}\frac{|\beta-b|^{\gamma+1-\nu}}{|\beta+b|^{\gamma+1-\nu}}|\sin((\beta-b+\phi_1-\Phi)/2)|^\nu}.
\end{align*}
Using now \eqref{H-interp} we get
\begin{align*}
&|\mathcal{H}_f(\beta,\phi_1,b,\Phi)-\mathcal{H}_f(\beta,\phi_2,b,\Phi)|\\
\leq& C||h_1||_{X_1^{\sigma,\delta}}||h_2||_{X_2^p} |\sin((\phi_1-\phi_2)/2)|\beta^{-\frac{\alpha-1}{1+\alpha}}\int_0^\infty \frac{b^{-\frac{\gamma}{1+\alpha}}b^\frac{-2}{1+\alpha}b^\sigma/(1+b)^{2\sigma}db }{\left\{\beta^{\frac{-\gamma}{1+\alpha}}+b^\frac{-\gamma}{1+\alpha}\right\}\frac{|\beta-b|^{\gamma+1-\nu}}{|\beta+b|^{\gamma+1-\nu}}},
\end{align*}
{for $p$ with $\frac{1}{p}+\frac{1}{q}=1$ and $q\nu<1$ as done in \eqref{H-interp}}.
Doing now the change of variable $b=z\beta$ amounts to
\begin{align*}
&|\mathcal{H}_f(\beta,\phi_1,b,\Phi)-\mathcal{H}_f(\beta,\phi_2,b,\Phi)|\\
\leq& C||h_1||_{\mathcal{X}^{\sigma,\delta}}||h_2||_{X_2^p} |\sin((\phi_1-\phi_2)/2)|\beta^\sigma\int_0^\infty \frac{z^\frac{-2}{1+\alpha}z^\sigma/(1+z\beta)^{2\sigma}dz }{\left\{1+z^\frac{\gamma}{1+\alpha}\right\}\frac{|1-z|^{\gamma+1-\nu}}{|1+z|^{\gamma+1-\nu}}}\\\
\leq& C||h_1||_{\mathcal{X}^{\sigma,\delta}}||h_2||_{X_2^p},
\end{align*}
provided that $\gamma+1-\nu<1$, $p$ satisfies \eqref{condition-p} and $\sigma$ satisfies \eqref{sigma1} and \eqref{sigma2}.

\noindent\medskip $\bullet$ $\beta\partial_\varphi \mathcal{H}_f\in L^\infty_\sigma$.

By using the expression of $\mathcal{H}_f$ we find
\begin{align*}
\partial_\varphi \mathcal{H}_f(h_1,h_2)(\beta,\phi)=&\frac{\alpha-1}{1+\alpha}\frac{1}{\beta}\mathcal{H}_f(h_1,h_2)(\beta,\phi)\\
&-\frac{\gamma}{2}\beta^{-\frac{\alpha-1}{1+\alpha}}\int_0^{2\pi}\int_0^\infty \frac{b^\frac{-\gamma}{1+\alpha}b^\frac{-2}{1+\alpha}h_1(b,\Phi)h_2(\Phi)\partial_\varphi D(f)(\beta,\phi,b,\Phi)}{D(f)(\beta,\phi,b,\Phi)^\frac{2+\gamma}{2}}dbd\Phi.
\end{align*}

Note that
\begin{align}\label{H-betavarphi}
\beta \partial_\varphi \mathcal{H}_f(h_1,h_2)(\beta,\phi)=&\frac{\alpha-1}{1+\alpha}\mathcal{H}_f(h_1,h_2)(\beta,\phi)\nonumber\\
&-\frac{\gamma}{2}\beta \beta^{-\frac{\alpha-1}{1+\alpha}}\int_0^{2\pi}\int_0^\infty \frac{b^\frac{-\gamma}{1+\alpha}b^\frac{-2}{1+\alpha}h_1(b,\Phi)h_2(\Phi)\partial_\varphi D(f)(\beta,\phi,b,\Phi)}{D(f)(\beta,\phi,b,\Phi)^\frac{2+\gamma}{2}}dbd\Phi\nonumber\\
=:&\mathcal{H}_1+\mathcal{H}_2,
\end{align}
where the first term was already studied in the previous steps. For the second term, let us compute the derivative of the denominator as follows:
\begin{align}\label{H-varphiD}
&\partial_\varphi D(f)(\beta,\phi,b,\Phi)=2\left(\beta^{\frac{-1}{1+\alpha}}\left\{\frac{\alpha-1}{1+\alpha}+\partial_{\overline{\beta}}f(\beta,\phi)\right\}^\frac12-b^{\frac{-1}{1+\alpha}}\left\{\frac{\alpha-1}{1+\alpha}+\partial_{\overline{b}}f(b,\Phi)\right\}^\frac12\right)\nonumber\\
&\times \beta^\frac{-1}{1+\alpha}\frac{\partial_\varphi\partial_{\overline{\beta}}f(\beta,\phi)}{2\left\{\frac{\alpha-1}{1+\alpha}+\partial_{\overline{\beta}}f(\beta,\phi)\right\}^\frac12}\nonumber\\
&+\beta^{\frac{-1}{1+\alpha}}b^{\frac{-1}{1+\alpha}}\left\{\frac{\alpha-1}{1+\alpha}+\partial_{\overline{\beta}}f(\beta,\phi)\right\}^\frac{-1}{2}\partial_\varphi \partial_{\overline{\beta}}f(\beta,\phi)\left\{\frac{\alpha-1}{1+\alpha}+\partial_{\overline{b}}f(b,\Phi)\right\}^\frac12(1-\cos(\phi+\beta-\Phi-b))\nonumber\\
&-\frac{2}{1+\alpha}\frac{1}{\beta}\beta^{\frac{-1}{1+\alpha}}b^{\frac{-1}{1+\alpha}}\left\{\frac{\alpha-1}{1+\alpha}+\partial_{\overline{\beta}}f(\beta,\phi)\right\}^\frac12\left\{\frac{\alpha-1}{1+\alpha}+\partial_{\overline{b}}f(b,\Phi)\right\}^\frac12(1-\cos(\phi+\beta-\Phi-b)),
\end{align}
implying that
\begin{align}\label{H-varphiD2}
|\partial_\varphi D(f)(\beta,\phi,b,\Phi)|\leq C D(f)(\beta,\phi,b,\Phi)^\frac12 \beta^{\frac{-1}{1+\alpha}}\frac{1}{\beta}.
\end{align}
In this way, we find that
\begin{align*}
&\left|\beta \partial_\varphi \mathcal{H}(h_1,h_2)(\beta,\phi)\right|\leq C|\mathcal{H}_f(h_1,h_2)(\beta,\phi)|\\
&+C\beta^{\frac{-1}{1+\alpha}}\beta^{-\frac{\alpha-1}{1+\alpha}}\int_0^{2\pi}\int_0^\infty \frac{b^{-\frac{\gamma}{1+\alpha}}b^\frac{-2}{1+\alpha}|h_1(b,\phi+\beta-\eta-b)||h_2(\phi+\beta-\eta-b)|}{D(f)(\beta,\phi,b,\phi+\beta-\eta-b)^\frac{1+\gamma}{2}}dbd\eta\\
\leq &C|\mathcal{H}_f(h_1,h_2)(\beta,\phi)|\\
&+C||h_1||_{X_1^{\sigma,\delta}}\beta^\sigma\beta^{\frac{-1}{1+\alpha}}\int_0^{2\pi}\int_0^\infty \frac{z^{-\frac{\gamma}{1+\alpha}}z^\frac{-2}{1+\alpha}b^\sigma/(1+\beta z)^{2\sigma}|h_2(\phi+\beta-\eta-b)|}{(\beta^\frac{-\gamma}{1+\alpha}+b^{\frac{-\gamma}{1+\alpha}})(\beta^\frac{-1}{1+\alpha}+b^{\frac{-1}{1+\alpha}})\left\{\frac{|\beta-b|^2}{|\beta+b|^2}+\sin^2(\eta/2)\right\}^{\frac{1+\gamma}{2}}}dbd\eta.
\end{align*}
Using \eqref{H-interp} we get
\begin{align*}
\left|\beta \partial_\varphi \mathcal{H}(h_1,h_2)(\beta,\phi)\right|\leq &C|\mathcal{H}_f(h_1,h_2)(\beta,\phi)|\\
&+C||h_1||_{\mathcal{X}^{\sigma,\delta}}||h_2||_{X_2^p}\beta^\sigma\int_0^\infty \frac{z^{-\frac{\gamma}{1+\alpha}}z^\frac{-2}{1+\alpha}b^\sigma/(1+\beta z)^{2\sigma}}{(\beta^\frac{-\gamma}{1+\alpha}+b^{\frac{-\gamma}{1+\alpha}})\frac{|\beta-b|^{1+\gamma-\nu}}{|\beta+b|^{1+\gamma-\nu}}}db.
\end{align*}

Then:
\begin{align}\label{h-est1}
||\beta \partial_\varphi \mathcal{H}(h_1,h_2)||_\sigma\leq C||h_1||_{\mathcal{X}^{\sigma,\delta}}||h_2||_{X_2^p},
\end{align}
provided that \eqref{sigma1} and \eqref{sigma2} hold, as well as \eqref{condition-p} and $1+\gamma-\nu<1$.

\noindent\medskip $\bullet$ $\beta \partial_\varphi \mathcal{H}_f(h_1,h_2)\in C^\delta_\beta$.

Since $\mathcal{H}_f(h_1,h_2)\in C^\delta_\beta$, the first term in \eqref{H-betavarphi} does. Hence, it remains to check the second term in \eqref{H-betavarphi}: $\mathcal{H}_2$. Take $\beta_1<\beta_2$, and then they satisfy \eqref{Holder-betas}. Let us write the difference of $\mathcal{H}_2$ in these two terms as follows
\begin{align*}
&\mathcal{H}_2(\beta_1)-\mathcal{H}_2(\beta_2)=C\beta_1^{\frac{2}{1+\alpha}}\int_0^{2\pi}\int_0^\infty \frac{b^\frac{-\gamma}{1+\alpha}b^\frac{-2}{1+\alpha}h_1(b,\Phi)h_2(\Phi)\partial_\varphi D(f)(\beta_1,\phi,b,\Phi)}{D(f)(\beta_1,\phi,b,\Phi)^\frac12}\\
&\times\left\{\frac{1}{D(f)(\beta_1,\phi,b,\Phi)^\frac{1+\gamma}{2}}-\frac{1}{D(f)(\beta_2,\phi,b,\Phi)^\frac{1+\gamma}{2}}\right\}dbd\Phi\\
&+C \beta_1^{\frac{2}{1+\alpha}}\int_0^{2\pi}\int_0^\infty \frac{b^\frac{-\gamma}{1+\alpha}b^\frac{-2}{1+\alpha}h_1(b,\Phi)h_2(\Phi)(\partial_\varphi D(f)(\beta_1,\phi,b,\Phi)-\partial_\varphi D(f)(\beta_2,\phi,b,\Phi))}{D(f)(\beta_1,\phi,b,\Phi)^\frac{1}{2}D(f)(\beta_2,\phi,b,\Phi)^\frac{1+\gamma}{2}}dbd\Phi\\
&+C\beta_1^{\frac{2}{1+\alpha}}\int_0^{2\pi}\int_0^\infty \frac{b^\frac{-\gamma}{1+\alpha}b^\frac{-2}{1+\alpha}h_1(b,\Phi)h_2(\Phi)\partial_\varphi D(f)(\beta_2,\phi,b,\Phi)}{D(f)(\beta_2,\phi,b,\Phi)^\frac{1+\gamma}{2}}\\
&\times\left\{\frac{1}{D(f)(\beta_1,\phi,b,\Phi)^\frac{1}{2}}-\frac{1}{D(f)(\beta_2,\phi,b,\Phi)^\frac{1}{2}}\right\}dbd\Phi\\
&+C(\beta_1^{\frac{2}{1+\alpha}}-\beta_2^{\frac{2}{1+\alpha}})\int_0^{2\pi}\int_0^\infty \frac{b^\frac{-\gamma}{1+\alpha}b^\frac{-2}{1+\alpha}h_1(b,\Phi)h_2(\Phi)\partial_\varphi D(f)(\beta_2,\phi,b,\Phi)}{D(f)(\beta_2,\phi,b,\Phi)^\frac{2+\gamma}{2}}dbd\Phi\\
=:&\mathcal{H}_{2,1}+\mathcal{H}_{2,2}+\mathcal{H}_{2,3}+\mathcal{H}_{2,4}.
\end{align*}
By using \eqref{Dfbetabeta2} and \eqref{H-varphiD2} we can bound $\mathcal{H}_{2,1}$ and $\mathcal{H}_{2,3}$. Let us show the estimate for $\mathcal{H}_{2,1}$. Indeed, from \eqref{H-holder2} we find for $\delta=1$ that
\begin{align*}
|D(f)(\beta_2,\phi,b,\Phi)-D(f)(\beta_1,\phi,b,\Phi)|\leq& C \frac{|\beta_1-\beta_2|}{|\beta_1+\beta_2|}(\beta_1^\frac{-1}{1+\alpha}+\beta_2^\frac{-1}{1+\alpha}+\beta_1^{\frac{-1}{1+\alpha}+1}+\beta_2^{\frac{-1}{1+\alpha}+1})\nonumber\\
&\times(D(f)(\beta_1,\phi,b,\Phi)^\frac12+D(f)(\beta_2,\phi,b,\Phi)^\frac12).
\end{align*}
By interpolating with
$$
|D(f)(\beta_2,\phi,b,\Phi)-D(f)(\beta_1,\phi,b,\Phi)|\leq|D(f)(\beta_2,\phi,b,\Phi)+D(f)(\beta_1,\phi,b,\Phi)|,
$$
we find
\begin{align}\label{Df-interp}
|D(f)(\beta_2,\phi,b,\Phi)-D(f)(\beta_1,\phi,b,\Phi)|\leq& C \frac{|\beta_1-\beta_2|^\delta}{|\beta_1+\beta_2|^\delta}(\beta_1^\frac{-\delta}{1+\alpha}+\beta_2^\frac{-\delta}{1+\alpha}+\beta_1^{\frac{-\delta}{1+\alpha}+\delta}+\beta_2^{\frac{-\delta}{1+\alpha}+\delta})\nonumber\\
&\times(D(f)(\beta_1,\phi,b,\Phi)^\frac{2-\delta}{2}+D(f)(\beta_2,\phi,b,\Phi)^\frac{2-\delta}{2}).
\end{align}
Hence using the ideas in \eqref{Dfbetabeta2} together with \eqref{Df-interp} we find
\begin{align}\label{Df-interp-2}
&\left|\frac{1}{D(f)(\beta_1,\phi,b,\Phi)^\frac{1+\gamma}{2}}-\frac{1}{D(f)(\beta_2,\phi,b,\Phi)^\frac{1+\gamma}{2}}\right|\leq C\frac{|\beta_1-\beta_2|^\delta}{|\beta_1+\beta_2|^\delta}\\
&\times\frac{\beta_1^{\frac{-\delta}{1+\alpha}}+\beta_2^{\frac{-\delta}{1+\alpha}}+\beta_1^{\frac{-\delta}{1+\alpha}+\delta}+\beta_2^{\frac{-\delta}{1+\alpha}+\delta}}{D(f)(\beta_1,\phi,b,\Phi)^\frac{\delta}{2}D(f)(\beta_2,\phi,b,\Phi)^\frac{1+\gamma}{2}+D(f)(\beta_2,\phi,b,\Phi)^\frac{\delta}{2}D(f)(\beta_1,\phi,b,\Phi)^\frac{1+\gamma}{2}}.
\end{align}
Hence, \eqref{Df-interp-2} together with \eqref{H-varphiD2} amounts to
{\footnotesize\begin{align*}
&|\mathcal{H}_{2,1}|\leq C\frac{|\beta_1-\beta_2|^\delta}{|\beta_1+\beta_2|^\delta}\beta_1^{\frac{2}{1+\alpha}}\beta_1^\frac{-1}{1+\alpha}\beta_1^{-1}\int_0^{2\pi}\int_0^\infty \frac{b^\frac{-\gamma}{1+\alpha}b^\frac{-2}{1+\alpha}|h_1(b,\Phi)||h_2(\Phi)|(\beta_1^{\frac{-\delta}{1+\alpha}}+\beta_1^{\frac{-\delta}{1+\alpha}+\delta})dbd\Phi}{D(f)(\beta_2,\phi,b,\Phi)^\frac{\delta}{2}D(f)(\beta_1,\phi,b,\Phi)^\frac{1+\gamma}{2}}\\
\leq &C\frac{|\beta_1-\beta_2|^\delta}{|\beta_1+\beta_2|^\delta}\beta_1^{\frac{2}{1+\alpha}}\beta_1^\frac{-1}{1+\alpha}\beta_1^{-1}\\
&\times \int_0^{2\pi}\int_0^\infty \frac{b^\frac{-\gamma}{1+\alpha}b^\frac{-2}{1+\alpha}|h_1(b,\Phi)||h_2(\Phi)|(\beta_1^{\frac{-\delta}{1+\alpha}}+\beta_1^{\frac{-\delta}{1+\alpha}+\delta})dbd\Phi}{|\beta_2^\frac{-\delta}{1+\alpha}+b^\frac{-\delta}{1+\alpha}|\left\{\frac{|\beta_2-b|^2}{|\beta_2+b|^2}+\sin^2((\beta_2-b+\phi-\Phi)/2)\right\}^\frac{\delta}{2}|\beta_1^\frac{-1-\gamma}{1+\alpha}+b^\frac{-1-\gamma}{1+\alpha}|\left\{\frac{|\beta_1-b|^2}{|\beta_1+b|^2}+\sin^2((\beta_1-b+\phi-\Phi)/2)\right\}^\frac{1+\gamma}{2}}\\
\end{align*}}
Using the ideas for \eqref{Hholder-2} we finally get
$$
|\mathcal{H}_{2,1}|\leq C\frac{|\beta_1-\beta_2|^\delta}{|\beta_1+\beta_2|^\delta}||h_1||_{L^\infty_\sigma}||h_2||_{L^p},
$$
for some constant $C$, provided that $\delta<1-\gamma$, $\delta<\sigma$ and $p>\frac{1}{1-\gamma}$.

Let us move to $\mathcal{H}_{2,3}$ and show the main idea. In order to work with that term, we should estimate the difference of $\partial_\varphi D(f)$ in two points. Note that
\begin{align*}
&\partial_\varphi D(f)(\beta,\phi,b,\Phi)=2\left(\beta^{\frac{-1}{1+\alpha}}\left\{\frac{\alpha-1}{1+\alpha}+\partial_{\overline{\beta}}f(\beta,\phi)\right\}^\frac12-b^{\frac{-1}{1+\alpha}}\left\{\frac{\alpha-1}{1+\alpha}+\partial_{\overline{b}}f(b,\Phi)\right\}^\frac12\right) \\
&\times \partial_\varphi\left[\beta^{\frac{-1}{1+\alpha}}\left\{\frac{\alpha-1}{1+\alpha}+\partial_{\overline{\beta}}f(\beta,\phi)\right\}^\frac12\right]\nonumber\\
&+2\partial_\varphi \left[\beta^{\frac{-1}{1+\alpha}}\left\{\frac{\alpha-1}{1+\alpha}+\partial_{\overline{\beta}}f(\beta,\phi)\right\}^\frac12\right] b^{\frac{-1}{1+\alpha}}\left\{\frac{\alpha-1}{1+\alpha}+\partial_{\overline{b}}f(b,\Phi)\right\}^\frac12(1-\cos(\phi+\beta-\Phi-b)).\nonumber
\end{align*}
By working with that term as for \eqref{H-holder2} we obtain
\begin{align}\label{Interpo-1}
|\partial_\varphi D(f)(\beta_1,\phi,b,\Phi)-\partial_\varphi D(f)(\beta_2,\phi,b,\Phi)|\leq C& \frac{|\beta_1-\beta_2|}{|\beta_1+\beta_2|}\\
&\times(\beta_1^{-1}\beta_1^\frac{-2}{1+\alpha}+\beta_2^{-1}\beta_2^\frac{-2}{1+\alpha}+\beta_1^{-1}\beta_1^{\frac{-2}{1+\alpha}+1}+\beta_2^{-1}\beta_2^{\frac{-2}{1+\alpha}+1})\nonumber,
\end{align}
which interpolating with \eqref{H-varphiD2} agrees with
\begin{align}\label{Interpo-2}
&|\partial_\varphi D(f)(\beta_1,\phi,b,\Phi)-\partial_\varphi D(f)(\beta_2,\phi,b,\Phi)|\leq C \frac{|\beta_1-\beta_2|^\delta}{|\beta_1+\beta_2|^\delta}\\
&\times(\beta_1^{-\delta}\beta_1^\frac{-2\delta}{1+\alpha}+\beta_2^{-\delta}\beta_2^\frac{-2\delta}{1+\alpha}+\beta_1^{-\delta}\beta_1^{\frac{-\delta}{1+\alpha}+\delta}+\beta_2^{-\delta}\beta_2^{\frac{-\delta}{1+\alpha}+\delta})\nonumber\\
&\times (\beta_1^{-\frac{1-\delta}{1+\alpha}}\beta_1^{-(1-\delta)}D(f)(\beta_1,\phi,b,\Phi)^\frac{1-\delta}{2}+\beta_2^{-\frac{1-\delta}{1+\alpha}}\beta_2^{-(1-\delta)}D(f)(\beta_2,\phi,b,\Phi)^\frac{1-\delta}{2}).\nonumber
\end{align}
Then, the key point to bound $\mathcal{H}_{2,2}$ is the use of \eqref{Interpo-2} instead of just \eqref{Interpo-1}.

Finally, the estimate for the last term $\mathcal{H}_{2,4}$ follows identically as for $\beta \partial_\varphi \mathcal{H}_f\in L^\infty_\sigma$ taking into account that
$$
|\beta_1^{\frac{2}{1+\alpha}}-\beta_2^{\frac{2}{1+\alpha}}|\leq C\frac{|\beta_1-\beta_2|}{|\beta_1+\beta_2|}(\beta_1^{\frac{2}{1+\alpha}}+\beta_2^{\frac{2}{1+\alpha}})\leq C\frac{|\beta_1-\beta_2|}{|\beta_1+\beta_2|}\beta_2^{\frac{2}{1+\alpha}}.
$$

The $C^1$ regularity of $F$ follows similarly by estimating the associated Fr\'echet derivatives.
\end{proof}

In the next proposition, we give the well-definition of the nonlinear function together with its $C^1$ regularity.

\begin{pro}\label{prop-well-def}
Let $\gamma,\sigma\in(0,1)$ and $\alpha\in(1,2)$ satisfying \eqref{sigma-rel}. There exists $\epsilon<1$ such that $\tilde{F}:B_{X_1^{\sigma,\delta}}(\epsilon)\times X_2^p\rightarrow Y^{\sigma,\delta}$ is well-defined and $C^1$, for any $p>\frac{1}{1-\gamma}$ and $\delta<\min\{1-\gamma,\sigma\}$.
\end{pro}
\begin{proof}

Let us prove that $\tilde{F}$ is well-defined and $C^1$ in $f$, where we shall use the expression of $\tilde{F}(f,\Omega)$ in \eqref{Ftilde}. Note that the first term in the expression of $\tilde{F}(f,\Omega)$ is just the constant $1$, which is not in $Y^{\sigma,\delta}$  since it does not decay appropriately at $0$ and $\infty$. For that reason, we need to add and subtract the contribution of the trivial solution in $\tilde{F}$ obtaining that
\begin{align*}
\tilde{F}(f,\Omega)(\beta,\phi)=f(\beta,\phi)-(-1)^\frac{-1}{2\mu}\frac{C_\gamma (\alpha-1)^\frac{\alpha-1}{2}}{4\pi C_0(1+\alpha)^{\frac{\alpha-1}{2}}}(I_1(f,\Omega)+I_2(f,\Omega)+I_3(f,\Omega)+I_4(f,\Omega))(\beta,\phi),
\end{align*}
where
\begin{align*}
&I_1(f,\Omega):=\beta^{-\frac{\alpha-1}{1+\alpha}}\int_0^{2\pi}\int_0^\infty\frac{b^{-\frac{\gamma}{1+\alpha}}b^{\frac{-2}{1+\alpha}}\left\{-\frac{\alpha-1}{1+\alpha}+\partial_{\overline{\varphi}}f(b,\Phi)\right\}^{-\frac{1}{2\mu}}\partial_{\overline{\beta\varphi}}f(b,\Phi)\Omega(\Phi)dbd\Phi}{D(f)(\beta,\phi,b,\Phi)^\frac{\gamma}{2}},\\
&I_2(f,\Omega):=\frac{2(\alpha-1)}{(1+\alpha)^2}\beta^{-\frac{\alpha-1}{1+\alpha}}\int_0^{2\pi}\int_0^\infty\frac{b^{-\frac{\gamma}{1+\alpha}}b^{\frac{-2}{1+\alpha}}\left[\left\{-\frac{\alpha-1}{1+\alpha}+\partial_{\overline{\varphi}}f(b,\Phi)\right\}^{-\frac{1}{2\mu}}-\left\{-\frac{\alpha-1}{1+\alpha}\right\}^{-\frac{1}{2\mu}}\right]\Omega(\Phi)dbd\Phi}{D(f)(\beta,\phi,b,\Phi)^\frac{\gamma}{2}},\\
&I_3(f,\Omega):=\frac{2(\alpha-1)}{(1+\alpha)^2}\left\{-\frac{\alpha-1}{1+\alpha}\right\}^{-\frac{1}{2\mu}}\beta^{-\frac{\alpha-1}{1+\alpha}}\int_0^{2\pi}\int_0^\infty\frac{b^{-\frac{\gamma}{1+\alpha}}b^{\frac{-2}{1+\alpha}}\Omega(\Phi)}{D(f)(\beta,\phi,b,\Phi)^\frac{\gamma}{2}D(0)(\beta,\phi,b,\Phi)^\frac{\gamma}{2}}\\
&\times\left\{D(0)(\beta,\phi,b,\Phi)^\frac{\gamma}{2}-D(f)(\beta,\phi,b,\Phi)^\frac{\gamma}{2}\right\}dbd\Phi,\\
&I_4(f,\Omega):=\frac{2(\alpha-1)}{(1+\alpha)^2}\left\{-\frac{\alpha-1}{1+\alpha}\right\}^{-\frac{1}{2\mu}}\beta^{-\frac{\alpha-1}{1+\alpha}}\int_0^{2\pi}\int_0^\infty\frac{b^{-\frac{\gamma}{1+\alpha}}b^{\frac{-2}{1+\alpha}}(\Omega(\Phi)-1)dbd\Phi}{D(0)(\beta,\phi,b,\Phi)^\frac{\gamma}{2}}.
\end{align*}
Note that now $f\in Y^{\sigma,\delta}$ and it remains to prove that $I_i\in Y^{\sigma,\delta}$, for $i=1,2,3,4$.

By using the operator $\mathcal{H}$ defined in Proposition \ref{prop-operator-H} we can write $I_1$, $I_2$ and $I_4$ as:
\begin{align*}
I_1(f,\Omega)=&\alpha_1\mathcal{H}\left(\left\{-\frac{\alpha-1}{1+\alpha}+\partial_{\overline{\varphi}}f\right\}^{-\frac{1}{2\mu}}\partial_{\overline{\beta\varphi}}f,\Omega\right),\\
I_2(f,\Omega)=&\alpha_2\mathcal{H}\left(\left[\left\{-\frac{\alpha-1}{1+\alpha}+\partial_{\overline{\varphi}}f\right\}^{-\frac{1}{2\mu}}-\left\{-\frac{\alpha-1}{1+\alpha}\right\}^{-\frac{1}{2\mu}}\right],\Omega\right),
\end{align*}
for some constants $\alpha_1, \alpha_2, \alpha_3$.
Then, by Proposition \ref{prop-operator-H} we find
\begin{align*}
&||I_1||_{Y^{\sigma,\delta}}+||I_2||_{Y^{\sigma,\delta}}\\
\leq& C||\Omega||_{X_2^p}\left\{\left\|\left\{-\frac{\alpha-1}{1+\alpha}+\partial_{\overline{\varphi}}f\right\}^{-\frac{1}{2\mu}}\partial_{\overline{\beta\varphi}}f\right\|_{\mathcal{X}_{\sigma,\alpha}}+\left\|\left[\left\{-\frac{\alpha-1}{1+\alpha}+\partial_{\overline{\varphi}}f\right\}^{-\frac{1}{2\mu}}-\left\{-\frac{\alpha-1}{1+\alpha}\right\}^{-\frac{1}{2\mu}}\right] \right\|_{\mathcal{X}_{\sigma,\alpha}}\right\}\\
\leq& C||\Omega||_{X_2^p}||f||_{X_1^{\sigma,\alpha}}.
\end{align*}

Moreover, Proposition \ref{prop-operator-H} gives us that $I_1$ and $I_2$  are $C^1$ in $f$. It remains to check the well-posedness and $C^1$ regularity of $I_3$ and $I_4$. Let us start with $I_4$.

Note that
\begin{align*}
I_4(f,\Omega)(\beta,\phi)=&C\beta^{-\frac{\alpha-1}{1+\alpha}}\int_0^{2\pi}\int_0^\infty\frac{b^{-\frac{\gamma}{1+\alpha}}b^{\frac{-2}{1+\alpha}}(\Omega(\Phi)-1)dbd\Phi}{\left\{\beta^\frac{-2}{1+\alpha}+b^\frac{-2}{1+\alpha}-2\beta^\frac{-1}{1+\alpha}b^\frac{-1}{1+\alpha}\cos(\phi+\beta-\Phi-b)\right\}^\frac{\gamma}{2}}\\
=&C\beta^{-\frac{\alpha-1}{1+\alpha}}\beta^\frac{\gamma}{1+\alpha}\int_0^{2\pi}\int_0^\infty\frac{b^{\frac{-2}{1+\alpha}}(\Omega(\Phi)-1)dbd\Phi}{\left\{\beta^\frac{2}{1+\alpha}+b^\frac{2}{1+\alpha}-2\beta^\frac{1}{1+\alpha}b^\frac{1}{1+\alpha}\cos(\phi+\beta-\Phi-b)\right\}^\frac{\gamma}{2}}\\
=&C\beta^{-\frac{\alpha-1}{1+\alpha}}\beta^\frac{\gamma}{1+\alpha} I_{4,1}(f,\Omega)(\beta,\phi),
\end{align*}
where $C$ is some constant which is not important in this proof. Let us first check the $L^\infty_\sigma$ norm. Take first $\beta\in[0,\beta_0]$. Since we have in the numerator $\Omega-1$, we get an extra cancellation at $\beta=0$. This comes from the expression of $\Omega-1$:
$$
\Omega(\phi)-1=\sum_{0\neq k\in \Z} g_k e^{ik\phi},
$$
where the mode $k=0$ vanishes implying that
$$
\int_0^{2\pi}(\Omega(\Phi)-1)d\Phi=0.
$$ 
In that way, we obtain that
\begin{align*}
I_{4,1}(f,\Omega)(0,\phi)=&\lim_{\varepsilon\rightarrow 0}\int_0^{2\pi}\int_\varepsilon^{1/\varepsilon}\frac{b^{\frac{-2}{1+\alpha}}(\Omega(\Phi)-1)dbd\Phi}{\left\{b^\frac{2}{1+\alpha}\right\}^\frac{\gamma}{2}}\\
=&\lim_{\varepsilon\rightarrow 0}\int_\varepsilon^{1/\varepsilon}\frac{b^{\frac{-2}{1+\alpha}}db}{\left\{b^\frac{2}{1+\alpha}\right\}^\frac{\gamma}{2}}\int_0^{2\pi}(\Omega(\Phi)-1)d\Phi\\
=&0.
\end{align*}
Then:
\begin{align*}
\sup_{\beta\in[0,\beta_0]}\frac{(1+\beta)^{2\sigma}}{\beta^\sigma}
I_4(f,\Omega)(\beta,\phi)\leq \sup_{\beta\in[0,\beta_0]}\frac{\beta^{-\frac{\alpha-1}{1+\alpha}}\beta^\frac{\gamma}{1+\alpha}}{\beta^\sigma}<C,
\end{align*}
since $\sigma<\frac{\gamma+1-\alpha}{1+\alpha}$ due to \eqref{sigma-rel}. Now take $\beta\in[\beta_0,\infty)$ and note that $I_4$ can be also written as
\begin{align}\label{I4}
I_4(f,\Omega)(\beta,\phi)=&C\beta^{-\frac{\alpha-1}{1+\alpha}}\int_0^{2\pi}\int_0^\infty\frac{b^{-\frac{\gamma}{1+\alpha}}b^{\frac{-2}{1+\alpha}}(\Omega(\Phi)-1)dbd\Phi}{\left\{\beta^\frac{-2}{1+\alpha}+b^\frac{-2}{1+\alpha}-2\beta^\frac{-1}{1+\alpha}b^\frac{-1}{1+\alpha}\cos(\phi+\beta-\Phi-b)\right\}^\frac{\gamma}{2}}\\
=&C\beta^{-\frac{\alpha-1}{1+\alpha}}I_{4,2}(f,\Omega)(\beta,\phi),\nonumber
\end{align}
where
$$
\lim_{\beta\rightarrow +\infty}I_{4,2}(f,\Omega)(\beta,\phi)=0,
$$
since $\Omega-1$ has a vanishing $0$ Fourier mode as before. Since $\sigma <\frac{\alpha-1}{1+\alpha}$ by \eqref{sigma-rel} we get
$$
\sup_{\beta\in[\beta_0,+\infty)}I_4(f,\Omega)(\beta,\phi)<C.
$$

Now make the change of variables $b=\beta z$ in the expression of $I_4$ to get
\begin{align}\label{I4-2}
I_4(f,\Omega)(\beta,\phi)=&C\int_0^{2\pi}\int_0^\infty\frac{z^{\frac{-2}{1+\alpha}}(\Omega(\Phi)-1)dbd\Phi}{\left\{1+z^\frac{2}{1+\alpha}-2z^\frac{1}{1+\alpha}\cos(\phi+\beta-\Phi-z\beta)\right\}^\frac{\gamma}{2}},
\end{align}
where we can prove easily the $C^\delta$ regularity in $\beta$ and the $\textnormal{Lip}$ regularity in $\phi$ using such expression and the ideas developed in Proposition \ref{prop-operator-H}. Taking $\partial_\varphi$ to \eqref{I4} we find
\begin{align*}
&\partial_\varphi I_4(f,\Omega)(\beta,\phi)=C\beta^{-1} I_4(f,\Omega)(\beta,\phi)\\
&+C\beta^{-1}\beta^{-\frac{\alpha-1}{1+\alpha}}\int_0^{2\pi}\int_0^\infty\frac{b^{-\frac{\gamma}{1+\alpha}}b^{\frac{-2}{1+\alpha}}(\Omega(\Phi)-1)\left\{\beta^{\frac{-2}{1+\alpha}}-\beta^\frac{-1}{1+\alpha}b^\frac{-1}{1+\alpha}\cos(\phi+\beta-\Phi-b)\right\}dbd\Phi}{\left\{\beta^\frac{-2}{1+\alpha}+b^\frac{-2}{1+\alpha}-2\beta^\frac{-1}{1+\alpha}b^\frac{-1}{1+\alpha}\cos(\phi+\beta-\Phi-b)\right\}^\frac{2+\gamma}{2}}.
\end{align*}
Let us check the $L^\infty_\sigma$ norm of $\beta \partial_\varphi I_4$. The first term is clearly in $L^\infty_\sigma$ and the second term can be written as
\begin{align*}
&\beta^{-\frac{\alpha-1}{1+\alpha}}\int_0^{2\pi}\int_0^\infty\frac{b^{-\frac{\gamma}{1+\alpha}}b^{\frac{-2}{1+\alpha}}(\Omega(\Phi)-1)\left\{\beta^{\frac{-2}{1+\alpha}}-\beta^\frac{-1}{1+\alpha}b^\frac{-1}{1+\alpha}\cos(\phi+\beta-\Phi-b)\right\}dbd\Phi}{\left\{\beta^\frac{-2}{1+\alpha}+b^\frac{-2}{1+\alpha}-2\beta^\frac{-1}{1+\alpha}b^\frac{-1}{1+\alpha}\cos(\phi+\beta-\Phi-b)\right\}^\frac{2+\gamma}{2}}\\
=&\beta^{-\frac{\alpha-1}{1+\alpha}} I_{4,3}(\beta,\phi)\\
=&\beta^{-\frac{\alpha-1}{1+\alpha}}\beta^\frac{\gamma}{1+\alpha} I_{4,4}(\beta,\phi)\\
\end{align*}
with $I_{4,4}(0,\phi)=0$ due to the structure of $\Omega-1$, and 
$$
\lim_{\beta\rightarrow +\infty}I_{4,3}(\beta,\phi)=0.
$$
Arguing as before for $I_4$ working with $I_{4,4}$ for $\beta\in[0,\beta_0]$ and with $I_{4,3}$ for $\beta\in[\beta_0,+\infty)$ we find that $\beta \partial_\varphi I_4\in L^\infty_\sigma$. In the same way we can prove the H\"older continuity of such term and the $C^1$ regularity in $f$.

Let us work with $I_3$. A priori, it can not be written in terms of $\mathcal{H}$, however it has the same scaling and the estimates works similarly. Let us just gave the first ideas to estimate it. Recall
\begin{align*}
I_3(f,\Omega)=&C\beta^{-\frac{\alpha-1}{1+\alpha}}\int_0^{2\pi}\int_0^\infty\frac{b^\frac{-\gamma}{1+\alpha}b^{\frac{-2}{1+\alpha}}\Omega(\Phi)}{D(f)(\beta,\phi,b,\Phi)^\frac{\gamma}{2}D(0)(\beta,\phi,b,\Phi)^\frac{\gamma}{2}}\\
&\times \left\{D(0)(\beta,\phi,b,\Phi)^\frac{\gamma}{2}-D(f)(\beta,\phi,b,\Phi)^\frac{\gamma}{2}\right\}dbd\Phi.
\end{align*}
Notice that
\begin{align*}
|D(0)(\beta,\phi,b,\Phi)^\frac{\gamma}{2}-D(f)(\beta,\phi,b,\Phi)^\frac{\gamma}{2}|\leq& C\frac{|D(0)(\beta,\phi,b,\Phi)-D(f)(\beta,\phi,b,\Phi)|}{|D(0)(\beta,\phi,b,\Phi)^{1-\frac{\gamma}{2}}+D(f)(\beta,\phi,b,\Phi)^{1-\frac{\gamma}{2}}|}.
\end{align*}
Via Proposition \ref{prop-denominator} and Remark \ref{rem-denominator}, we have that $D(f)$ and $D(0)$ are comparable: $D(f)(\beta,\phi,b,\Phi)\geq D(0)(\beta,\phi,b,\Phi)$ and then
\begin{equation}\label{I3-exp-2}
|I_3(f,\Omega)|\leq C\beta^{-\frac{\alpha-1}{1+\alpha}}\int_0^{2\pi}\int_0^\infty\frac{b^\frac{-\gamma}{1+\alpha}b^{\frac{-2}{1+\alpha}}|\Omega(\Phi)|}{D(f)(\beta,\phi,b,\Phi)^\frac{\gamma}{2}}\frac{|D(0)(\beta,\phi,b,\Phi)-D(f)(\beta,\phi,b,\Phi)|}{D(0)(\beta,\phi,b,\Phi)}dbd\Phi.
\end{equation}

Moreover, using the expression of $D(f)$ in \eqref{exp-denominator-2} we find
\begin{align*}
&|D(0)(\beta,\phi,b,\Phi)-D(f)(\beta,\phi,b,\Phi)|\leq C|D(0)(\beta,\phi,b,\Phi)^\frac12+D(f)(\beta,\phi,b,\Phi)^\frac12| \\
&\times \Bigg[ \left|\beta^\frac{-1}{1+\alpha}\left\{\frac{\alpha-1}{1+\alpha}+\partial_{\overline{\beta}}f(\beta,\phi)\right\}^\frac12-\beta^\frac{-1}{1+\alpha}\left\{\frac{\alpha-1}{1+\alpha}\right\}^\frac12\right|\\
&+\left|b^\frac{-1}{1+\alpha}\left\{\frac{\alpha-1}{1+\alpha}+\partial_{\overline{b}}f(b,\Phi)\right\}^\frac12-b^\frac{-1}{1+\alpha}\left\{\frac{\alpha-1}{1+\alpha}\right\}^\frac12\right|\\
&+\beta^\frac{-1}{1+\alpha} \left|\left\{\frac{\alpha-1}{1+\alpha}+\partial_{\overline{\beta}}f(\beta,\phi)\right\}^\frac12-\left\{\frac{\alpha-1}{1+\alpha}\right\}^\frac12\right| \left|\frac{\alpha-1}{1+\alpha}+\partial_{\overline{b}}f(b,\phi)\right|^\frac12||\sin((\beta+\phi-b-\Phi)/2)|\\
&+b^\frac{-1}{1+\alpha} \left|\frac{\alpha-1}{1+\alpha}\right|^\frac12\left| \left\{\frac{\alpha-1}{1+\alpha}+\partial_{\overline{b}}f(b,\Phi)\right\}^\frac12-\left\{\frac{\alpha-1}{1+\alpha}\right\}^\frac12\right||\sin((\beta+\phi-b-\Phi)/2)|\Bigg].
\end{align*}
Since $\partial_b f(\beta,\Phi)$ is bounded in $L^\infty$ we get 
\begin{align*}
&|D(0)(\beta,\phi,b,\Phi)-D(f)(\beta,\phi,b,\Phi)|\leq C|D(0)(\beta,\phi,b,\Phi)^\frac12+D(f)(\beta,\phi,b,\Phi)^\frac12| \\
&\times \Bigg[ \beta^\frac{-1}{1+\alpha} \left|\left\{\frac{\alpha-1}{1+\alpha}+\partial_{\overline{\beta}}f(\beta,\phi)\right\}^\frac12-\left\{\frac{\alpha-1}{1+\alpha}\right\}^\frac12\right|+b^\frac{-1}{1+\alpha}\left| \left\{\frac{\alpha-1}{1+\alpha}+\partial_{\overline{b}}f(b,\Phi)\right\}^\frac12-\left\{\frac{\alpha-1}{1+\alpha}\right\}^\frac12\right|\Bigg].
\end{align*}
Now, we wish to take into account the bound for $f$ and hence using that the norm of $f$ is small, we can write it as
\begin{align*}
|D(0)(\beta,\phi,b,\Phi)-D(f)(\beta,\phi,b,\Phi)|\leq& C|D(0)(\beta,\phi,b,\Phi)^\frac12+D(f)(\beta,\phi,b,\Phi)^\frac12| \\
&\times \Bigg[ \beta^\frac{-1}{1+\alpha} \left|\partial_{\overline{\beta}}f(\beta,\phi)\right|+b^\frac{-1}{1+\alpha}\left|\partial_{\overline{b}}f(b,\Phi)\right|\Bigg].
\end{align*}

In this way, we separate $I_3$ in two integrals and use again Remark \ref{rem-denominator} to compare $D(f)$ with $D(0)$:
\begin{align*}
|I_3(f,\Omega)|\leq& C\beta^{-\frac{\alpha-1}{1+\alpha}}\int_0^{2\pi}\int_0^\infty\frac{b^\frac{-\gamma}{1+\alpha}b^{\frac{-2}{1+\alpha}}|\Omega(\Phi)|}{D(f)(\beta,\phi,b,\Phi)^\frac{\gamma}{2}}\frac{\beta^{\frac{-1}{1+\alpha}}|\partial_{\overline{\beta}}f(\beta,\phi)|+b^\frac{-1}{1+\alpha}|\partial_{\overline{b}}f(b,\Phi)|}{D(0)(\beta,\phi,b,\Phi)^\frac12}dbd\Phi.
\end{align*}

By using similar arguments than in Proposition \ref{prop-operator-H}, we finally achieve that
$$
||I_{3}||_{L^\infty_\sigma}\leq C||\Omega||_{L^p}||\partial_{\overline{\beta}}f||_{L^\infty_\sigma},
$$
for any $p>1$.

By working with $D(f)$ and using the ideas of Proposition \ref{prop-operator-H}, we can check that the others norms in $X_1^{\sigma,\delta}$ are also bounded.
\end{proof}

\section{Linearized operator}\label{sec-linearized}

This section aims to give different expression of the linear operator of $\tilde{F}$ at the trivial solution $(0,1)$. First, we will show that there is an extra cancellation that help us to prove that the linear operator is a compact perturbation of a isomorphism. Hence, we can conclude that the linearized operator is Fredholm of zero index, which will be useful in checking that it is a isomorphism in the following section. Moreover, at the end of this section, we give its expression in Fourier series.

\subsection{Linearization at the equilibrium}

First, let us give the expression for the linear operator of $\tilde{F}$ at $(0,1)$.
\begin{pro}\label{prop-lin-op-1}
The linearized operator of $\tilde{F}$ at the trivial solution $(0,1)$ reads as follows:
\begin{align}\label{L-I}
\nonumber\partial_f\tilde{F}(0,1)h(\beta,\phi)=&\frac{(1+\alpha)}{2}\partial_\beta(\beta h(\beta,\phi))\\
\nonumber&-\frac{C_\gamma }{4\pi C_0}\frac{1+\alpha}{\alpha-1}\left(\frac{\alpha-1}{1+\alpha}\right)^\frac{\gamma}{2}\beta^{-\frac{\alpha-1}{1+\alpha}}\int_0^{2\pi}\int_0^\infty\frac{\beta^{-\frac{\gamma}{1+\alpha}}\beta^{\frac{-2}{1+\alpha}}\partial_{\overline{b\varPsi}}h(b,\Phi)dbd\Phi}{D(0)(\beta,\phi,b,\Phi)^\frac{\gamma}{2}},\\
\nonumber&-\frac{1}{\mu}\frac{C_\gamma }{4\pi C_0}\frac{1}{\alpha-1}\left(\frac{\alpha-1}{1+\alpha}\right)^\frac{\gamma}{2}\beta^{-\frac{\alpha-1}{1+\alpha}}\int_0^{2\pi}\int_0^\infty\frac{\beta^{-\frac{\gamma}{1+\alpha}}\beta^{\frac{-2}{1+\alpha}}\partial_{\overline{\varPsi}}h(b,\Phi)dbd\Phi}{D(0)(\beta,\phi,b,\Phi)^\frac{\gamma}{2}},\\
\nonumber&+\frac{\gamma}{2}\frac{C_\gamma }{2\pi C_0}\frac{1}{1+\alpha}\left(\frac{\alpha-1}{1+\alpha}\right)^\frac{\gamma}{2}\beta^{-\frac{\alpha-1}{1+\alpha}}\int_0^{2\pi}\int_0^\infty\frac{\beta^{-\frac{\gamma}{1+\alpha}}\beta^{\frac{-2}{1+\alpha}}}{D(0)(\beta,\phi,b,\Phi)^{\frac{\gamma}{2}+1}}\\
\nonumber&\times \left\{b^{\frac{-2}{1+\alpha}}-\beta^{\frac{-1}{1+\alpha}}b^{\frac{-1}{1+\alpha}}\cos(\phi+\beta-\Phi-b)\right\}\partial_{\overline{b}}h(b,\Phi)dbd\Phi\\
=:&\frac{(1+\alpha)}{2}\partial_\beta(\beta h(\beta,\phi))+I_1+I_2+I_3.
\end{align}
\end{pro}
\begin{proof}
Note that
\begin{align*}
&\partial_f\tilde{F}(0,1)h(\beta,\phi)=h(\beta,\phi)-L(h)(\beta,\phi)\\
&-\frac{C_\gamma (\alpha-1)^\frac{\alpha-1}{2}}{4\pi C_0(1+\alpha)^{\frac{\alpha-1}{2}}}\left\{\frac{\alpha-1}{1+\alpha}\right\}^{-\frac{1}{2\mu}}\beta^{-\frac{\alpha-1}{1+\alpha}}\int_0^{2\pi}\int_0^\infty\frac{\beta^{-\frac{\gamma}{1+\alpha}}\beta^{\frac{-2}{1+\alpha}}\partial_{\overline{b\varPsi}}h(b,\Phi)dbd\Phi}{D(0)(\beta,\phi,b,\Phi)^\frac{\gamma}{2}},\\
&-\frac{1}{2\mu}\frac{C_\gamma (\alpha-1)^\frac{\alpha-1}{2}}{4\pi C_0(1+\alpha)^{\frac{\alpha-1}{2}}}\left\{\frac{\alpha-1}{1+\alpha}\right\}^{-\frac{1}{2\mu}-1}\left\{\frac{2(\alpha-1)}{(1+\alpha)^2}\right\}\beta^{-\frac{\alpha-1}{1+\alpha}}\int_0^{2\pi}\int_0^\infty\frac{\beta^{-\frac{\gamma}{1+\alpha}}\beta^{\frac{-2}{1+\alpha}}\partial_{\overline{\varPsi}}h(b,\Phi)dbd\Phi}{D(0)(\beta,\phi,b,\Phi)^\frac{\gamma}{2}},\\
&+\frac{\gamma}{2}\frac{C_\gamma (\alpha-1)^\frac{\alpha-1}{2}}{4\pi C_0(1+\alpha)^{\frac{\alpha-1}{2}}}\left\{\frac{\alpha-1}{1+\alpha}\right\}^{-\frac{1}{2\mu}}\left\{\frac{2(\alpha-1)}{(1+\alpha)^2}\right\}\beta^{-\frac{\alpha-1}{1+\alpha}}\int_0^{2\pi}\int_0^\infty\frac{\beta^{-\frac{\gamma}{1+\alpha}}\beta^{\frac{-2}{1+\alpha}}}{D(0)(\beta,\phi,b,\Phi)^{\frac{\gamma}{2}+1}}\\
&\times \left\{b^{\frac{-2}{1+\alpha}}-\beta^{\frac{-1}{1+\alpha}}b^{\frac{-1}{1+\alpha}}\cos(\phi+\beta-\Phi-b)\right\}\partial_{\overline{b}}h(b,\Phi)dbd\Phi
\end{align*}

Let us explain the meaning of $L$. The function $L$ means the linearization of the local part associated to $-(-\Delta)^{-1+\frac{\gamma}{2}}\hat{\theta}$ after performing the adapted coordinates and the appropriate scaling from $\Psi$ to $f$. That is, take the initial stream function (which is a radial function):
$$
\hat{\psi}^0(re^{i\theta})=-C_0 r^{1-\alpha},
$$
and now we use the fact that $r=(1+\alpha)^\frac{1}{2}(-\Psi_\beta)^\frac{1}{2}$, where we recall
$$
\Psi=-C_0^\frac{2}{1+\alpha}(\alpha-1)^\frac{1-\alpha}{1+\alpha}\beta^\frac{\alpha-1}{1+\alpha}(1+f).
$$
Hence, after the scaling done from $F$ to $\tilde{F}$ in Section \ref{sec-scaling}, we achieve
\begin{align*}
L(h)=&\beta^{-\frac{\alpha-1}{1+\alpha}}\left\{-C_0^\frac{-2}{1+\alpha}(\alpha-1)^\frac{\alpha-1}{1+\alpha}\right\}\partial_f \left[-C_0 (1+\alpha)^\frac{1-\alpha}{2}(-\Psi_\beta)^\frac{1-\alpha}{2}\right]h\\
=&-\beta^{-\frac{\alpha-1}{1+\alpha}}\left\{-C_0^\frac{-2}{1+\alpha}(\alpha-1)^\frac{\alpha-1}{1+\alpha}\right\}C_0 (1+\alpha)^\frac{1-\alpha}{2} \frac{1-\alpha}{2}\left[(-\Psi_\beta^0)^{\frac{1-\alpha}{2}-1}\right]C_0^\frac{2}{1+\alpha}(\alpha-1)^\frac{1-\alpha}{1+\alpha}\partial_\beta (\beta^\frac{\alpha-1}{1+\alpha}h).
\end{align*}
Using that $\partial_\beta(\beta^\frac{\alpha-1}{1+\alpha}h)=\beta^\frac{-2}{1+\alpha}\partial_{\overline{\beta}}h$ together with
$$
(-\Psi_\beta^0)(\beta)=C_0^\frac{2}{1+\alpha}\frac{(\alpha-1)^\frac{2}{1+\alpha}}{1+\alpha}\beta^\frac{-2}{1+\alpha},
$$
we find
\begin{align*}
L(h)=&-\beta^{-\frac{\alpha-1}{1+\alpha}}\left\{-C_0^\frac{-2}{1+\alpha}(\alpha-1)^\frac{\alpha-1}{1+\alpha}\right\}C_0 (1+\alpha)^\frac{1-\alpha}{2} \frac{1-\alpha}{2} C_0^\frac{1-\alpha}{1+\alpha}\frac{(\alpha-1)^\frac{1-\alpha}{1+\alpha}}{(1+\alpha)^\frac{1-\alpha}{2}}\beta^\frac{\alpha-1}{1+\alpha}   C_0^\frac{-2}{1+\alpha}\frac{(\alpha-1)^\frac{-2}{1+\alpha}}{(1+\alpha)^{-1}}\beta^\frac{2}{1+\alpha}\\
&\times  C_0^\frac{2}{1+\alpha}(\alpha-1)^\frac{1-\alpha}{1+\alpha}  \beta^\frac{-2}{1+\alpha}\partial_{\overline{\beta}}h\\
=& -    \frac{(1+\alpha)}{2} \partial_{\overline{\beta}}h\\
\end{align*}
Using now the expression of $\partial_{\overline{\beta}}h$ we finally get 
$$
L(h)=-\frac{\alpha-1}{2}h-\frac{(1+\alpha)}{2}\beta h_\beta
$$
Hence, we find
\begin{align*}
\partial_f\tilde{F}(0,1)h(\beta,\phi)=&\frac{(1+\alpha)}{2}\partial_\beta(\beta h(\beta,\phi))\\
&-\frac{C_\gamma }{4\pi C_0}\frac{1+\alpha}{\alpha-1}\left(\frac{\alpha-1}{1+\alpha}\right)^\frac{\gamma}{2}\beta^{-\frac{\alpha-1}{1+\alpha}}\int_0^{2\pi}\int_0^\infty\frac{\beta^{-\frac{\gamma}{1+\alpha}}\beta^{\frac{-2}{1+\alpha}}\partial_{\overline{b\varPsi}}h(b,\Phi)dbd\Phi}{D(0)(\beta,\phi,b,\Phi)^\frac{\gamma}{2}},\\
&-\frac{1}{\mu}\frac{C_\gamma }{4\pi C_0}\frac{1}{\alpha-1}\left(\frac{\alpha-1}{1+\alpha}\right)^\frac{\gamma}{2}\beta^{-\frac{\alpha-1}{1+\alpha}}\int_0^{2\pi}\int_0^\infty\frac{\beta^{-\frac{\gamma}{1+\alpha}}\beta^{\frac{-2}{1+\alpha}}\partial_{\overline{\varPsi}}h(b,\Phi)dbd\Phi}{D(0)(\beta,\phi,b,\Phi)^\frac{\gamma}{2}},\\
&+\frac{\gamma}{2}\frac{C_\gamma }{2\pi C_0}\frac{1}{1+\alpha}\left(\frac{\alpha-1}{1+\alpha}\right)^\frac{\gamma}{2}\beta^{-\frac{\alpha-1}{1+\alpha}}\int_0^{2\pi}\int_0^\infty\frac{\beta^{-\frac{\gamma}{1+\alpha}}\beta^{\frac{-2}{1+\alpha}}}{D(0)(\beta,\phi,b,\Phi)^{\frac{\gamma}{2}+1}}\\
&\times \left\{b^{\frac{-2}{1+\alpha}}-\beta^{\frac{-1}{1+\alpha}}b^{\frac{-1}{1+\alpha}}\cos(\phi+\beta-\Phi-b)\right\}\partial_{\overline{b}}h(b,\Phi)dbd\Phi.
\end{align*}
\end{proof}

\begin{rem}
The denominator $D(0)$ agrees with
\begin{align*}
D(0)(\beta,\phi,b,\Phi)=\left(\frac{\alpha-1}{1+\alpha}\right)\left\{(\beta^{\frac{-1}{1+\alpha}}-b^{\frac{-1}{1+\alpha}})^2+2\beta^{\frac{-1}{1+\alpha}}b^{\frac{-1}{1+\alpha}}(1-\cos(\phi+\beta-\Phi-b)\right\}.
\end{align*}

\end{rem}

In the following proposition, we show an extra cancellation of the linearized operator amounting to a more compact expression of it.

\begin{pro}\label{prop-lin-op-cancelation-1}
The linearized operator in Proposition \ref{prop-lin-op-1} can be written as
\begin{align}\label{partialfF-1-2}
\nonumber&\partial_f\tilde{F}(0,1)h(\beta,\phi)=\frac{(1+\alpha)}{2}\partial_\beta(\beta h(\beta,\phi))\\
&-\frac{1}{\mu}\frac{C_\gamma }{4\pi C_0}\left(\frac{1}{\alpha-1}\right)\beta^{-\frac{\alpha-1}{1+\alpha}}\int_0^{2\pi}\int_0^\infty \frac{ b^{-\frac{\gamma}{1+\alpha}}b^{\frac{-2}{1+\alpha}}(\partial_{\overline{\varPsi}}h(b,\Phi)+\partial_{\overline{b}}h(b,\Phi))dbd\Phi}{\{\beta^{\frac{-2}{1+\alpha}}+b^{\frac{-2}{1+\alpha}}-2\beta^{\frac{-1}{1+\alpha}}b^{\frac{-1}{1+\alpha}}\cos(\phi+\beta-\Phi-b)\}^\frac{\gamma}{2}}.
\end{align}
\end{pro}
\begin{proof}
To obtain such expression, we will do an integration by parts in the last term of $I$ in \eqref{L-I}. 

First, note that using $\partial_\varPsi=\partial_\Phi-\partial_b$ we get
\begin{align*}
\partial_\varPsi\frac{1}{D(0)(\beta,\phi,b,\Phi)^\frac{\gamma}{2}}=&-\frac{\gamma}{2}\frac{2}{1+\alpha}\left(\frac{\alpha-1}{1+\alpha}\right)b^{-1}\frac{b^{\frac{-2}{1+\alpha}}-b^\frac{-1}{1+\alpha}\beta^\frac{-1}{1+\alpha}\cos(\phi+\beta-\Phi-b)}{D(0)(\beta,\phi,b,\Phi)^{\frac{\gamma}{2}+1}}.
\end{align*}
Then
\begin{align*}
I_3=&-\frac{C_\gamma }{4\pi C_0}\left(\frac{1+\alpha}{\alpha-1}\right)\left(\frac{\alpha-1}{1+\alpha}\right)^\frac{\gamma}{2}\beta^{-\frac{\alpha-1}{1+\alpha}}\int_0^{2\pi}\int_0^\infty b^{-\frac{\gamma}{1+\alpha}}b^{\frac{-2}{1+\alpha}}b\partial_{\overline{b}}h(b,\Phi)\partial_\varPsi\frac{1}{D(0)(\beta,\phi,b,\Phi)^{\frac{\gamma}{2}}}dbd\Phi\\
=&-\frac{C_\gamma }{4\pi C_0}\left(\frac{1+\alpha}{\alpha-1}\right)\left(\frac{\alpha-1}{1+\alpha}\right)^\frac{\gamma}{2}\beta^{-\frac{\alpha-1}{1+\alpha}}\int_0^{2\pi}\int_0^\infty b^{-\frac{\gamma}{1+\alpha}}b^{\frac{-2}{1+\alpha}}b\partial_{\overline{b}}h(b,\Phi)\partial_\Phi\frac{1}{D(0)(\beta,\phi,b,\Phi)^{\frac{\gamma}{2}}}dbd\Phi\\
&+\frac{C_\gamma }{4\pi C_0}\left(\frac{1+\alpha}{\alpha-1}\right)\left(\frac{\alpha-1}{1+\alpha}\right)^\frac{\gamma}{2}\beta^{-\frac{\alpha-1}{1+\alpha}}\int_0^{2\pi}\int_0^\infty b^{-\frac{\gamma}{1+\alpha}}b^{\frac{-2}{1+\alpha}}b\partial_{\overline{b}}h(b,\Phi)\partial_b\frac{1}{D(0)(\beta,\phi,b,\Phi)^{\frac{\gamma}{2}}}dbd\Phi.
\end{align*}
Taking into account the definition of $\varPsi$, we can integrate by parts there

Let us start with $\partial_\Phi$. Note that
\begin{align*}
\int_0^{2\pi}&b^{-\frac{\gamma}{1+\alpha}}b^{\frac{-2}{1+\alpha}}b\partial_{\overline{b}}h(b,\Phi)\partial_\Phi \frac{1}{\{\beta^{\frac{-2}{1+\alpha}}+b^{\frac{-2}{1+\alpha}}-2\beta^{\frac{-1}{1+\alpha}}b^{\frac{-1}{1+\alpha}}\cos(\phi+\beta-\Phi-b)\}^\frac{\gamma}{2}}d\Phi\\
&=-\int_0^{2\pi}\frac{b^{-\frac{\gamma}{1+\alpha}}b^{\frac{-2}{1+\alpha}}b\partial_\Phi\partial_{\overline{b}}h(b,\Phi)}{\{\beta^{\frac{-2}{1+\alpha}}+b^{\frac{-2}{1+\alpha}}-2\beta^{\frac{-1}{1+\alpha}}b^{\frac{-1}{1+\alpha}}\cos(\phi+\beta-\Phi-b)\}^\frac{\gamma}{2}}d\Phi,
\end{align*}
since the boundary terms vanish due to the periodicity of the function. However, the boundary terms with respect to the integration in $b$ will depend on the choice of the parameters:
\begin{align*}
\int_0^\infty& b^{-\frac{\gamma}{1+\alpha}}b^{\frac{-2}{1+\alpha}}b\partial_{\overline{b}}h(b,\Phi)\partial_b \frac{1}{\{\beta^{\frac{-2}{1+\alpha}}+b^{\frac{-2}{1+\alpha}}-2\beta^{\frac{-1}{1+\alpha}}b^{\frac{-1}{1+\alpha}}\cos(\phi+\beta-\Phi-b)\}^\frac{\gamma}{2}}db\\
=&\lim_{b\rightarrow +\infty} \frac{b^{-\frac{\gamma}{1+\alpha}}b^{\frac{-2}{1+\alpha}}b\partial_{\overline{b}}h(b,\Phi)}{\{\beta^{\frac{-2}{1+\alpha}}+b^{\frac{-2}{1+\alpha}}-2\beta^{\frac{-1}{1+\alpha}}b^{\frac{-1}{1+\alpha}}\cos(\phi+\beta-\Phi-b)\}^\frac{\gamma}{2}}\\
&-\lim_{b\rightarrow 0} \frac{b^{-\frac{\gamma}{1+\alpha}}b^{\frac{-2}{1+\alpha}}b\partial_{\overline{b}}h(b,\Phi)}{\{\beta^{\frac{-2}{1+\alpha}}+b^{\frac{-2}{1+\alpha}}-2\beta^{\frac{-1}{1+\alpha}}b^{\frac{-1}{1+\alpha}}\cos(\phi+\beta-\Phi-b)\}^\frac{\gamma}{2}}\\
&-\int_0^\infty \frac{\frac{-\gamma-1+\alpha}{1+\alpha} b^{-\frac{\gamma}{1+\alpha}}b^{\frac{-2}{1+\alpha}}\partial_{\overline{b}}h(b,\Phi)+b^{-\frac{\gamma}{1+\alpha}}b^{\frac{-2}{1+\alpha}}b\partial_b\partial_{\overline{b}}h(b,\Phi)}{\{\beta^{\frac{-2}{1+\alpha}}+b^{\frac{-2}{1+\alpha}}-2\beta^{\frac{-1}{1+\alpha}}b^{\frac{-1}{1+\alpha}}\cos(\phi+\beta-\Phi-b)\}^\frac{\gamma}{2}}db.
\end{align*}
Since $\partial_{\overline{b}} h\in\mathcal{X}^{\sigma,\delta}$, we have that
\begin{align*}
\lim_{b\rightarrow +\infty} \frac{b^{-\frac{\gamma}{1+\alpha}}b^{\frac{-2}{1+\alpha}}b\partial_{\overline{b}}h(b,\Phi)}{\{\beta^{\frac{-2}{1+\alpha}}+b^{\frac{-2}{1+\alpha}}-2\beta^{\frac{-1}{1+\alpha}}b^{\frac{-1}{1+\alpha}}\cos(\phi+\beta-\Phi-b)\}^\frac{\gamma}{2}}=\lim_{b\rightarrow +\infty} \beta^\frac{\gamma}{1+\alpha}b^\frac{-\gamma}{1+\alpha}b^\frac{\alpha-1}{1+\alpha}b^{-\sigma}=0,
\end{align*}
since $\gamma>\alpha-1$ by \eqref{sigma-rel}. On the other hand
\begin{align*}
\lim_{b\rightarrow 0} \frac{b^{-\frac{\gamma}{1+\alpha}}b^{\frac{-2}{1+\alpha}}b\partial_{\overline{b}}h(b,\Phi)}{\{\beta^{\frac{-2}{1+\alpha}}+b^{\frac{-2}{1+\alpha}}-2\beta^{\frac{-1}{1+\alpha}}b^{\frac{-1}{1+\alpha}}\cos(\phi+\beta-\Phi-b)\}^\frac{\gamma}{2}}=\lim_{b\rightarrow 0} b^\frac{\gamma}{1+\alpha}b^\frac{-\gamma}{1+\alpha}b^\frac{\alpha-1}{1+\alpha}b^{\sigma}=0,
\end{align*}
since $\alpha>1$. Hence, in this term the boundary terms also vanish obtaining
\begin{align*}
\int_0^\infty& b^{-\frac{\gamma}{1+\alpha}}b^{\frac{-2}{1+\alpha}}b\partial_{\overline{b}}h(b,\Phi)\partial_b \frac{1}{\{\beta^{\frac{-2}{1+\alpha}}+b^{\frac{-2}{1+\alpha}}-2\beta^{\frac{-1}{1+\alpha}}b^{\frac{-1}{1+\alpha}}\cos(\phi+\beta-\Phi-b)\}^\frac{\gamma}{2}}db\\
=&-\int_0^\infty \frac{\frac{-\gamma-2+\alpha}{1+\alpha} b^{-\frac{\gamma}{1+\alpha}}b^{\frac{-2}{1+\alpha}}\partial_{\overline{b}}h(b,\Phi)+b^{-\frac{\gamma}{1+\alpha}}b^{\frac{-2}{1+\alpha}}b\partial_b\partial_{\overline{b}}h(b,\Phi)}{\{\beta^{\frac{-2}{1+\alpha}}+b^{\frac{-2}{1+\alpha}}-2\beta^{\frac{-1}{1+\alpha}}b^{\frac{-1}{1+\alpha}}\cos(\phi+\beta-\Phi-b)\}^\frac{\gamma}{2}}db.
\end{align*}
By putting everything together we find
\begin{align*}
\mathcal{I}_3=&\frac{C_\gamma }{4\pi C_0}\left(\frac{1+\alpha}{\alpha-1}\right)\beta^{-\frac{\alpha-1}{1+\alpha}}\int_0^{2\pi}\int_0^\infty\frac{b^{-\frac{\gamma}{1+\alpha}}b^{\frac{-2}{1+\alpha}}b\partial_\Phi\partial_{\overline{b}}h(b,\Phi)}{\{\beta^{\frac{-2}{1+\alpha}}+b^{\frac{-2}{1+\alpha}}-2\beta^{\frac{-1}{1+\alpha}}b^{\frac{-1}{1+\alpha}}\cos(\phi+\beta-\Phi-b)\}^\frac{\gamma}{2}}dbd\Phi,\\
&-\frac{C_\gamma }{4\pi C_0}\left(\frac{1+\alpha}{\alpha-1}\right)\beta^{-\frac{\alpha-1}{1+\alpha}}\int_0^{2\pi}\int_0^\infty \frac{\frac{-\gamma-2+\alpha}{1+\alpha} b^{-\frac{\gamma}{1+\alpha}}b^{\frac{-2}{1+\alpha}}\partial_{\overline{b}}h(b,\Phi)+b^{-\frac{\gamma}{1+\alpha}}b^{\frac{-2}{1+\alpha}}b\partial_b\partial_{\overline{b}}h(b,\Phi)}{\{\beta^{\frac{-2}{1+\alpha}}+b^{\frac{-2}{1+\alpha}}-2\beta^{\frac{-1}{1+\alpha}}b^{\frac{-1}{1+\alpha}}\cos(\phi+\beta-\Phi-b)\}^\frac{\gamma}{2}}dbd\Phi\\
=&\frac{C_\gamma }{4\pi C_0}\left(\frac{1+\alpha}{\alpha-1}\right)\beta^{-\frac{\alpha-1}{1+\alpha}}\int_0^{2\pi}\int_0^\infty\frac{b^{-\frac{\gamma}{1+\alpha}}b^{\frac{-2}{1+\alpha}}b\partial_\varPsi\partial_{\overline{b}}h(b,\Phi)}{\{\beta^{\frac{-2}{1+\alpha}}+b^{\frac{-2}{1+\alpha}}-2\beta^{\frac{-1}{1+\alpha}}b^{\frac{-1}{1+\alpha}}\cos(\phi+\beta-\Phi-b)\}^\frac{\gamma}{2}}dbd\Phi,\\
&-\frac{1/\mu-2}{1+\alpha}\frac{C_\gamma }{4\pi C_0}\left(\frac{1+\alpha}{\alpha-1}\right)\beta^{-\frac{\alpha-1}{1+\alpha}}\int_0^{2\pi}\int_0^\infty \frac{ b^{-\frac{\gamma}{1+\alpha}}b^{\frac{-2}{1+\alpha}}\partial_{\overline{b}}h(b,\Phi)dbd\Phi}{\{\beta^{\frac{-2}{1+\alpha}}+b^{\frac{-2}{1+\alpha}}-2\beta^{\frac{-1}{1+\alpha}}b^{\frac{-1}{1+\alpha}}\cos(\phi+\beta-\Phi-b)\}^\frac{\gamma}{2}}\\
\end{align*}
Note now
$$
b\partial_\varPsi \partial_{\overline{b}}h=\partial_{\overline{b\varPsi}}h-\frac{2}{1+\alpha}\partial_{\overline{b}}h,
$$
implying
\begin{align*}
\mathcal{I}_3=&\frac{C_\gamma }{4\pi C_0}\left(\frac{1+\alpha}{\alpha-1}\right)\beta^{-\frac{\alpha-1}{1+\alpha}}\int_0^{2\pi}\int_0^\infty\frac{b^{-\frac{\gamma}{1+\alpha}}b^{\frac{-2}{1+\alpha}}\partial_{\overline{b\varPsi}}h(b,\Phi)}{\{\beta^{\frac{-2}{1+\alpha}}+b^{\frac{-2}{1+\alpha}}-2\beta^{\frac{-1}{1+\alpha}}b^{\frac{-1}{1+\alpha}}\cos(\phi+\beta-\Phi-b)\}^\frac{\gamma}{2}}dbd\Phi,\\
&-\frac{2}{1+\alpha}\frac{C_\gamma }{4\pi C_0}\left(\frac{1+\alpha}{\alpha-1}\right)\beta^{-\frac{\alpha-1}{1+\alpha}}\int_0^{2\pi}\int_0^\infty\frac{b^{-\frac{\gamma}{1+\alpha}}b^{\frac{-2}{1+\alpha}}\partial_{\overline{b}}h(b,\Phi)}{\{\beta^{\frac{-2}{1+\alpha}}+b^{\frac{-2}{1+\alpha}}-2\beta^{\frac{-1}{1+\alpha}}b^{\frac{-1}{1+\alpha}}\cos(\phi+\beta-\Phi-b)\}^\frac{\gamma}{2}}dbd\Phi,\\
&-\frac{1/\mu-2}{1+\alpha}\frac{C_\gamma }{4\pi C_0}\left(\frac{1+\alpha}{\alpha-1}\right)\beta^{-\frac{\alpha-1}{1+\alpha}}\int_0^{2\pi}\int_0^\infty \frac{ b^{-\frac{\gamma}{1+\alpha}}b^{\frac{-2}{1+\alpha}}\partial_{\overline{b}}h(b,\Phi)dbd\Phi}{\{\beta^{\frac{-2}{1+\alpha}}+b^{\frac{-2}{1+\alpha}}-2\beta^{\frac{-1}{1+\alpha}}b^{\frac{-1}{1+\alpha}}\cos(\phi+\beta-\Phi-b)\}^\frac{\gamma}{2}}\\
=&\frac{C_\gamma }{4\pi C_0}\left(\frac{1+\alpha}{\alpha-1}\right)\beta^{-\frac{\alpha-1}{1+\alpha}}\int_0^{2\pi}\int_0^\infty\frac{b^{-\frac{\gamma}{1+\alpha}}b^{\frac{-2}{1+\alpha}}\partial_{\overline{b\varPsi}}h(b,\Phi)}{\{\beta^{\frac{-2}{1+\alpha}}+b^{\frac{-2}{1+\alpha}}-2\beta^{\frac{-1}{1+\alpha}}b^{\frac{-1}{1+\alpha}}\cos(\phi+\beta-\Phi-b)\}^\frac{\gamma}{2}}dbd\Phi,\\
&-\frac{1/\mu}{1+\alpha}\frac{C_\gamma }{4\pi C_0}\left(\frac{1+\alpha}{\alpha-1}\right)\beta^{-\frac{\alpha-1}{1+\alpha}}\int_0^{2\pi}\int_0^\infty \frac{ b^{-\frac{\gamma}{1+\alpha}}b^{\frac{-2}{1+\alpha}}\partial_{\overline{b}}h(b,\Phi)dbd\Phi}{\{\beta^{\frac{-2}{1+\alpha}}+b^{\frac{-2}{1+\alpha}}-2\beta^{\frac{-1}{1+\alpha}}b^{\frac{-1}{1+\alpha}}\cos(\phi+\beta-\Phi-b)\}^\frac{\gamma}{2}}.
\end{align*}
Joining it with $I_1$ and $I_2$ we find
\begin{align*}
&\partial_f\tilde{F}(0,1)(\beta,\phi)=\frac{(1+\alpha)}{2}\partial_\beta(\beta h(\beta,\phi))\\
&-\frac{1}{\mu}\frac{C_\gamma }{4\pi C_0}\left(\frac{1}{\alpha-1}\right)\beta^{-\frac{\alpha-1}{1+\alpha}}\int_0^{2\pi}\int_0^\infty \frac{ b^{-\frac{\gamma}{1+\alpha}}b^{\frac{-2}{1+\alpha}}(\partial_{\overline{\varPsi}}h(b,\Phi)+\partial_{\overline{b}}h(b,\Phi))dbd\Phi}{\{\beta^{\frac{-2}{1+\alpha}}+b^{\frac{-2}{1+\alpha}}-2\beta^{\frac{-1}{1+\alpha}}b^{\frac{-1}{1+\alpha}}\cos(\phi+\beta-\Phi-b)\}^\frac{\gamma}{2}}.
\end{align*}
That concludes the proof.
\end{proof}

\subsection{Fredholmness at the trivial solution}

By using the expression from Proposition \ref{prop-lin-op-cancelation-1} we can prove that the linear operator is Fredholm of zero index in the following proposition. Since our goal is the use of the Implicit Function theorem, and then, proving that $\partial_f F(0,1)$ is an isomorphism we need to study its kernel and range. However, by the general theory on Fredholm operators we can restrict ourselves to check the kernel.

\begin{pro}\label{prop-fredholm}
Let $\gamma,\sigma\in(0,1)$ and $\alpha\in(1,2)$ satisfying \eqref{sigma-rel}. Then the linear operator $\partial_f \tilde{F}(0,1):X_1^{\sigma,\delta}\rightarrow Y^{\sigma,\delta}$ is Fredholm of zero index, for any $\delta<\min\{1-\gamma,\sigma\}$.
\end{pro}
\begin{proof}
We can decompose the linear operator as follows
\begin{align}
&\partial_f\tilde{F}(0,1)(\beta,\phi)=\frac{(1+\alpha)}{2}\partial_\beta(\beta h(\beta,\phi))\\
&-\frac{1}{\mu}\frac{C_\gamma }{4\pi C_0}\left(\frac{1}{\alpha-1}\right)\beta^{-\frac{\alpha-1}{1+\alpha}}\int_0^{2\pi}\int_0^\infty \frac{ b^{-\frac{\gamma}{1+\alpha}}b^{\frac{-2}{1+\alpha}}(\partial_{\overline{\varphi}}h(b,\Phi)+\partial_{\overline{b}}h(b,\Phi))dbd\Phi}{\{\beta^{\frac{-2}{1+\alpha}}+b^{\frac{-2}{1+\alpha}}-2\beta^{\frac{-1}{1+\alpha}}b^{\frac{-1}{1+\alpha}}\cos(\phi+\beta-\Phi-b)\}^\frac{\gamma}{2}}\\
=&\frac{(1+\alpha)}{2}\mathcal{L}_1-\frac{1}{\mu}\frac{C_\gamma }{4\pi C_0}\left(\frac{1}{\alpha-1}\right)\mathcal{L}_2,
\end{align}
where
$$
\mathcal{L}_1(h):=\partial_\beta(\beta h(\beta,\phi)),
$$
and
$$
\mathcal{L}_2(h):=\beta^{-\frac{\alpha-1}{1+\alpha}}\int_0^{2\pi}\int_0^\infty \frac{ b^{-\frac{\gamma}{1+\alpha}}b^{\frac{-2}{1+\alpha}}(\partial_{\overline{\varphi}}h(b,\Phi)+\partial_{\overline{b}}h(b,\Phi))dbd\Phi}{\{\beta^{\frac{-2}{1+\alpha}}+b^{\frac{-2}{1+\alpha}}-2\beta^{\frac{-1}{1+\alpha}}b^{\frac{-1}{1+\alpha}}\cos(\phi+\beta-\Phi-b)\}^\frac{\gamma}{2}}.
$$
The idea of the proof is to show that $\mathcal{L}_1:X_1^{\sigma,\delta}\rightarrow Y^{\sigma,\delta}$ is an isomorphism, whereas $\mathcal{L}_2:X_1^{\sigma,\delta}\rightarrow Y^{\sigma,\delta}$ is compact. Then, since compact perturbations of Fredholm operators remain being Fredholm of zero index, we can conclude the proof.

\noindent $\bullet$ $\mathcal{L}_1:X_1^{\sigma,\delta}\rightarrow Y^{\sigma,\delta}$ is an isomorphism.

Take $g\in Y^{\sigma,\delta}$, we wish to solve $\mathcal{L}_1(h)=g$. We can explicitly solve such differential equation as
$$
h(\beta,\phi)=\frac{K}{\beta}+\frac{1}{\beta}\int_0^\beta g(s,\phi)ds.
$$
for any $K$. However, $h$ must lie in $X_1^{\sigma,\delta}$. In particular, $h$ vanishes at $0$ and $K=0$:
\begin{equation}\label{h-iso-1}
h(\beta,\phi)=\frac{1}{\beta}\int_0^\beta g(s,\phi)ds=\int_0^1 g(s\beta,\phi).
\end{equation}
Let us now check that if $g\in Y^{\sigma,\delta}$, then $h\in X_1^{\sigma,\delta}$, where the norm of $X_1^{\sigma,\delta}$ is given in \eqref{norm-X1}. Let us start showing that
$$
||h||_{\mathcal{X}^{\sigma,\delta}}=||h||_{L^\infty_\sigma}+||h||_{C^\delta_\beta}<C.
$$
We find that
\begin{align*}
||h||_{L^\infty_\sigma}=\sup_{\beta\in[0,+\infty)}\frac{(1+\beta)^{2\sigma}}{\beta^\sigma} \int_0^1 g(s\beta,\phi)ds\leq ||g||_{L^\infty_\sigma} \sup_{\beta\in[0,+\infty)}(1+\beta)^{2\sigma} \int_0^1 \frac{s^\sigma}{(1+s\beta)^{2\sigma}}ds<C,
\end{align*}
since $\sigma<1$. Let us now show that $||h||_{C^\delta_\beta}<C$:
\begin{align*}
|h(\beta,\phi)-h(b,\phi)|\leq \int_0^1|g(s\beta,\phi)-g(sb,\phi)|ds\leq C||g||_{C^\delta_\beta}\int_0^1 \frac{|s\beta-sb|^\alpha}{|s\beta+sb|^\alpha}ds\leq C||g||_{C^\delta_\beta}\frac{|\beta-b|^\delta}{|\beta+b|^\delta},
\end{align*}
implying that $||h||_{C^\delta_\beta}<C||g||_{C^\delta_\beta}$. Let us show that $||h||_{\textnormal{Lip}_\phi}<C$:
$$
|h(\beta,\phi)-h(\beta,\Phi)|\leq \int_0^1 |g(s\beta,\phi)-g(s\beta,\Phi)|ds\leq ||g||_{\textnormal{Lip}_\phi} \sin((\Phi-\phi)/2),
$$
implying that 
$$
||h||_{\textnormal{Lip}_\phi}<||g||_{\textnormal{Lip}_\phi}.
$$
Let us move to $\beta h_\beta$ and note that
$$
\beta h_\beta (\beta,\phi)=-h(\beta,\phi)+g(\beta,\phi),
$$
and then
\begin{align*}
||\beta h_\beta||_{\mathcal{X}^{\sigma,\delta}}+||\beta h_\beta||_{\textnormal{Lip}_\phi}\leq& ||h||_{\mathcal{X}^{\sigma,\delta}}+||h||_{\textnormal{Lip}_\phi}\\
&+||g||_{\mathcal{X}^{\sigma,\delta}}+||g||_{\textnormal{Lip}_\phi}\\
\leq& C.
\end{align*}
Next, we will show that $\beta h_\varphi\in\mathcal{X}^{\sigma,\delta}$ and we will separate in $\beta\in[0,\beta_0)$ and $\beta\in[\beta_0,+\infty)$.  For the first case,  it can be written as
\begin{align*}
h_\varphi=&\frac{1}{\beta}\int_0^\beta g_\phi(s,\phi)ds+\frac{1}{\beta} h(\beta,\phi)-\frac{1}{\beta}g(\beta,\phi)\\
=&\frac{1}{\beta}\int_0^\beta g_\varphi(s,\phi)ds+\frac{1}{\beta}\int_0^\beta g_s(s,\phi)ds+\frac{1}{\beta}h(\beta,\phi)-\frac{1}{\beta}g(\beta,\phi)\\
=&\frac{1}{\beta}\int_0^\beta g_\varphi(s,\phi)ds+\frac{1}{\beta} g(\beta,\phi)ds+\frac{1}{\beta}h(\beta,\phi)-\frac{1}{\beta}g(\beta,\phi),
\end{align*}
and then
\begin{equation}\label{h-iso-2}
\beta h_\varphi=\int_0^\beta g_\varphi(s,\phi)ds+ h(\beta,\phi).
\end{equation}
Since we have showed that $h\in \mathcal{X}^{\sigma,\delta}$, it remains only to check the first term:
\begin{align*}
\sup_{\beta\in[0,\beta_0)}\frac{(1+\beta)^{2\sigma}}{\beta^\sigma}\int_0^\beta g_\varphi(s,\phi)ds\leq& ||\beta g_\varphi||_{L^\infty_\sigma}\sup_{\beta\in[0,\beta_0)}\frac{1}{\beta^\sigma}\int_0^\beta \frac{1}{s}\frac{s^\sigma}{(1+s)^{2\sigma}}ds\\
\leq& ||\beta g_\varphi||_{L^\infty_\sigma}\sup_{\beta\in[0,\beta_0)}\int_0^1 s^{\sigma-1}ds\\
\leq& ||\beta g_\varphi||_{L^\infty_\sigma}.
\end{align*}
In the case $\beta\in[\beta_0,+\infty)$ we write $h$ as
$$
h(\beta,\phi)=\frac{1}{\beta}\int_{\infty}^\beta g(s,\phi)ds,
$$
instead of \eqref{h-iso-1} by using that $h$ must vanish at $0$ and $+\infty$. In that case, we have
\begin{equation}\label{h-iso-3}
\beta h_{\varphi}=\int_{\infty}^\beta g_\varphi(s,\phi)ds+h(\beta,\phi),
\end{equation}
and hence it remains to check the first term:
\begin{align*}
\sup_{\beta\in[\beta_0,+\infty)}\frac{(1+\beta)^{2\sigma}}{\beta^\sigma}\int_\infty^\beta g_\varphi(s,\phi)ds\leq& ||\beta g_\varphi||_{L^\infty_\sigma}\sup_{\beta\in[\beta_0,+\infty)}\beta^\sigma\int_\infty^\beta \frac{1}{s^{1+\sigma}}ds\\
\lesssim& ||\beta g_\varphi||_{L^\infty_\sigma},
\end{align*}
since $\sigma\in(0,1)$. In that way, we have showed that
$$
||\beta h_\varphi||_{L^\infty_\sigma}\leq C ||\beta g_\varphi||.
$$
Let us now check that $||\beta h_{\varphi}||_{C^\delta_\beta}<C$, which amounts to check such property to the first term in \eqref{h-iso-2}, that is,
\begin{align*}
\left|\int_{\beta}^{b} g_\varphi(s,\phi)ds\right|\leq C\frac{|\beta-b|^\delta}{|\beta+b|^\delta}.
\end{align*}
Take first $|\beta+b|<\beta_0<1$, and then
\begin{align*}
\left|\int_{\beta}^{b} g_\varphi(s,\phi)ds\right|\leq C||g_\varphi||_{\mathcal{X}^{\sigma,\delta}}\int_\beta^b s^{\sigma-1}ds\leq C|\beta-b|^\sigma \leq C\frac{|\beta-b|^\delta}{|\beta+b|^\delta},
\end{align*}
if $\delta<\sigma<1$. In the case $||\beta+b||>\beta_0$ we use the expression of $\beta h_\varphi$ we find
\begin{align*}
\left|\int_{\beta}^{b} g_\varphi(s,\phi)ds\right|\leq C||g_\varphi||_{\mathcal{X}^{\sigma,\delta}}\int_\beta^b s^{-\sigma-1}ds\leq C\left|\frac{1}{\beta^\sigma}-\frac{1}{b^\sigma}\right| \leq C\frac{|\beta-b|^\sigma}{|\beta+b|^{2\sigma}}\leq C\frac{|\beta-b|^\sigma}{|\beta+b|^{\sigma}}\leq C\frac{|\beta-b|^\delta}{|\beta+b|^{\delta}}.
\end{align*}
Let us finish with $\beta^2 h_{\beta,\varphi}$ noting that
\begin{align*}
h_{\beta,\varphi}(\beta,\phi)=&-\frac{1}{\beta}\frac{1}{\beta}\int_0^\beta g_\varphi (s,\phi)ds+\frac{1}{\beta}g_{\varphi}(\beta,\phi)+\frac{1}{\beta}h_\beta (\beta,\phi)-\frac{1}{\beta^2}h(\beta,\phi)\\
&=-\frac{1}{\beta}(h_\varphi(\beta,\phi)-\frac{1}{\beta}h(\beta,\phi))+\frac{1}{\beta}g_{\varphi}(\beta,\phi)+\frac{1}{\beta}h_\beta (\beta,\phi)-\frac{1}{\beta^2}h(\beta,\phi),
\end{align*}
and then
$$
\beta^2 h_{\beta,\varphi}(\beta,\phi)=-\beta h_\varphi(\beta,\phi)+\beta g_{\varphi}(\beta,\phi)+\beta h_\beta (\beta,\phi),
$$
implying then that $\beta^2 h_{\beta,\varphi}\in \mathcal{X_{\sigma,\alpha}}.$

\noindent $\bullet$ $\mathcal{L}_2:X_1^{\sigma,\delta}\rightarrow Y^{\sigma,\delta}$ is compact.
From Proposition \ref{prop-well-def} we have that $\mathcal{L}_2\in Y^{\sigma,\delta}$. To prove that $\mathcal{L}_2$ is compact, we will prove that $\mathcal{L}_2\in Y^{\sigma,\delta'}$, with $\delta'>\delta$. Then, since the embedding $Y^{\sigma,\delta'}\subset Y^{\sigma,\delta}$ is compact (we refer to \cite[Lemma 2.2]{Pacard-Riviere:book-vortices}) we have that $\mathcal{L}_2:X_1^{\sigma,\delta}\rightarrow Y^{\sigma,\delta}$ is a compact operator.

Take $\mathcal{L}_2$ and let us prove that $\beta \partial_\varphi\mathcal{L}_2\in C^{\delta'}_\beta$ with $\delta'>\delta$. For that, we compute $\partial_\varphi \mathcal{L}_2$:
\begin{align*}
&\partial_\varphi \mathcal{L}_2(h)=C\beta^{-1}\mathcal{L}_2(h)\\
&+C\beta^{-1}\beta^{-\frac{\alpha-1}{1+\alpha}}\int_0^{2\pi}\int_0^\infty \frac{ b^{-\frac{\gamma}{1+\alpha}}b^{\frac{-2}{1+\alpha}}(\partial_{\overline{\varphi}}h(b,\Phi)+\partial_{\overline{b}}h(b,\Phi))\left\{\beta^\frac{-2}{1+\alpha}-\beta^\frac{-1}{1+\alpha}b^\frac{-1}{1+\alpha}\cos(\phi+\beta-\Phi-b)\right\}dbd\Phi}{\{\beta^{\frac{-2}{1+\alpha}}+b^{\frac{-2}{1+\alpha}}-2\beta^{\frac{-1}{1+\alpha}}b^{\frac{-1}{1+\alpha}}\cos(\phi+\beta-\Phi-b)\}^\frac{2+\gamma}{2}},
\end{align*}
for some constant $C$. Then
{\footnotesize\begin{align*}
&\beta \partial_\varphi\mathcal{L}_2(h)=C\mathcal{L}_2(h)\\
&+C\beta^{-\frac{\alpha-1}{1+\alpha}}\int_0^{2\pi}\int_0^\infty \frac{ b^{-\frac{\gamma}{1+\alpha}}b^{\frac{-2}{1+\alpha}}(\partial_{\overline{\varphi}}h(b,\Phi)+\partial_{\overline{b}}h(b,\Phi))\left\{\beta^\frac{-1}{1+\alpha}(\beta^\frac{-1}{1+\alpha}-b^\frac{-1}{1+\alpha})+\beta^\frac{-1}{1+\alpha}b^\frac{-1}{1+\alpha}(1-\cos(\phi+\beta-\Phi-b))\right\}dbd\Phi}{\{\beta^{\frac{-2}{1+\alpha}}+b^{\frac{-2}{1+\alpha}}-2\beta^{\frac{-1}{1+\alpha}}b^{\frac{-1}{1+\alpha}}\cos(\phi+\beta-\Phi-b)\}^\frac{2+\gamma}{2}}.
\end{align*}
}
Note that in the previous expression the singularity is critical. To solve this problem we should integrate by parts somehow but we can not differentiate $\partial_{\overline{\varphi}}h$ with respect to $\partial_\varphi$ and then we have to do the following decomposition:
{\footnotesize\begin{align*}
&\beta \partial_\varphi \mathcal{L}_2(h)=C\mathcal{L}_2(h)\\
&+C\beta^{-\frac{\alpha-1}{1+\alpha}}\int_0^{2\pi}\int_0^\infty \frac{ b^{-\frac{\gamma}{1+\alpha}}b^{\frac{-2}{1+\alpha}}(\partial_{\overline{\varphi}}h(b,\Phi)-\partial_{\overline{\varphi}}h(\beta,\phi))\left\{\beta^\frac{-1}{1+\alpha}(\beta^\frac{-1}{1+\alpha}-b^\frac{-1}{1+\alpha})+\beta^\frac{-1}{1+\alpha}b^\frac{-1}{1+\alpha}(1-\cos(\phi+\beta-\Phi-b))\right\}dbd\Phi}{\{\beta^{\frac{-2}{1+\alpha}}+b^{\frac{-2}{1+\alpha}}-2\beta^{\frac{-1}{1+\alpha}}b^{\frac{-1}{1+\alpha}}\cos(\phi+\beta-\Phi-b)\}^\frac{2+\gamma}{2}}\\
&+C\beta^{-\frac{\alpha-1}{1+\alpha}}\int_0^{2\pi}\int_0^\infty \frac{ b^{-\frac{\gamma}{1+\alpha}}b^{\frac{-2}{1+\alpha}}(\partial_{\overline{b}}h(b,\Phi))\left\{\beta^\frac{-1}{1+\alpha}(\beta^\frac{-1}{1+\alpha}-b^\frac{-1}{1+\alpha})+\beta^\frac{-1}{1+\alpha}b^\frac{-1}{1+\alpha}(1-\cos(\phi+\beta-\Phi-b))\right\}dbd\Phi}{\{\beta^{\frac{-2}{1+\alpha}}+b^{\frac{-2}{1+\alpha}}-2\beta^{\frac{-1}{1+\alpha}}b^{\frac{-1}{1+\alpha}}\cos(\phi+\beta-\Phi-b)\}^\frac{2+\gamma}{2}}\\
&+C\partial_{\overline{\varphi}}h(\beta,\phi)\beta^{-\frac{\alpha-1}{1+\alpha}}\int_0^{2\pi}\int_0^\infty \frac{ b^{-\frac{\gamma}{1+\alpha}}b^{\frac{-2}{1+\alpha}}\left\{\beta^\frac{-1}{1+\alpha}(\beta^\frac{-1}{1+\alpha}-b^\frac{-1}{1+\alpha})+\beta^\frac{-1}{1+\alpha}b^\frac{-1}{1+\alpha}(1-\cos(\phi+\beta-\Phi-b))\right\}dbd\Phi}{\{\beta^{\frac{-2}{1+\alpha}}+b^{\frac{-2}{1+\alpha}}-2\beta^{\frac{-1}{1+\alpha}}b^{\frac{-1}{1+\alpha}}\cos(\phi+\beta-\Phi-b)\}^\frac{2+\gamma}{2}}\\
=:&C\mathcal{L}_2(h)+L_{2,1}(h)+L_{2,2}(h)+L_{2,3}(h).
\end{align*}
}
Let us show the details for one of the terms and we have chosen $L_{2,2}$ due to its difficulty. Note that the singularity coming from the denominator is not a priori integrable and we need to use some extra cancellation. To achieve that, we need to integrate by parts and then we have to add and subtract some appropriate function in the denominator. Indeed, denote

$$
J(\beta,\phi,b,\Phi):=\{\beta_1^{\frac{-2}{1+\alpha}}+b^{\frac{-2}{1+\alpha}}-2\beta_1^{\frac{-1}{1+\alpha}}b^{\frac{-1}{1+\alpha}}\cos(\phi+\beta_1-\Phi-b)\}.
$$
Notice that $J$ is indeed the denominator at the trivial solution and then $J=C D(0)$, for some constant $C$, where $D(f)$ is defined in \eqref{exp-denominator-2}. In particular, following Proposition \ref{prop-denominator} and the particular expression for $J$ one finds
\begin{align}\label{J-bound}
|J(\beta,\phi,b,\Phi)|\geq C (\beta^{\frac{-2}{1+\alpha}}+b^\frac{-2}{1+\alpha})\left\{\frac{|\beta-b|^2}{|\beta+b|^2}+\sin^2((\beta-b+\phi-\Phi)/2)\right\}.
\end{align}
By using the expression of $J$ we also find
\begin{align*}
\partial_\varphi J(\beta,\phi,b,\Phi)=\frac{2}{1+\alpha}\beta^{-1}\left\{ \beta^\frac{-1}{1+\alpha}(\beta^\frac{-1}{1+\alpha}-b^\frac{-1}{1+\alpha})+\beta^\frac{-1}{1+\alpha}b^\frac{-1}{1+\alpha}(1-\cos(\beta+\phi-b-\Phi)\right\},
\end{align*}
and then
\begin{align*}
\beta \partial_\varphi J(\beta,\phi,b,\Phi)+b\partial_\varPsi J(\beta,\phi,b,\Phi)=&\frac{2}{1+\alpha}\left\{ (\beta^\frac{-1}{1+\alpha}-b^\frac{-1}{1+\alpha})^2+2\beta^\frac{-1}{1+\alpha}b^\frac{-1}{1+\alpha}(1-\cos(\beta+\phi-b-\Phi)\right\}\\
=&\frac{2}{1+\alpha} J(\beta,\phi,b,\Phi).
\end{align*}
Motivated by the previous computation, we decompose $L_{2,2}$ as follows
{\footnotesize\begin{align*}
L_{2,2}(h)(\beta)=&C \beta^{-\frac{\alpha-1}{1+\alpha}}\int_0^{2\pi}\int_0^\infty \frac{ b^{-\frac{\gamma}{1+\alpha}}b^{\frac{-2}{1+\alpha}}(\partial_{\overline{b}}h(b,\Phi))}{\{\beta^{\frac{-2}{1+\alpha}}+b^{\frac{-2}{1+\alpha}}-2\beta^{\frac{-1}{1+\alpha}}b^{\frac{-1}{1+\alpha}}\cos(\phi+\beta-\Phi-b)\}^\frac{\gamma}{2}}dbd\Phi\\
&+\frac{2}{\gamma} C \beta^{-\frac{\alpha-1}{1+\alpha}}\int_0^{2\pi}\int_0^\infty b b^{-\frac{\gamma}{1+\alpha}}b^{\frac{-2}{1+\alpha}}(\partial_{\overline{b}}h(b,\Phi))\\
&\partial_{\varPsi}\frac{1}{\{\beta^{\frac{-2}{1+\alpha}}+b^{\frac{-2}{1+\alpha}}-2\beta^{\frac{-1}{1+\alpha}}b^{\frac{-1}{1+\alpha}}\cos(\phi+\beta-\Phi-b)\}^\frac{\gamma}{2}}dbd\Phi.
\end{align*}
}
Integrating by parts in the last term we have that the boundary terms vanish since $h\in X_1^{\sigma,\delta}$ (we refer to the proof of Proposition \ref{prop-lin-op-cancelation-1}) and then
{\footnotesize\begin{align*}
L_{2,2}(h)(\beta)=&C \beta^{-\frac{\alpha-1}{1+\alpha}}\int_0^{2\pi}\int_0^\infty \frac{ b^{-\frac{\gamma}{1+\alpha}}b^{\frac{-2}{1+\alpha}}(\partial_{\overline{b}}h(b,\Phi))}{\{\beta^{\frac{-2}{1+\alpha}}+b^{\frac{-2}{1+\alpha}}-2\beta^{\frac{-1}{1+\alpha}}b^{\frac{-1}{1+\alpha}}\cos(\phi+\beta-\Phi-b)\}^\frac{\gamma}{2}}dbd\Phi\\
&-\frac{2}{\gamma} C \beta^{-\frac{\alpha-1}{1+\alpha}}\int_0^{2\pi}\int_0^\infty  \frac{bb^{-\frac{\gamma}{1+\alpha}}b^{\frac{-2}{1+\alpha}}(\partial_{\overline{b},\varPsi}h(b,\Phi))dbd\Phi}{\{\beta^{\frac{-2}{1+\alpha}}+b^{\frac{-2}{1+\alpha}}-2\beta^{\frac{-1}{1+\alpha}}b^{\frac{-1}{1+\alpha}}\cos(\phi+\beta-\Phi-b)\}^\frac{\gamma}{2}}\\
&-\frac{2}{\gamma} C\frac{1+\gamma-\alpha}{1+\alpha} \beta^{-\frac{\alpha-1}{1+\alpha}}\int_0^{2\pi}\int_0^\infty  \frac{b^{-\frac{\gamma}{1+\alpha}}b^{\frac{-2}{1+\alpha}}(\partial_{\overline{b}}h(b,\Phi))dbd\Phi}{\{\beta^{\frac{-2}{1+\alpha}}+b^{\frac{-2}{1+\alpha}}-2\beta^{\frac{-1}{1+\alpha}}b^{\frac{-1}{1+\alpha}}\cos(\phi+\beta-\Phi-b)\}^\frac{\gamma}{2}}\\
=:&L_{2,2,1}+L_{2,2,2}+L_{2,2,3}.
\end{align*}
}


Take now $\beta_1<\beta_2$ implying
\begin{equation}\label{beta1beta2-2}
\beta_2^{-\frac{\alpha-1}{1+\alpha}}<\beta_1^{-\frac{\alpha-1}{1+\alpha}},
\end{equation}
and let us show the H\"older regularity for $L_{2,2}$. By implementing the difference of $L_{2,2}$ in the two points we find
{\footnotesize\begin{align*}
|L_{2,2}(h)(\beta_1)-&L_{2,2}(h)(\beta_2)|\leq C\frac{|\beta_1-\beta_2|^{\delta'}}{|\beta_1+\beta_2|^{\delta'}}(\beta_1^{-\frac{\alpha-1}{1+\alpha}}+\beta_2^{-\frac{\alpha-1}{1+\alpha}})\Big\{\int_0^{2\pi}\int_0^\infty \frac{ b^{-\frac{\gamma}{1+\alpha}}b^{\frac{-2}{1+\alpha}}(\partial_{\overline{b}}h(b,\Phi))}{J(\beta_1,\phi,b,\Phi)^\frac{\gamma}{2}}dbd\Phi\\
&+\int_0^{2\pi}\int_0^\infty  \frac{bb^{-\frac{\gamma}{1+\alpha}}b^{\frac{-2}{1+\alpha}}(\partial_{\overline{b},\varPsi}h(b,\Phi))dbd\Phi}{J(\beta_1,\phi,b,\Phi)^\frac{\gamma}{2}}\\
&+\int_0^\infty  \frac{b^{-\frac{\gamma}{1+\alpha}}b^{\frac{-2}{1+\alpha}}(\partial_{\overline{b}}h(b,\Phi))dbd\Phi}{J(\beta_1,\phi,b,\Phi)^\frac{\gamma}{2}}\Big\}\\
&+\beta_2^{-\frac{\alpha-1}{1+\alpha}}\int_0^{2\pi}\int_0^\infty \frac{ b^{-\frac{\gamma}{1+\alpha}}b^{\frac{-2}{1+\alpha}}(\partial_{\overline{b}}h(b,\Phi))}{J(\beta_1,\phi,b,\Phi)^\frac{\gamma}{2}J(\beta_2,\phi,b,\Phi)^\frac{\gamma}{2}}\frac{|J(\beta_1,\phi,b,\Phi)-J(\beta_2,\phi,b,\Phi)|}{J(\beta_1,\phi,b,\Phi)^\frac{2-\gamma}{2}+J(\beta_2,\phi,b,\Phi)^\frac{2-\gamma}{2}}dbd\Phi \\
&+\beta_2^{-\frac{\alpha-1}{1+\alpha}}      \int_0^{2\pi}\int_0^\infty  \frac{bb^{-\frac{\gamma}{1+\alpha}}b^{\frac{-2}{1+\alpha}}(\partial_{\overline{b},\varPsi}h(b,\Phi))}{J(\beta_1,\phi,b,\Phi)^\frac{\gamma}{2}J(\beta_2,\phi,b,\Phi)^\frac{\gamma}{2}}\frac{|J(\beta_1,\phi,b,\Phi)-J(\beta_2,\phi,b,\Phi)|}{J(\beta_1,\phi,b,\Phi)^\frac{2-\gamma}{2}+J(\beta_2,\phi,b,\Phi)^\frac{2-\gamma}{2}}dbd\Phi \\
&+\beta_2^{-\frac{\alpha-1}{1+\alpha}}      \int_0^\infty  \frac{b^{-\frac{\gamma}{1+\alpha}}b^{\frac{-2}{1+\alpha}}(\partial_{\overline{b}}h(b,\Phi))}{J(\beta_1,\phi,b,\Phi)^\frac{\gamma}{2}J(\beta_2,\phi,b,\Phi)^\frac{\gamma}{2}}\frac{|J(\beta_1,\phi,b,\Phi)-J(\beta_2,\phi,b,\Phi)|}{J(\beta_1,\phi,b,\Phi)^\frac{2-\gamma}{2}+J(\beta_2,\phi,b,\Phi)^\frac{2-\gamma}{2}}dbd\Phi,
\end{align*}
}
where $\delta'$ is any number in $(0,1)$. Taking into account \eqref{beta1beta2-2} the first three terms can be related to the operator $\mathcal{H}$ defined in \eqref{H-operator}, and for the last three we use that $h\in X_1^{\sigma,\delta}$ obtaining
{\footnotesize\begin{align*}
&|L_{2,2}(h)(\beta_1)-L_{2,2}(h)(\beta_2)|\leq C\frac{|\beta_1-\beta_2|^{\delta'}}{|\beta_1+\beta_2|^{\delta'}}\Big\{||\mathcal{H}_0(\partial_{\overline{b}}h(b,\Phi),1)||_{L^\infty}+||\mathcal{H}_0(b(\partial_{\overline{b},\varPsi}h(b,\Phi)),1)||_{L^\infty}\Big\}\\
&+C||h||_{X_1^{\sigma,\delta}}\beta_2^{-\frac{\alpha-1}{1+\alpha}}\int_0^{2\pi}\int_0^\infty \frac{ b^{-\frac{\gamma}{1+\alpha}}b^{\frac{-2}{1+\alpha}}b^\sigma/(1+b)^{2\sigma}}{J(\beta_1,\phi,b,\Phi)^\frac{\gamma}{2}J(\beta_2,\phi,b,\Phi)^\frac{\gamma}{2}}\frac{|J(\beta_1,\phi,b,\Phi)-J(\beta_2,\phi,b,\Phi)|}{J(\beta_1,\phi,b,\Phi)^\frac{2-\gamma}{2}+J(\beta_2,\phi,b,\Phi)^\frac{2-\gamma}{2}}dbd\Phi.
\end{align*}
}
By Proposition \ref{prop-operator-H} we have that 
$$
||\mathcal{H}_0(\partial_{\overline{b}}h(b,\Phi),1)||_{L^\infty}+||\mathcal{H}_0(b(\partial_{\overline{b},\varPsi}h(b,\Phi)),1)||_{L^\infty}\leq C||h||_{X_1^{\sigma,\delta}}
$$

Next, we need to estimate the difference of $J$ in these two points:
\begin{align}\label{J-est}
&|J(\beta_1,\phi,b,\Phi)-J(\beta_2,\phi,b,\Phi)|\leq C\frac{|\beta_1-\beta_2|}{|\beta_1+\beta_2|}(\beta_1^\frac{-1}{1+\alpha}+\beta_2^\frac{-1}{1+\alpha})(J(\beta_1,\phi,b,\Phi)^\frac12+J(\beta_2,\phi,b,\Phi)^\frac12)\nonumber\\
&+C\frac{|\beta_1-\beta_2|}{|\beta_1+\beta_2|}\beta_2^\frac{-1}{1+\alpha}|\beta_1+\beta_2|(J(\beta_1,\phi,b,\Phi)^\frac12+J(\beta_2,\phi,b,\Phi)^\frac12)\\
\leq& C\frac{|\beta_1-\beta_2|}{|\beta_1+\beta_2|}(1+|\beta_1+\beta_2|)(\beta_1^\frac{-1}{1+\alpha}+\beta_2^\frac{-1}{1+\alpha})(J(\beta_1,\phi,b,\Phi)^\frac12+J(\beta_2,\phi,b,\Phi)^\frac12)\nonumber\\
\end{align}
and interpolating with
$$
|J(\beta_1,\phi,b,\Phi)-J(\beta_2,\phi,b,\Phi)|\leq C|J(\beta_1,\phi,b,\Phi)+J(\beta_2,\phi,b,\Phi)|,
$$
we achieve
\begin{align}\label{J-est-2}
|J(\beta_1,\phi,b,\Phi)-J(\beta_2,\phi,b,\Phi)|\leq& C\frac{|\beta_1-\beta_2|^{\delta'}}{|\beta_1+\beta_2|^{\delta'}}(1+|\beta_1+\beta_2|^{\delta'})(\beta_1^\frac{-1}{1+\alpha}+\beta_2^\frac{-1}{1+\alpha})^{\delta'}\\
&\times(J(\beta_1,\phi,b,\Phi)^{1-\frac{\delta'}{2}}+J(\beta_2,\phi,b,\Phi)^{1-\frac{\delta'}{2}}),
\end{align}
where $\delta'\in(0,1)$. Finally,
{\footnotesize\begin{align*}
&|L_{2,2}(h)(\beta_1)-L_{2,2}(h)(\beta_2)|\leq C||h||_{X_1^{\sigma,\delta}} \frac{|\beta_1-\beta_2|^{\delta'}}{|\beta_1+\beta_2|^{\delta'}}\Big[1\\
&+(1+|\beta_1+\beta_2|^{\delta'})(\beta_1^\frac{-1}{1+\alpha}+\beta_2^\frac{-1}{1+\alpha})^{\delta'}\beta_2^{-\frac{\alpha-1}{1+\alpha}}\\
&\times\int_0^{2\pi}\int_0^\infty \frac{ b^{-\frac{\gamma}{1+\alpha}}b^{\frac{-2}{1+\alpha}}b^\sigma/(1+b)^{2\sigma}}{J(\beta_1,\phi,b,\Phi)^\frac{\gamma}{2}J(\beta_2,\phi,b,\Phi)^\frac{\gamma}{2}}\frac{J(\beta_1,\phi,b,\Phi)^{-\frac{\delta'}{2}}+J(\beta_2,\phi,b,\Phi)^{-\frac{\delta'}{2}}}{J(\beta_1,\phi,b,\Phi)^\frac{-\gamma}{2}+J(\beta_2,\phi,b,\Phi)^\frac{-\gamma}{2}}dbd\Phi\Big]\\
&\leq C||h||_{X_1^{\sigma,\delta}} \frac{|\beta_1-\beta_2|^{\delta'}}{|\beta_1+\beta_2|^{\delta'}}\Big[1\\
&+(1+|\beta_1+\beta_2|^{\delta'})(\beta_1^\frac{-1}{1+\alpha}+\beta_2^\frac{-1}{1+\alpha})^{\delta'}\beta_2^{-\frac{\alpha-1}{1+\alpha}}\int_0^{2\pi}\int_0^\infty { b^{-\frac{\gamma}{1+\alpha}}b^{\frac{-2}{1+\alpha}}b^\sigma/(1+b)^{2\sigma}}\\
&\times\left(\frac{1}{J(\beta_1,\phi,b,\Phi)^\frac{\delta'}{2}J(\beta_2,\phi,b,\Phi)^\frac{\gamma}{2}}+\frac{1}{J(\beta_1,\phi,b,\Phi)^\frac{\delta'}{2}J(\beta_2,\phi,b,\Phi)^\frac{\gamma}{2}}\right)dbd\Phi\Big]
\end{align*}
}
The last integral is similar to the one in \eqref{Hholder-2} and then we find that it is bounded if $\delta'<\sigma$. Furthermore, since $\delta'$ is arbitrary we can choose $\delta'>\delta$ to have the compactness.
\end{proof}

\subsection{Expression of the linearized operator in Fourier series}
In the next proposition, we show that the linear operator can be computed in Fourier modes depending on the Fourier coefficients of $h$. Assume then that
$h$ takes the form  
\begin{equation}\label{h-fourier}
h(\beta,\phi)=\sum_{k\in\Z} h_k(\beta)e^{ik\phi}.
\end{equation}

\begin{pro}\label{prop-lin-op-cancelation-2}
Let $h$ be expressed as \eqref{h-fourier}. Then
\begin{align*}
&\partial_f\tilde{F}(0,1)(\beta,\phi)=\sum_{k\in \Z}e^{ik\phi}\left\{\frac{(1+\alpha)}{2}\partial_\beta(\beta h_k(\beta))\right.\\
&\left.-ik \frac{1}{\mu}\frac{C_\gamma }{4\pi C_0}\left(\frac{1}{\alpha-1}\right)e^{ik\beta}\beta^{-\frac{\alpha-1}{1+\alpha}}\int_0^{2\pi}\int_0^\infty \frac{ e^{-ikb}b^{-\frac{\gamma}{1+\alpha}}b^{\frac{-2}{1+\alpha}}b h_k(b)\cos(k\eta)dbd\eta}{\{\beta^{\frac{-2}{1+\alpha}}+b^{\frac{-2}{1+\alpha}}-2\beta^{\frac{-1}{1+\alpha}}b^{\frac{-1}{1+\alpha}}\cos(\eta)\}^\frac{\gamma}{2}}\right\}.\nonumber
\end{align*}
\end{pro}
\begin{proof}
Let us assume that $h(\beta,\phi)=\sum_{k\in\Z} h_k(\beta)e^{ik\phi}$ in the expression of Proposition \ref{prop-lin-op-cancelation-1}. Using the definition of $\partial_{\overline{\beta}}$ and $\partial_{\overline{\varphi}}$ we find
$$
\partial_{\overline{\beta}}h(\beta,\phi)=\frac{\alpha-1}{1+\alpha}h(\beta,\phi)+\beta \partial_\beta h(\beta,\phi)
$$
and
$$
\partial_{\overline{\varphi}}h(\beta,\phi)=-\frac{\alpha-1}{1+\alpha}h(\beta,\phi)+\beta \partial_\varphi h(\beta,\phi)
$$
obtaining
$$
\partial_{\overline{\varphi}}h(\beta,\phi)+\partial_{\overline{\beta}}h(\beta,\phi)=\beta (\partial_\varphi h(\beta,\phi)+\partial_\beta h(\beta,\phi))=\beta \partial_\phi h(\beta,\phi),
$$
taking into account that $\partial_\varphi=\partial_\phi-\partial_\beta$. Hence, we can compute the Fourier series of the above term as
\begin{equation}\label{partialpsi-fourier}
\partial_{\overline{\varphi}}h(\beta,\phi)+\partial_{\overline{\beta}}h(\beta,\phi)=\sum_{k\in \Z}ik\beta h_k(\beta)e^{ik\phi}.
\end{equation}
Moreover
\begin{equation}\label{partialbeta-fourier}
\partial_\beta (\beta h(\beta,\phi))=\sum_{k\in \Z}\partial_{\beta}(\beta h_k(\beta))e^{ik\phi}.
\end{equation}
Inserting \eqref{partialpsi-fourier} and \eqref{partialbeta-fourier} into \eqref{partialfF-1-2} we get

\begin{align}\label{partialfF-1-2-3}
&\partial_f\tilde{F}(0,1)(\beta,\phi)=\sum_{k\in \Z}e^{ik\phi}\left\{\frac{(1+\alpha)}{2}\partial_\beta(\beta h_k(\beta))\right.\\
&\left.-ik e^{-ik\phi}\frac{1}{\mu}\frac{C_\gamma }{4\pi C_0}\left(\frac{1}{\alpha-1}\right)\beta^{-\frac{\alpha-1}{1+\alpha}}\int_0^{2\pi}\int_0^\infty \frac{ b^{-\frac{\gamma}{1+\alpha}}b^{\frac{-2}{1+\alpha}}b h_k(b)e^{ik\Phi}dbd\Phi}{\{\beta^{\frac{-2}{1+\alpha}}+b^{\frac{-2}{1+\alpha}}-2\beta^{\frac{-1}{1+\alpha}}b^{\frac{-1}{1+\alpha}}\cos(\phi+\beta-\Phi-b)\}^\frac{\gamma}{2}}\right\}.\nonumber
\end{align}
Note that we can write the integral as follows
\begin{align*}
\int_0^{2\pi}\int_0^\infty& \frac{ b^{-\frac{\gamma}{1+\alpha}}b^{\frac{-2}{1+\alpha}}b h_k(b)e^{ik\Phi}dbd\Phi}{\{\beta^{\frac{-2}{1+\alpha}}+b^{\frac{-2}{1+\alpha}}-2\beta^{\frac{-1}{1+\alpha}}b^{\frac{-1}{1+\alpha}}\cos(\phi+\beta-\Phi-b)\}^\frac{\gamma}{2}}\\
&=e^{ik\phi}e^{ik\beta}\int_0^{2\pi}\int_0^\infty \frac{ e^{-ikb}b^{-\frac{\gamma}{1+\alpha}}b^{\frac{-2}{1+\alpha}}b h_k(b)e^{ik(\Phi+b-\phi-\beta)}dbd\Phi}{\{\beta^{\frac{-2}{1+\alpha}}+b^{\frac{-2}{1+\alpha}}-2\beta^{\frac{-1}{1+\alpha}}b^{\frac{-1}{1+\alpha}}\cos(\phi+\beta-\Phi-b)\}^\frac{\gamma}{2}}\\
&=e^{ik\phi}e^{ik\beta}\int_0^{2\pi}\int_0^\infty \frac{ e^{-ikb}b^{-\frac{\gamma}{1+\alpha}}b^{\frac{-2}{1+\alpha}}b h_k(b)\cos(k\eta)dbd\eta}{\{\beta^{\frac{-2}{1+\alpha}}+b^{\frac{-2}{1+\alpha}}-2\beta^{\frac{-1}{1+\alpha}}b^{\frac{-1}{1+\alpha}}\cos(\eta)\}^\frac{\gamma}{2}}
\end{align*}
Hence, coming back to \eqref{partialfF-1-2-3} we find
\begin{align*}
&\partial_f\tilde{F}(0,1)(\beta,\phi)=\sum_{k\in \Z}e^{ik\phi}\left\{\frac{(1+\alpha)}{2}\partial_\beta(\beta h_k(\beta))\right.\\
&\left.-ik \frac{1}{\mu}\frac{C_\gamma }{4\pi C_0}\left(\frac{1}{\alpha-1}\right)e^{ik\beta}\beta^{-\frac{\alpha-1}{1+\alpha}}\int_0^{2\pi}\int_0^\infty \frac{ e^{-ikb}b^{-\frac{\gamma}{1+\alpha}}b^{\frac{-2}{1+\alpha}}b h_k(b)\cos(k\eta)dbd\eta}{\{\beta^{\frac{-2}{1+\alpha}}+b^{\frac{-2}{1+\alpha}}-2\beta^{\frac{-1}{1+\alpha}}b^{\frac{-1}{1+\alpha}}\cos(\eta)\}^\frac{\gamma}{2}}\right\}.\nonumber
\end{align*}
\end{proof}

\section{Invertibility of the linear operator}\label{sec-spectral}
In this section, we shall study the invertibility of the linearized operator at the trivial solution in order to later apply the infinite dimensional Implicit Function theorem. From Proposition \ref{prop-fredholm} we know that the linear operator is a compact perturbation of a isomorphism and thus it is a Fredholm operator of zero index. In that way, it reduces to check that the kernel is trivial. In order to study the kernel equation, we relate this problem with the moment problem by using the Mellin and Laplace transform. More specifically, once we apply the Mellin transform to the kernel equation, we arrive to a recurrence equation with singular coefficients. However, we can overcome this by studying the special structure of such recurrence equation. Later, we shall relate the solution of the recurrence equation with the moment problem by using the Laplace transform, concluding that the solution is singular and thus it does not belong to our function space. That will amount to a trivial kernel.

\subsection{Kernel study}
Note that the kernel study reduces to the following problem.
\begin{equation}\label{kernel-eq}
\partial_f F(0,1)h=0, h\in X_1^{\sigma,\delta} \Rightarrow h=0.
\end{equation}
Since we can restrict ourselves to work with the linearized operator as a Fourier series, hence we can directly check that the linear operator at the mode $k$ is trivial, for any $k\in\Z$. The linearized operator is given by Proposition \ref{prop-lin-op-cancelation-2} and can be written as 
\begin{align*}
&\partial_f\tilde{F}(0,1)(\beta,\phi)=\sum_{k\in \Z}e^{ik\phi}\left\{\frac{(1+\alpha)}{2}\partial_\beta(\beta h_k(\beta))\right.\\
&\left.-ik \frac{1}{\mu}\frac{C_\gamma }{4\pi C_0}\left(\frac{1}{\alpha-1}\right)e^{ik\beta}\beta^{-\frac{\alpha-1}{1+\alpha}}\int_0^{2\pi}\int_0^\infty \frac{ e^{-ikb}b^{-\frac{\gamma}{1+\alpha}}b^{\frac{-2}{1+\alpha}}b h_k(b)\cos(k\eta)dbd\eta}{\{\beta^{\frac{-2}{1+\alpha}}+b^{\frac{-2}{1+\alpha}}-2\beta^{\frac{-1}{1+\alpha}}b^{\frac{-1}{1+\alpha}}\cos(\eta)\}^\frac{\gamma}{2}}\right\}.\nonumber\\
=&\sum_{k\in \Z}e^{ik\phi}e^{ik\beta}\left\{\frac{(1+\alpha)}{2}e^{-ik\beta}\partial_\beta(\beta h_k(\beta))-ik J_{k}(e^{-ikb}h_k)(\beta)\right\},\nonumber
\end{align*}
where
$$
J_{k}(g_k)(\beta):=\frac{1}{\mu}\frac{C_\gamma }{4\pi C_0}\left(\frac{1}{\alpha-1}\right)\beta^{-\frac{\alpha-1}{1+\alpha}}\int_0^{2\pi}\int_0^\infty \frac{ b^{-\frac{\gamma}{1+\alpha}}b^{\frac{-2}{1+\alpha}}b g_k(b)\cos(k\eta)dbd\Phi}{\{\beta^{\frac{-2}{1+\alpha}}+b^{\frac{-2}{1+\alpha}}-2\beta^{\frac{-1}{1+\alpha}}b^{\frac{-1}{1+\alpha}}\cos(\eta)\}^\frac{\gamma}{2}},
$$
and
$$
g_k(\beta):= e^{-ik\beta}h_k(\beta).
$$
Define also
\begin{equation}\label{Lkm}
\mathcal{L}_{k}(h)(\beta):=\frac{(1+\alpha)}{2}e^{-ik\beta}\partial_\beta(\beta h_k(\beta))-ik J_{k}(e^{-ikb}h_k)(\beta),
\end{equation}
and we wish to prove that $\mathcal{L}_{k}$ has a trivial kernel for each $k\in\Z$.  For $h\in X_1^{\sigma,\delta}$ one has that $h_k\in \mathcal{X}_0$ defined in \eqref{X01} where
\begin{align*}
\mathcal{X}_0=&\left\{f:\beta\in[0,\infty)\rightarrow \R,\quad f,\beta f'\in\mathcal{X}^{\sigma,\delta}\right\}.
\end{align*}
Hence, we will check that the kernel of $\mathcal{L}_{k}$ is trivial in the larger space $\mathcal{X}_0$ for any $k\in\Z$, which implies \eqref{kernel-eq}.

The case $k=0$ is straightforward since
\begin{equation*}
\mathcal{L}_{0}(h)(\beta):=\frac{(1+\alpha)}{2}e^{-ik\beta}\partial_\beta(\beta h_k(\beta)),
\end{equation*}
which has a trivial kernel.

In the following proposition we use an auxiliary function $g_k$ to work with \eqref{Lkm}.
\begin{pro}\label{prop-kernel-mellin-1}
Let $g_k(\beta)=e^{-ik\beta}h_k(\beta)$ then $\mathcal{L}_{k}=0$, defined in \eqref{Lkm}, agrees with
\begin{align*}
&\tilde{\mathcal{L}}_{k}(g_k):=\frac{1+\alpha}{2}g_k(\beta)+ik \frac{1+\alpha}{2}\beta g_k(\beta)+\frac{1+\alpha}{2}\beta g_k'(\beta)-ikJ_{k}(g_k)(\beta)=0,
\end{align*}
where
\begin{equation}\label{J-k}
J_{k}(\beta)=\int_0^\infty g_k(b)K_{k}(\beta/b)db,
\end{equation}
and
$$
K_{k}(z):=\frac{C_\gamma}{\mu 4\pi C_0(\alpha-1)}\int_0^{2\pi}\frac{z^{-\frac{\alpha-1}{1+\alpha}}\cos(k\eta)d\eta}{\{z^{\frac{-2}{1+\alpha}}+1-2z^{\frac{-1}{1+\alpha}}\cos(\eta)\}^\frac{\gamma}{2}}.
$$
\end{pro}
\begin{proof}
Just note that 
$$
e^{ik\beta}\beta h_k'(\beta)=\beta g_k'(\beta)+ik\beta g_k(\beta),
$$
and use the expression of $J_{k}$.
\end{proof}

Note that $K_{-k}=K_k$ and hence although we will work with $k\in\N^\star$, we can always exchange $k$ by $-k$. In the following, we give the monotonicity of $K_k$ with respect to $k$.

\begin{pro}\label{pro-monotonicity-K}
For any $z\in (0,\infty)$, the sequence $k\in\N^\star \mapsto K_k(z)$ is strictly decreasing.
\end{pro}

\begin{proof}
Let us define the auxiliary function
\begin{equation}\label{T-k}
T_{k}(z):=\int_0^{2\pi}\frac{\cos(k\eta)d\eta}{\{z^2+1-2z\cos(\eta)\}^\frac{\gamma}{2}},
\end{equation}
and then $K_k$ can be written as
$$
K_{k}(z)=\frac{C_\gamma}{\mu 4\pi C_0(\alpha-1)}z^{-\frac{\alpha-1}{1+\alpha}} T_{k}(z^{\frac{-1}{1+\alpha}}).
$$
Hence, in order to study the monotonicity of $k\mapsto K_k(z)$ with $z\in(0,\infty)$, it reduces to study the monotonicity of $k\mapsto T_k(z)$ with $z\in(0,\infty)$. Moreover, by using Lemma \ref{lem-integral} we can write $T_k$ in terms of the Gauss Hypergeometric function as follows
\begin{equation}\label{T-k-hyp}
T_k(z)=\frac{2\pi \left(\frac{\gamma}{2}\right)_k 2^k \left(\frac12\right)_k}{(2k)!}\frac{1}{(2z)^\frac{\gamma}{2}}\left(\frac{2z}{(1+z)^2}\right)^{k+\frac{\gamma}{2}}{}_2F_1\left(k+\frac{\gamma}{2},k+\frac12,2k+1,\frac{4z}{(1+z)^2}\right).
\end{equation}
Denoting
$$
x:=\frac{4z}{(1+z)^2},
$$
where note that $x\in[0,1]$, and by virtue of the integral representation of the Gauss Hypergeometric function \eqref{hypergeom-integral} we obtain
\begin{align*}
T_k(z)=&\frac{2\pi \left(\frac{\gamma}{2}\right)_k 2^k \left(\frac12\right)_k}{(2k)!2^{k+\frac{\gamma}{2}}}\frac{1}{(2z)^\frac{\gamma}{2}}x^{k+\frac{\gamma}{2}}{}_2F_1\left(k+\frac{\gamma}{2},k+\frac12,2k+1,x\right)\\
=&\frac{2\pi \left(\frac{\gamma}{2}\right)_k 2^k \left(\frac12\right)_k}{(2k)!2^{k+\frac{\gamma}{2}}}\frac{1}{(2z)^\frac{\gamma}{2}}x^{k+\frac{\gamma}{2}}\frac{(2k)!}{\Gamma(k+\frac{\gamma}{2})\Gamma(k+1-\frac{\gamma}{2})}\int_0^1 t^{k+\frac{\gamma}{2}-1}(1-t)^{k-\frac{\gamma}{2}}(1-xt)^{-k-\frac12}dt\\
=&\frac{2\pi }{\Gamma(\frac{\gamma}{2})\Gamma(\frac12)2^\gamma z^\frac{\gamma}{2}}\frac{\Gamma(\frac12+k)}{\Gamma(k+1-\frac{\gamma}{2})}x^{k+\frac{\gamma}{2}}\int_0^1 t^{k+\frac{\gamma}{2}-1}(1-t)^{k-\frac{\gamma}{2}}(1-xt)^{-k-\frac12}dt\\
=:&\frac{2\pi }{\Gamma(\frac{\gamma}{2})\Gamma(\frac12)2^\gamma z^\frac{\gamma}{2}}T^1_kT^2_k(x)T^3_k(x),
\end{align*}
with
\begin{align*}
T^1_k:=&\frac{\Gamma(\frac12+k)}{\Gamma(k+1-\frac{\gamma}{2})},\\
T^2_k(x):=&x^{k+\frac{\gamma}{2}},\\
T^3_k(x):=&\int_0^1 t^{k+\frac{\gamma}{2}-1}(1-t)^{k-\frac{\gamma}{2}}(1-xt)^{-k-\frac12}dt.
\end{align*}
Note that $T^1_k, T^2_k, T^3_k$ are positive functions for $x\in[0,1]$. Moreover, $k\mapsto T^2_k(x)$ is strictly decreasing for $x\in(0,1)$. The last function $T^3_k$ can be written as
$$
T^3_k(x)=\int_0^1 t^{\frac{\gamma}{2}-1}(1-t)^{-\frac{\gamma}{2}}(1-xt)^{-\frac12}\left(\frac{t(1-t)}{1-xt}\right)^kdt,
$$
obtaining that $k\mapsto T^3_k(x)$ is strictly decreasing for $x\in(0,1)$ since $\frac{t(1-t)}{1-xt}\in(0,1)$. For the first term $T^1_k$ let us differentiate in $k$ obtaining
$$
\partial_k T^1_k=\frac{\Gamma'(\frac12+k)\Gamma(k+1-\frac{\gamma}{2})-\Gamma(\frac12+k)\Gamma'(k+1-\frac{\gamma}{2})}{\Gamma(k+1-\frac{\gamma}{2})^2}.
$$
We claim that 
\begin{equation}\label{claim-0}\partial_k T^1_k<0.\end{equation} Indeed, that condition agrees with
$$
\Gamma'(\frac12+k)\Gamma(k+1-\frac{\gamma}{2})<\Gamma(\frac12+k)\Gamma'(k+1-\frac{\gamma}{2}),
$$
which agrees with
$$
\psi^0(\frac12+k)=\frac{\Gamma'(\frac12+k)}{\Gamma(\frac12+k)}<\frac{\Gamma'(k+1-\frac{\gamma}{2})}{\Gamma(k+1-\frac{\gamma}{2})}=\psi^0(k+1-\frac{\gamma}{2}),
$$
where $\psi^0$ is the digamma function which is increasing. Hence, our claim \eqref{claim-0} holds true. As a consequence, we get that $k\mapsto T^1_kT^2_k(x)T^3_k(x)$ strictly decreases for $x\in(0,1)$, obtaining the announced result. 
\end{proof}

\subsection{Mellin transform}
Mellin transforms are very useful when dealing with integral differential equations where the integral operator is of the type \eqref{J-k} due to the product property \eqref{mellin-1}:
\begin{equation}\label{mellin-J}
M[J_k](s)=M[K](s)M[g_k](s+1).
\end{equation}
Hence, applying the Mellin transform to the expression given in Proposition \ref{prop-kernel-mellin-1} we obtain the following equivalent equation.

\begin{pro}
The equation $\tilde{\mathcal{L}}_{k}(g_k)=0$ agrees with 
\begin{align}\label{eq-mellin-2}
M[\tilde{\mathcal{L}}_{k}(g_k)](s)=\frac{1+\alpha}{2}\left(1-s\right)G_k(s)-ik\left(\tilde{K}_{k}(s)-\frac{1+\alpha}{2}\right) G_k(s+1)=0,
\end{align}
where $G_k=M[g_k]$ and $\tilde{K}_{k}=M[K_{k}]$.
\end{pro}
\begin{proof}
The proof is based on the properties \eqref{mellin-1}--\eqref{mellin-3} together with \eqref{mellin-J}.
\end{proof}
\begin{rem}
Given $h_k\in \mathcal{X}_0^{\sigma,\delta}$, then $M[g_k]$ is well defined in $$\{s\in\C, \quad -\sigma<\textnormal{Re}(s)<\sigma\}.$$
\end{rem}

\begin{rem}
Let us again remark that the kernel $K_k$ is symmetric with respect to $k$: $K_{-k}=K_k$. Hence, in order to work with $k\in \Z/\N$ we can do the change $k\mapsto -k$. In that way, the kernel equation \eqref{eq-mellin-2} reads as
\begin{align*}
M[\tilde{\mathcal{L}}_{k}(g_k)](s)=\frac{1+\alpha}{2}\left(1-s\right)G_k(s)+ik\left(\tilde{K}_{k}(s)-\frac{1+\alpha}{2}\right) G_k(s+1)=0,
\end{align*}
with $k\in\N^\star$.
\end{rem}


In that way, Equation \eqref{eq-mellin-2} can be written as
\begin{equation}\label{kernel-eq-2}
G_k(s+1)=\frac{1}{k} F_{k}(s) G_k(s),
\end{equation}
where
$$
F_{k}(s):=i\frac{1-s}{(1-\frac{2}{1+\alpha}\tilde{K}_{k}(s))}.
$$
In the following, let us give a better expression for $F_{k}$ together with some properties. As a consequence, we shall look for the values of $s$ such that \eqref{kernel-eq-2} is well-defined.

\begin{pro}\label{prop-mellinK-2}
The Mellin transform of $K_k$ can be written as follows
$$
\tilde{K}_{k}(s)=\frac{C_\gamma(1+\alpha)}{\mu 4\pi C_0(\alpha-1)} M[T_{k}(z)](-(1+\alpha)s+\alpha-1),
$$
for 
$$
\frac{\alpha-1-k-\gamma}{1+\alpha}<s<\frac{\alpha-1+k}{1+\alpha},
$$
and where $T_k$ is defined in \eqref{T-k}. Moreover, 
$$
\lim_{s\rightarrow \frac{\alpha-1-k-\gamma}{1+\alpha}}\tilde{K}_{k}(s)=\lim_{s\rightarrow \frac{\alpha-1+k}{1+\alpha}}\tilde{K}_{k}(s)=+\infty,
$$
for $k\in\N^\star$.
\end{pro}
\begin{proof}
Recall the expression for $K_k$:
$$
K_{k}(z)=\frac{C_\gamma}{\mu 4\pi C_0(\alpha-1)}\int_0^{2\pi}\frac{z^{-\frac{\alpha-1}{1+\alpha}}\cos(k\eta)d\eta}{\{z^{\frac{-2}{1+\alpha}}+1-2z^{\frac{-1}{1+\alpha}}\cos(\eta)\}^\frac{\gamma}{2}}.
$$
Using $T_k$ in \eqref{T-k}  we have
$$
K_{k}(z)=\frac{C_\gamma}{\mu 4\pi C_0(\alpha-1)}z^{-\frac{\alpha-1}{1+\alpha}} T_{k}(z^{\frac{-1}{1+\alpha}}).
$$
Now, let us use some properties about the Mellin transform described in Appendix \ref{ap-special}. By \eqref{mellin-2} we get
$$
M[K_{k}](s)=\tilde{K}_{k}(s)=\frac{C_\gamma}{\mu 4\pi C_0(\alpha-1)} M[T_{k}(z^{\frac{-1}{1+\alpha}})]\left(s-\frac{\alpha-1}{1+\alpha}\right).
$$
Moreover, by \eqref{mellin-5} one finds
$$
M[T_{k}(z^{\frac{-1}{1+\alpha}})](s)=(1+\alpha)M[T_{k}(z)](-(1+\alpha)s),
$$
and hence
$$
M[T_{k}(z^{\frac{-1}{1+\alpha}})](s-\frac{\alpha-1}{1+\alpha})=(1+\alpha)M[T_{k}(z)](-(1+\alpha)s+\alpha-1),
$$
which implies
\begin{equation}\label{Ktilde-3}
\tilde{K}_{k}(s)=\frac{C_\gamma(1+\alpha)}{\mu 4\pi C_0(\alpha-1)} M[T_{k}(z)](-(1+\alpha)s+\alpha-1).
\end{equation}

Let us check where $M[T_{k}(z)](s)$ is well-defined. Since $T_k$ can be related to the Gauss Hypergeometric function by using Lemma \ref{lem-integral}, we get
\begin{align*}
M[T_{k}](s)=&\int_0^\infty \int_0^{2\pi}\frac{z^{s-1}\cos(k\eta)d\eta dz}{\left\{z^2+1-2z\cos(\eta)\right\}^\frac{\gamma}{2}}\\
=&\frac{2\pi 2^{2k}\left(\frac{\gamma}{2}\right)_{k}\left(\frac{1}{2}\right)_{k}}{(2k)!}\int_0^\infty \frac{z^{s-1+k}}{(1+z)^{\gamma+2k}}F\left(k+\frac{\gamma}{2},k+\frac12,2k+1,\frac{4z}{(1+z)^2}\right)dz.
\end{align*}
Since the parameters of the Gauss Hypergeometric function are positive, this function increases and then  for any $k\in\N^\star$, $\gamma\in(0,1)$ and $z\in(0+\infty)$ one has
$$
1<F\left(k+\frac{\gamma}{2},k+\frac12,2k+1,\frac{4z}{(1+z)^2}\right)<F\left(k+\frac{\gamma}{2},k+\frac12,2k+1,1\right)<+\infty.
$$
Hence $M[T_{k}](s)$ is well-defined for $s\in (-k,k+\gamma)$, where that condition comes form the integrability of the integral. Moreover
\begin{align}\label{MT-bound}
M[T_{k}](s)\geq \frac{2\pi 2^{2k}\left(\frac{\gamma}{2}\right)_{k}\left(\frac{1}{2}\right)_{k}}{(2k)!}\int_0^\infty \frac{z^{s-1+k}}{(1+z)^{\gamma+2k}}dz,
\end{align}
implying
$$
\lim_{s\rightarrow -k}T_{k}(s)=\lim_{s\rightarrow k+\gamma}T_{k}(s)=+\infty.
$$

Coming back to $\tilde{K}$ through \eqref{Ktilde-3}, we get that $\tilde{K}$ is well-defined if
$$
\frac{\alpha-1-k-\gamma}{1+\alpha}<s<\frac{\alpha-1+k}{1+\alpha},
$$
and 
$$
\lim_{s\rightarrow \frac{\alpha-1-k-\gamma}{1+\alpha}}\tilde{K}(s)=\lim_{s\rightarrow \frac{\alpha-1+k}{1+\alpha}}\tilde{K}(s)=+\infty.
$$
\end{proof}

As a consequence of Proposition \ref{pro-monotonicity-K} we get the monotonicity of the Mellin transform of $K$, that is, $\tilde{K}$:
\begin{cor}\label{Cor-monotonicity-k}
For any $s\in[0,\frac{\alpha-1+k}{1+\alpha})$, the sequence $k\in \N^\star\mapsto \tilde{K}_k(s)$ is strictly decreasing.
\end{cor}

In the following proposition, we give the monotonicity of $\tilde{K}_k$ with respect to $s$.

\begin{pro}\label{Ktilde-properties}
The function $s\in[0,\frac{\alpha-1+k}{1+\alpha})\mapsto \tilde{K}_{k}(s)$ increases, for any $k\in\N^\star$. Moreover, $\frac{2}{1+\alpha}\tilde{K}_{k}(0)<1$.
\end{pro}
\begin{proof}

By the expression of $\tilde{K}$ in Proposition \ref{prop-mellinK-2}, we can relate $\tilde{K}_k$ with $M[T_k]$. Hence if $s\in(-k,\alpha-1]\mapsto M[T_k](s)$ decreases, we achieve that $s\in[0,\frac{\alpha-1+k}{1+\alpha})\mapsto \tilde{K}_{k}(s)$ increases. Moreover, using the explicit expression of $T_k$ we get that $T_{k}(z^{-1})=z^\gamma T_{k}(z)$. In that way, $M[T_k]$ can be written as follows

\begin{align*}
M[T_k](s)=&\int_0^\infty z^{s-1}T_k(z)dz\\
=&\int_0^1 z^{s-1}T_k(z)dz+\int_1^\infty z^{s-1}T_k(z)dz\\
=&\int_0^1 T_k(z)(z^{s-1}+z^{1-s}z^\gamma z^{-2})dz.
\end{align*}
Then, computing its derivative we observe
\begin{align*}
M[T_k]'(s)
=&\int_0^1 \ln(z) T_k(z)(z^{s-1}-z^{1-s}z^\gamma z^{-2})dz\\
=&\int_0^1 \ln(z) T_k(z)z^{s-1}(1-z^{\gamma-2s})dz.
\end{align*}
Since $s\in(-k,\alpha-1)$ and $\alpha$ satisfies \eqref{sigma-rel}, we get that $\gamma-2s>0$ and hence $M[T_k]'(s)<0$. As a consequence $M[T_k]$ is strictly decreasing.

About the second point of the statement and by virtue of Corollary \ref{Cor-monotonicity-k}, it is enough to prove that $\frac{2}{1+\alpha}\tilde{K}_1(0)<1$. Moreover, let us check that $\alpha\mapsto \frac{2}{1+\alpha}\tilde{K}_1(0)$ strictly decreases and hence we can restrict ourselves to $\alpha=1$. From Proposition \ref{prop-kernel-mellin-1} we get
\begin{align*}
\frac{2}{1+\alpha}\tilde{K}_1(0)=\frac{C_\gamma}{\mu 2\pi C_0(\alpha-1)}M[T_1](\alpha-1).
\end{align*}
Moreover, from Proposition \ref{prop-C0-2} and the expression of $\mu$ in \eqref{mu} we get 
\begin{align*}
\frac{2}{1+\alpha}\tilde{K}_1(0)=\frac{(1+\alpha-\gamma)}{(\alpha-1) M[T_0](1-\alpha+\gamma)}M[T_1](\alpha-1).
\end{align*}

By the previous estimates concerning the monotonicity of $s\mapsto M[T_k](s)$, we have that $\alpha\in(1,1+\gamma)\mapsto M[T_1](\alpha-1)$ is strictly decreasing. Moreover, one has that $M[T_0]'(1-\alpha+\gamma)<0$ by similar arguments. Hence, we claim that
$$
g(\alpha):=\frac{(1+\alpha-\gamma)}{(\alpha-1) M[T_0](1-\alpha+\gamma)},
$$
is positive and strictly decreasing. The positivity comes from the fact that $M[T_0]$ is positive. In order to prove the monotonicity, let us differentiate:
\begin{align*}
g'(\alpha)=&\frac{(\alpha-1) M[T_0](1-\alpha+\gamma)-(1+\alpha-\gamma)M[T_0](1-\alpha+\gamma)+(\alpha-1)(1+\alpha-\gamma)M[T_0]'(1-\alpha+\gamma)}{(\alpha-1) M[T_0](1-\alpha+\gamma)^2}\\
=&\frac{-(2-\gamma) M[T_0](1-\alpha+\gamma)+(\alpha-1)(1+\alpha-\gamma)M[T_0]'(1-\alpha+\gamma)}{(\alpha-1) M[T_0](1-\alpha+\gamma)^2},
\end{align*}
which is negative since $\gamma<1$. Hence, we have proved our claim. Then
$$
\frac{2}{1+\alpha}\tilde{K}_1(0)\leq \lim_{\alpha\rightarrow 1^+} \frac{(1+\alpha-\gamma)}{(\alpha-1) M[T_0](1-\alpha+\gamma)}M[T_1](\alpha-1)= \lim_{\alpha\rightarrow 1^+} \frac{(2-\gamma)}{(\alpha-1) M[T_0](1-\alpha+\gamma)}M[T_1](0).
$$
Note that 
$$
\lim_{\alpha\rightarrow 1^+}M[T_0](1-\alpha+\gamma)=+\infty,
$$
and hence there is a competition in the denominator that we will have to solve. Note that we can again relate $M[T_0]$ with the Gauss Hypergeometric function:
\begin{align*}
M[T_0](1-\alpha+\gamma)=&2\pi \int_0^\infty \frac{z^{\gamma-\alpha}}{(1+z)^{\gamma}}{}_2F_1\left(\frac{\gamma}{2},\frac12,1,\frac{4z}{(1+z)^2}\right)dz\\
> &2\pi \int_0^\infty \frac{z^{\gamma-\alpha}}{(1+z)^{\gamma}}dz\\
=&2\pi  \frac{\Gamma(1+\gamma-\alpha)\Gamma(\alpha-1)}{\Gamma(\gamma)}.
\end{align*}
Since $x\Gamma(x)=\Gamma(x+1)$ we get
$$
\lim_{\alpha\rightarrow 1}(\alpha-1)\Gamma(\alpha-1)=1,
$$
and hence
$$
\lim_{\alpha\rightarrow 1^+} \frac{(2-\gamma)}{(\alpha-1) M[T_0](1-\alpha+\gamma)}M[T_1](0)< \frac{(2-\gamma)}{2\pi}M[T_1](0).
$$
Using now the expression of $T_1$ in terms of the Gauss Hypergeometric function \eqref{T-k-hyp} we find
\begin{align*}
M[T_1](0)=\pi \gamma \int_0^\infty \frac{1}{(1+z)^{2+\gamma}}{}_2F_1\left(1+\frac{\gamma}{2},\frac32, 3, \frac{4z}{(1+z)^2}\right)dz.
\end{align*}
By writing the Gauss Hypergeometric function in series, integrating and summing again we find the following expression
\begin{align*}
M[T_1](0)=\pi \gamma \frac{\Gamma(-\frac{\gamma}{2})\Gamma(\frac{\gamma-1}{2})}{\Gamma(\frac{1-\gamma}{2})\Gamma(\frac{\gamma}{2})}.
\end{align*}
Using that
$$
\Gamma(z)=-\frac{\Gamma(-z)\Gamma(1+z)}{\Gamma(1-z)},
$$
we find
\begin{align*}
M[T_1](0)=\pi \gamma \frac{\Gamma(1-\frac{\gamma}{2})\Gamma(\frac{\gamma+1}{2})}{\Gamma(1+\frac{\gamma}{2})\Gamma(\frac32-\frac{\gamma}{2})},
\end{align*}
and hence
$$
\lim_{\alpha\rightarrow 1^+} \frac{(2-\gamma)}{(\alpha-1) M[T_0](1-\alpha+\gamma)}M[T_1](0)< \frac{(2-\gamma)\gamma}{2} \frac{\Gamma(1-\frac{\gamma}{2})\Gamma(\frac{\gamma+1}{2})}{\Gamma(1+\frac{\gamma}{2})\Gamma(\frac32-\frac{\gamma}{2})}=:p(\gamma).
$$
Note that $p(1)=1$ and if we check that $p$ is strictly increasing, we can arrive to the announced result given that:
$$
\lim_{\alpha\rightarrow 1^+} \frac{(2-\gamma)}{(\alpha-1) M[T_0](1-\alpha+\gamma)}M[T_1](0)< \frac{(2-\gamma)\gamma}{2} \frac{\Gamma(1-\frac{\gamma}{2})\Gamma(\frac{\gamma+1}{2})}{\Gamma(1+\frac{\gamma}{2})\Gamma(\frac32-\frac{\gamma}{2})}<p(1)=1,
$$
for any $\gamma\in(0,1)$.

Hence, it remains to check that $p$ is strictly increasing. Note that $p$ can be written as
$$
p(\gamma)=2\frac{\Gamma(2-\frac{\gamma}{2})\Gamma(\frac{1+\gamma}{2})}{\Gamma(\frac{3-\gamma}{2})\Gamma(\frac{\gamma}{2})},
$$
and whose derivative reads as follows
$$
p'(\gamma)=\frac{\Gamma(2-\frac{\gamma}{2})\Gamma(\frac{1+\gamma}{2})}{\Gamma(\frac{3-\gamma}{2})\Gamma(\frac{\gamma}{2})}\left\{\psi^0(\frac{3-\gamma}{2})-\psi^0(2-\frac{\gamma}{2})-\psi^0(\frac{\gamma}{2})+\psi^0(\frac{1+\gamma}{2})\right\},
$$
where $\psi^0$ is the digamma function. Now, using the integral representation of the digamma function \eqref{digamma-int} we find
$$
\psi^0(\frac{3-\gamma}{2})-\psi^0(2-\frac{\gamma}{2})-\psi^0(\frac{\gamma}{2})+\psi^0(\frac{1+\gamma}{2})=\int_0^1\frac{-t^{\frac{1-\gamma}{2}}-t^{-\frac{1-\gamma}{2}}+t^{\frac{2-\gamma}{2}}+t^{-\frac{2-\gamma}{2}}}{1-t}dt,
$$
where we can prove that the function inside the integral is positive for any $t\in(0,1)$ since $\gamma\in(0,1)$. Hence we obtain $p'(\gamma)>0$, for any $\gamma\in(0,1)$, and the claim holds true. That concludes the proof.
\end{proof}

As a consequence of Proposition \ref{prop-mellinK-2} and Proposition \ref{Ktilde-properties}, and using the intermediate theorem, we get the following corollary:
\begin{cor}
For any $k\in\N^\star$, there exists a unique $s_0\in[0,\frac{\alpha-1+k}{1+\alpha})$ such that $\tilde{K}_{k}(s_0)=1$.
\end{cor}

Hence, using the expression of $\tilde{K}$ in Proposition \ref{prop-mellinK-2} we find
$$
(1-s)G_k(s)+ik\left(1-\frac{C_\gamma}{\mu 2\pi C_0(\alpha-1)} M[T_{k}(z)](-(1+\alpha)s+\alpha-1)\right)G_k(s+1)=0,
$$
amounting to
\begin{equation}\label{MellinT-1}
G_k(s+1)=\frac{i}{k}\frac{1-s}{1-\frac{C_\gamma}{\mu 2\pi C_0(\alpha-1)} M[T_{k}(z)](-(1+\alpha)s+\alpha-1)}G_k(s).
\end{equation}
Note that by Proposition \ref{prop-mellinK-2} $M[T_{k}(z)](-(1+\alpha)s+\alpha-1)$ is well-defined for 
$$
\frac{\alpha-1-k-\gamma}{1+\alpha}<s<\frac{\alpha-1+k}{1+\alpha},
$$
Using $F_{k}$ one has
$$
F_{k}(s)=i\frac{1-s}{1-\frac{C_\gamma}{\mu 2\pi C_0(\alpha-1)} M[T_{k}(z)](-(1+\alpha)s+\alpha-1)}.
$$

Moreover, due to the limits of $\tilde{K}$ in Proposition \ref{prop-mellinK-2}, we get that
$$
\lim_{s\rightarrow \frac{\alpha-1-k-\gamma}{1+\alpha}}F_{k}(s)=\lim_{s\rightarrow \frac{\alpha-1+k}{1+\alpha}}F_{k}(s)=0,
$$
and hence we can extend it to zero outside $(\frac{\alpha-1-k-\gamma}{1+\alpha},\frac{\alpha-1+k}{1+\alpha})$. Moreover, due to \eqref{sigma-rel} we have that
$$
\frac{\alpha-1-\gamma-k}{1+\alpha}<\frac{\alpha-1-\gamma}{1+\alpha}<-\sigma<s<\sigma<\frac{\alpha-1}{1+\alpha}<\frac{\alpha-1+k}{1+\alpha}.
$$
Since a priori $G_k$ is well-defined for $s\in(-\sigma,\sigma)$, one  obtains that \eqref{MellinT-1} is well-defined for $s\in(-\sigma,\sigma)$, and then $G_k$ is well-defined for $s\in(-\sigma+1, \sigma+1)$.

Finally, in the following proposition we achieve that the kernel is trivial. For this goal, we observe \eqref{MellinT-1} as a recurrence equation and we relate the solutions to the moment problem. In that way, using the Laplace transform we achieve that the solution $h_k$ is singular and thus it does not belong to our function spaces, concluding that the kernel is trivial.

\begin{pro}\label{prop-kernel}
If $G_k$ solves \eqref{MellinT-1}, with $g_k(\beta)=e^{-ik\beta}h_k(\beta)$ and $h_k\in \mathcal{X}_0^{\sigma,\delta}$, hence $h_k=0$.
\end{pro}

\begin{proof}
Taking $s=M\in\N^\star$, we can solve the recurrence equation \eqref{MellinT-1} as
$$
G_k(M)=\frac{1}{k}G_k(0)\prod_{n=1}^{M-1} F_{k}(n).
$$
Note that $G_k(0)$ is well-defined and then
$$
G_k(1)=\frac{1}{k}F_k(0)G_k(0),
$$
which is finite due to Proposition \ref{Ktilde-properties}.

First, assume that $\frac{2}{1+\alpha}\tilde{K}_{k}(N)\neq 1$, for any $N\in\N$ and $N\leq \frac{\alpha-1+k}{1+\alpha}$. In that case, since $F_k(1)=0$ we get $G_k(M)=0$ for any $M\geq 2$. Let $\mathcal{L}$ be the Laplace Transform operator, and thus
\begin{align*}
\mathcal{L}[g_k](t)=&\int_0^\infty e^{-tz}g(z)dz\\
=&\sum_{n=0}^\infty \frac{(-t)^n}{n!}\int_0^\infty z^n g_(z)dz\\
=&\sum_{n=0}^\infty \frac{(-t)^n}{n!} G_k(n+1)\\
=&G_k(1).
\end{align*}
Since the inverse Laplace transform of a constant is a dirac mass and $g_k$ is continuous, we get that necessarily $g_k=0.$

Second, assume that there exists $N_0\in\N$ with $N_0\in[1,\frac{\alpha-1+k}{1+\alpha})$ such that 
$$
\frac{2}{1+\alpha}\tilde{K}_{k}(N_0)= 1.
$$
In that case $\frac{2}{1+\alpha}\tilde{K}_{k}$ has a simple root at $N_0$. On the one hand,  if $N_0=1$ we have that
$$
\lim_{s\rightarrow 1}\frac{s-1}{\frac{2}{1+\alpha}\tilde{K}_{k}(s)-1}\in\R,
$$
since the root is simple. In that case $G_k(2)$ is well-defined and we can write 
$$
G_k(M)=\frac{1}{k}G_k(0)\prod_{n=1}^{M-1} F_{k}(M).
$$
Moreover, note that for $M\geq \frac{\alpha-1+k}{1+\alpha}$ we have that $F_k(M)=0$, implying that $G_k(M)=0$ for $M\geq \frac{\alpha-1+k}{1+\alpha}$. Hence the Laplace transform of $g_k$ is a polynomial.  On the other hand, if $N_0\geq 2$, we have
$$
\lim_{s\rightarrow N_0} F_k(1) F_k(s)\in\R,
$$
since $F_k(1)=0$ and the pole of $F_k$ at $N_0$ is simple. In that case, we find again that the Laplace transform of $g_k$ is a polynomial. Now, by using the expression of the Laplace transform we find
$$
\mathcal{L}_{g_k}(t)=\int_0^\infty e^{-tz}g_k(z)dz.
$$
Since $g_k\in L^\infty([0,\infty))$ we get that
$$
\lim_{t\rightarrow +\infty}|\mathcal{L}_{g_k}(t)|=0,
$$
which contradicts the fact that the Laplace transform is a polynomial. Hence $g_k=0$.
\end{proof}

\section{$m$-fold symmetry and main theorem in the variables $(\beta,\phi)$} \label{sec-mfold}
In this section, we introduce the $m$-fold symmetry to the function spaces, and then we will work with the extra parameter $m$. Moreover, we will give the main result in the adapted coordinates which is a consequence of the infinite dimensional Implicit Function theorem applied to $\tilde{F}$.

\subsection{$m$-fold symmetry in the function spaces}
In the following, we introduce the $m$-fold symmetry in the function spaces \eqref{X1}, \eqref{X2} and \eqref{Y}:
\begin{align*}
X_{1,m}^{\sigma,\delta}:=&\left\{f\in X_1^{\sigma,\delta}, \quad f\left(\beta,\phi+\frac{2\pi}{m}\right)=f(\beta,\phi), \, (\beta,\phi)\in[0,+\infty)\times[0,2\pi]\right\},\\
X_{2,m}^p:=&\left\{f\in X_2^p, \quad f\left(\phi+\frac{2\pi}{m}\right)=f(\phi), \, \phi\in[0,2\pi]\right\},\\
Y_m^{\sigma,\delta}:=&\left\{f\in Y^{\sigma,\delta}, \quad f\left(\beta,\phi+\frac{2\pi}{m}\right)=f(\beta,\phi), \, (\beta,\phi)\in[0,+\infty)\times[0,2\pi]\right\}.
\end{align*}
Note that $f\in X_{1,m}^{\sigma,\delta}$ can be written in Fourier series as
$$
f(\beta,\phi)=\sum_{k\in\Z} f_k(\beta)e^{ikm\phi}.
$$

Since the only variation that we have introduced in the function spaces is the $m$-fold symmetry, all the previous analysis also works here. We only need to check the symmetry persistence of the nonlinear functional.

\begin{pro}\label{prop-well-def-m}
If $f\left(\beta,\phi+\frac{2\pi}{m}\right)=f(\beta,\phi)$ and $\Omega\left(\phi+\frac{2\pi}{m}\right)=\Omega(\phi)$ then $\tilde{F}(f,\Omega)\left(\beta,\phi+\frac{2\pi}{m}\right)=\tilde{F}(f,\Omega)(\beta,\phi)$.
\end{pro}
\begin{proof}
Let us write $\tilde{F}(f,\Omega)(\beta,\phi+\frac{2\pi }{m})$ using \eqref{Ftilde}:
\begin{align*}
&\tilde{F}(f,\Omega)(\beta,\phi+\frac{2\pi }{m})=1+f(\beta,\phi+\frac{2\pi }{m})\\
\nonumber&-(-1)^\frac{-1}{2\mu}\frac{C_\gamma (\alpha-1)^\frac{\alpha-1}{2}}{4\pi C_0(1+\alpha)^{\frac{\alpha-1}{2}}}\beta^{-\frac{\alpha-1}{1+\alpha}}\\
&\times \int_0^{2\pi}\int_0^\infty\frac{b^{-\frac{\gamma}{1+\alpha}}b^{\frac{-2}{1+\alpha}}\left\{-\frac{\alpha-1}{1+\alpha}+\partial_{\overline{\varphi}}f(b,\Phi)\right\}^{-\frac{1}{2\mu}}\left\{\frac{2(\alpha-1)}{(1+\alpha)^2}+\partial_{\overline{\beta\varphi}}f(b,\Phi)\right\}\tilde{\Omega}(\Phi)dbd\Phi}{D(f)(\beta,\phi+\frac{2\pi }{m},b,\Phi)^\frac{\gamma}{2}}.
\end{align*}
Doing the change of variables $\Phi\mapsto \Phi+\frac{2\pi}{m}$ we get 
\begin{align*}
&\tilde{F}(f,\Omega)(\beta,\phi+\frac{2\pi }{m})=1+f(\beta,\phi+\frac{2\pi }{m})\\
\nonumber&-(-1)^\frac{-1}{2\mu}\frac{C_\gamma (\alpha-1)^\frac{\alpha-1}{2}}{4\pi C_0(1+\alpha)^{\frac{\alpha-1}{2}}}\beta^{-\frac{\alpha-1}{1+\alpha}}\\
&\times \int_0^{2\pi}\int_0^\infty\frac{b^{-\frac{\gamma}{1+\alpha}}b^{\frac{-2}{1+\alpha}}\left\{-\frac{\alpha-1}{1+\alpha}+\partial_{\overline{\varphi}}f(b,\Phi+\frac{2\pi }{m})\right\}^{-\frac{1}{2\mu}}\left\{\frac{2(\alpha-1)}{(1+\alpha)^2}+\partial_{\overline{\beta\varphi}}f(b,\Phi+\frac{2\pi }{m})\right\}\tilde{\Omega}(\Phi+\frac{2\pi }{m})dbd\Phi}{D(f)(\beta,\phi+\frac{2\pi }{m},b,\Phi+\frac{2\pi }{m})^\frac{\gamma}{2}}.
\end{align*}
Using that $f\left(\beta,\phi+\frac{2\pi}{m}\right)=f(\beta,\phi)$ and $\Omega\left(\phi+\frac{2\pi}{m}\right)=\Omega(\phi)$, and the expression of $D(f)$ in \eqref{denominator-0} we achieve that$\tilde{F}(f,\Omega)\left(\beta,\phi+\frac{2\pi}{m}\right)=\tilde{F}(f,\Omega)(\beta,\phi)$, concluding the proof.
\end{proof}

Let us remark that the spectral study developed in Section \ref{sec-spectral} follows similarly, one just have to change $k\mapsto km$.

\subsection{Main theorem in the adapted coordinates}
Finally, let us state the main theorem in the adapted coordinates, which is a consequence of the infinite dimensional Implicit Function theorem.
\begin{theo}\label{theo-1}
Let $\gamma,\sigma\in(0,1)$ and $\alpha\in(1,1+\gamma)$ satisfying \eqref{sigma-rel}. Then, for each $m\geq 1$, there exists $\varepsilon>0$ and a $C^1$ function $\tilde{\Omega}\in B_{X_{2,m}^p}(1,\varepsilon)\mapsto \tilde{f}(\tilde{\Omega})\in X_{1,m}^{\sigma,\delta}$ such that
$$
\tilde{F}(f,\tilde{\Omega})=0, \, (f,\tilde{\Omega})\in B_{X_{1,m}^{\sigma,\delta}}(0,\varepsilon)\times B_{X_{2,m}^p}(1,\varepsilon) \Longleftrightarrow f=\tilde{f}(\Omega),
$$
for $\delta<\min\{1-\gamma,\sigma\}$ and $p>\frac{1}{1-\gamma}$. 
\end{theo}
\begin{proof}
From Proposition \ref{prop-well-def} and Proposition \ref{prop-well-def-m} we get that $\tilde{F}:B_{X_{1,m}^{\sigma,\delta}}(0,\varepsilon)\times X_{2,m}^p (1,\varepsilon)\rightarrow Y_m^{\sigma,\delta}$ is well-defined and $C^1$. Moreover, by Proposition \ref{prop-trivial} we have that $\tilde{F}(0,1)=0$. Moreover, $\partial_f \tilde{F}(0,1)$ is Fredholm of zero index by Proposition \ref{prop-fredholm} and has a trivial kernel by \eqref{prop-kernel}. Hence, $\partial_f \tilde{F}(0,1)$ is a isomorphism. The infinite dimensional Implicit Function theorem to $\tilde{F}$ concludes the proof.
\end{proof}

\begin{rem}
A priori $\Psi$ and $\Omega$ are complex functions due to the function spaces, however we are interested in real functions. That is not a problem since assuming in the function space that $f$ and $\tilde{\Omega}$ are real functions we obtain that the nonlinear operator is also a real function.
\end{rem}

\begin{rem}\label{rem-th1}
Taking $\tilde{f}$ constructed in Theorem \ref{theo-1}, we find
\begin{align*}
\Psi(\beta,\phi)=&-C_0^{\frac{2}{1+\alpha}}(\alpha-1)^{\frac{1-\alpha}{1+\alpha}}\beta^{\frac{\alpha-1}{1+\alpha}}(1+\tilde{f}(\tilde{\Omega})(\beta,\phi)),\\
\Omega(\phi)=&(1+\alpha)^{-\frac{1}{2\mu}}\tilde{\Omega}(\phi),
\end{align*}
satisfying $F(\Psi,\Omega)=0$. Moreover, we can come back to $\Theta$ and find
\begin{align*}
\Theta(\beta,\phi)=(\Psi_{\varphi})^{-\frac{1+\alpha-\gamma}{2}}\Omega(\phi),
\end{align*}
which together with $\Psi$ verifies \eqref{eq-self-sim-2}, meaning the vorticity equation in the adapted coordinates $(\beta,\phi)$. Note that $\Theta$ is just a perturbation of the trivial one. Indeed, we can compute $\Psi_{\varphi}$ as
\begin{align}\label{Psi-perturb}
\Psi_{\varphi}(\beta,\phi)=& -C_0^{\frac{2}{1+\alpha}}(\alpha-1)^{\frac{1-\alpha}{1+\alpha}}\beta^{\frac{-2}{1+\alpha}}\left(-\frac{\alpha-1}{1+\alpha}-\frac{\alpha-1}{1+\alpha}\tilde{f}(\tilde{\Omega})(\beta,\phi)+\beta \tilde{f}_\varphi(\tilde{\Omega})(\beta,\phi)\right)\nonumber\\
=& -C_0^{\frac{2}{1+\alpha}}(\alpha-1)^{\frac{1-\alpha}{1+\alpha}}\beta^{\frac{-2}{1+\alpha}}\left(-\frac{\alpha-1}{1+\alpha}+\tilde{f}_{\overline{\varphi}}(\tilde{\Omega})(\beta,\phi)\right),
\end{align}
where recall that $\tilde{f}_\varphi(\tilde{\Omega})$ decays at $0$ and $\infty$ in $\beta$. Hence $\Psi_\varphi$ is singular at $\beta=0$ as the trivial solution, and decays at $\beta\rightarrow +\infty$ as the trivial solution. Coming back to $\Theta$ we get
$$
\Theta(\beta,\phi)=(1+\alpha)^{-\frac{1}{2\mu}}\left[C_0^{\frac{2}{1+\alpha}}(\alpha-1)^{\frac{1-\alpha}{1+\alpha}}\right]^{-\frac{1+\alpha-\gamma}{2}}\beta^{\frac{1+\alpha-\gamma}{1+\alpha}}\left(\frac{\alpha-1}{1+\alpha}-\tilde{f}_{\overline{\varphi}}(\tilde{\Omega})(\beta,\phi)\right)^{-\frac{1+\alpha-\gamma}{2}}\tilde{\Omega}(\phi).
$$
Note that $1+\alpha-\gamma>0$ and hence $\Theta(0,\phi)=0$, for any $\phi\in[0,2\pi])$. In particular, we find
$$
\Theta(\beta,\phi)\sim \beta^{\frac{1+\alpha-\gamma}{1+\alpha}}, \quad \beta\sim 0,\infty.
$$

\end{rem}

\section{Recovering the solution in the original coordinates}\label{sec-recovering}
This section is devoted to recover the solution in the physical variables $x\in\R^2$ and hence give the associated Theorem \ref{theo-1} in those variables. First, we study the invertibility of the change of coordinates, which is a nonlinear implicit change of variables, and later we state the main result.

\subsection{Invertibility of the change of coordinates}\label{sec-invertibility-change}
In order to come back to the physical variable $z\in\R^2$ we should invert the adapted coordinates. From Proposition \ref{prop-coordinates} we have that
\begin{equation}\label{r-vartheta}
r(\beta,\phi)=(-(1+\alpha)\Psi_\beta)^\frac12, \quad \vartheta(\beta,\phi)=\beta+\phi.
\end{equation}
In the following proposition, let us study the regularity of $r(\beta,\phi)$ and $\vartheta(\beta,\phi)$. For that aim, let us define the following function space:
$$
L^\infty_{0,\sigma}:=\left\{g=g(0)+\tilde{g},\quad \tilde{g}\in L^\infty_\sigma\right\}.
$$

\begin{pro}
Let $(r,\vartheta):[0,\infty)\times[0,2\pi]\rightarrow [0,\infty)\times[0,2\pi]$ defined as \eqref{r-vartheta}. There exists $\varepsilon<1$ such that if $f\in B_{X_1^{\sigma,\delta}}(0,\varepsilon)$ then $\vartheta\in C^\infty([0,\infty)\times[0,2\pi])$ and
$$\beta^\frac{-1}{1+\alpha}r\in L^\infty_{0,\sigma}\cap C^\delta_\beta\cap \textnormal{Lip}_\phi,$$ and $$\beta^\frac{-1}{1+\alpha}\beta r_\varphi\in L^\infty_{\sigma}\cap C^\delta_\beta.$$
\end{pro}
\begin{proof}
First, note that $\vartheta\in C^\infty([0,+\infty)\times[0,2\pi])$. Next, from \eqref{Psi-perturb} we find
\begin{equation}\label{Psibeta}
(-\Psi_{\beta})^\frac12(\beta,\phi)=C_0^{\frac{1}{1+\alpha}}(\alpha-1)^{\frac{1-\alpha}{2(1+\alpha)}}\beta^{\frac{-1}{1+\alpha}}\left(\frac{\alpha-1}{1+\alpha}+{f}_{\overline{\beta}}(\beta,\phi)\right)^\frac12.
\end{equation}
Let us check the regularity of such function, which will give us the regularity of $r$ due to \eqref{r-vartheta}. Denote
$$
g(\beta,\phi):=\left(\frac{\alpha-1}{1+\alpha}+{f}_{\overline{\beta}}(\beta,\phi)\right)^\frac12=\left(\frac{\alpha-1}{1+\alpha}+\frac{\alpha-1}{1+\alpha}f(\beta,\phi)+\beta {f}_{{\beta}}(\beta,\phi)\right)^\frac12.
$$
Since $f\in X_1^{\sigma,\delta}$, we have that $f, \beta f_\beta\in L^\infty_\sigma\cap C^\delta_\beta\cap \textnormal{Lip}_\phi$. Hence, denoting
$$
g_0=g(0,\phi)=\left(\frac{\alpha-1}{1+\alpha}\right)^\frac12,
$$
we get that
$$
g-g_0\in L^\infty_\sigma, \quad g\in C^\delta_\beta\cap \textnormal{Lip}_\phi,
$$
and then
$$\beta^\frac{-1}{1+\alpha}r\in L^\infty_{0,\sigma}\cap C^\delta_\beta\cap \textnormal{Lip}_\phi.$$
Hence, we find the following asymptotic for $g$:
$$
g\sim_{0,\infty} g_0+\frac{\beta^\sigma}{(1+\beta)^{2\sigma}}.
$$
Coming back to $(\Psi_\beta)^\frac12$ we get that
$$
\beta^\frac{1}{1+\alpha}(\Psi_\beta)^\frac12\in L^\infty_{0,\sigma}\cap C^\delta_\beta\cap \textnormal{Lip}_\phi,
$$
with the following asymptotics:
$$
(\Psi_\beta)^\frac12 \sim_{0,\infty} \beta^{\frac{-1}{1+\alpha}}\left(1+\frac{\beta^\sigma}{(1+\beta)^{2\sigma}}\right),
$$
and then
$$
(\Psi_\beta)^\frac12 \sim_{0,\infty} \beta^{\frac{-1}{1+\alpha}}.
$$
We can also compute $\partial_\varphi g$ as
$$
\partial_\varphi g(\beta,\phi)=\frac{\frac{\alpha-1}{1+\alpha}f_\varphi(\beta,\phi)+\beta f_{\beta\varphi}(\beta,\phi)}{2\left(\frac{\alpha-1}{1+\alpha}+\frac{\alpha-1}{1+\alpha}f(\beta,\phi)+\beta {f}_{{\beta}}(\beta,\phi)\right)^\frac12}.
$$
Since $\beta f_\varphi, \beta^2 f_{\beta,\varphi}\in L^\infty_\sigma\cap C^\delta_\beta$ we obtain
$$
\beta \partial_\varphi g\in L^\infty_\sigma\cap C^\delta_\beta,
$$
implying
$$
\beta^\frac{-1}{1+\alpha}\beta \partial_\varphi (\Psi_\beta)^\frac12\in L^\infty_\sigma\cap C^\delta_\beta.
$$
Finally, one finds that $$\beta^\frac{-1}{1+\alpha}r\in L^\infty_{0,\sigma}\cap C^\delta_\beta\cap \textnormal{Lip}_\phi$$ and $$\beta^\frac{-1}{1+\alpha}\beta r_\varphi\in L^\infty_{\sigma}\cap C^\delta_\beta.$$
\end{proof}

In the next proposition, we show that if $f\in B_{X_1^{\sigma,\delta}}(0,\varepsilon)$ hence the jacobian of the transformation $(\beta,\phi)\mapsto (z_1,z_2)$ is non vanishing except from $\beta=+\infty$.

\begin{pro}\label{prop-jacobian-inv}
There exists $\varepsilon<1$ such that if $f\in B_{X_1^{\sigma,\delta}}(0,\varepsilon)$ hence the jacobian $|J|$ of the transformation $(\beta,\phi)\mapsto (z_1,z_2)$ does not vanish for $(\beta,\phi)\in(0,\infty)\times[0,2\pi]$.
\end{pro}
\begin{proof}
From Proposition \ref{prop-jacobian} we have that
$$
|J|=\frac{1+\alpha}{2}\Psi_{\beta\varphi}.
$$
Writing $\Psi$ in terms of $f$ and using \eqref{Psi-perturb} we find
$$
\Psi_{\beta\varphi}=-C_0^{\frac{2}{1+\alpha}}(\alpha-1)^{\frac{1-\alpha}{1+\alpha}}\beta^{-1}\beta^{\frac{-2}{1+\alpha}}\left\{\frac{2(\alpha-1)}{(1+\alpha)^2}+\partial_{\overline{\beta\varphi}}f(\beta,\phi)\right\}.
$$
Note that
$$
\frac{2(\alpha-1)}{(1+\alpha)^2}+\partial_{\overline{\beta\varphi}}f(\beta,\phi)\geq \frac{2(\alpha-1)}{(1+\alpha)^2}-\partial_{\overline{\beta\varphi}}f(\beta,\phi)\geq \frac{2(\alpha-1)}{(1+\alpha)^2}-||f||_{ B_{X_1^{\sigma,\delta}}}\geq \varepsilon_0>0,
$$
which, together with $\beta\in(0,\infty)$ we find that $|J|>0$. Moreover, notice that
$$
\lim_{\beta\rightarrow 0}|J|=+\infty, \quad \lim_{\beta\rightarrow +\infty}|J|=0.
$$
\end{proof}

By the previous Proposition \ref{prop-jacobian-inv} we have that the jacobian is not vanishing and then we can invert the change of coordinates and find $(\beta(z_1,z_2), \phi(z_1,z_2))$, for $(\beta,\phi)\in(0,\infty)\times[0,2\pi]$. Since the jacobian of $(z_1,z_2)\mapsto (\beta,\phi))$ is $|J|^{-1}$ which is vanishing at $\beta=0$, we will identify such point with $(z_1,z_2)=(0,0)$.

Indeed, we shall first use the polar coordinates and then the main task is to invert
$$
(r(\beta,\phi), \vartheta(\beta,\phi))=(\tilde{r},\tilde{\vartheta}),
$$
finding $\beta=\beta(\tilde{r},\tilde{\vartheta})$ and $\phi=\phi(\tilde{r},\tilde{\vartheta})$, which is possible due to Proposition \ref{prop-jacobian-inv} by using the Inverse Function theorem.  More precisely, in the next proposition we understand better the functions $\beta(r,\vartheta)$ and $\phi(r,\vartheta)$. For that, let us define $Z^{\sigma,\alpha}$ as
$$
Z^{\sigma,\alpha}:=\left\{\tilde{\beta}\in C_b([0,\infty)\times[0,2\pi]), \quad \tilde{\beta},r\partial_r \tilde{\beta}\in L^\infty_{\sigma(1+\alpha)}\right\}.
$$

\begin{pro}\label{prop-invert}
Consider $(r,\vartheta):[0,\infty)\times[0,2\pi]\rightarrow [0,\infty)\times[0,2\pi]$ in \eqref{r-vartheta}. There exists $\varepsilon>0$ such that for $f\in B_{X_1^{\sigma,\delta}}(0,\varepsilon)$ the following is satisfied. There exists a function $\tilde{\beta}\in Z^{\sigma,\alpha}$ such that defining $$
\beta(\tilde{r},\tilde{\vartheta})=r^{-1-\alpha}(1+\tilde{\beta}(\tilde{r},\tilde{\vartheta})),\quad \phi(\tilde{r},\tilde{\vartheta})=\tilde{\vartheta}-\beta(\tilde{r},\tilde{\vartheta}),$$ we find
$$
(r(\beta(\tilde{r},\tilde{\vartheta}),\phi(\tilde{r},\tilde{\vartheta})),\vartheta(\beta(\tilde{r},\tilde{\vartheta}),\phi(\tilde{r},\tilde{\vartheta})))=(\tilde{r},\tilde{\vartheta}),
$$
for any
$(\tilde{r},\tilde{\vartheta})\in [0,\infty)\times[0,2\pi]$.
\end{pro}
\begin{proof}
Using \eqref{r-vartheta}, our main task is to invert
$$
\left((-(1+\alpha)\Psi_\beta)^\frac12,\beta+\phi\right)=(r,\vartheta),
$$
for any $(r,\vartheta)\in[0,+\infty)\times[0,2\pi]$. Hence, we want to find two functions $\beta,\phi:[0,+\infty)\times[0,2\pi]\rightarrow [0,+\infty)\times[0,2\pi]$ such that
\begin{align*}
(-(1+\alpha)\Psi_\beta(\beta(r,\vartheta),\phi(r,\vartheta))^\frac12=&r,\\
\beta(r,\vartheta)+\phi(r,\vartheta)=&\vartheta.
\end{align*}
From the second equation we find
\begin{equation}\label{phi-r-theta}
\phi(r,\vartheta)=\vartheta-\beta(r,\vartheta),
\end{equation}
and inserting it in the first equation we shall solve
\begin{equation}\label{eq-inver-1}
(-(1+\alpha)\Psi_\beta(\beta(r,\vartheta),\vartheta-\beta(r,\vartheta))^\frac12-r=0,
\end{equation}
for any $(r,\vartheta)\in[0,+\infty)\times[0,2\pi]$. Let us recall that $\Psi$ is indeed a perturbation of the trivial one, and by \eqref{Psibeta} we find
$$
(-\Psi_{\beta})^\frac12(\beta,\phi)=C_0^{\frac{1}{1+\alpha}}(\alpha-1)^{\frac{1-\alpha}{2(1+\alpha)}}\beta^{\frac{-1}{1+\alpha}}\left(\frac{\alpha-1}{1+\alpha}+{f}_{\overline{\beta}}(\beta,\phi)\right)^\frac12,
$$
for $f\in B_{X_1^{\sigma,\delta}}(0,\varepsilon)$. Hence, \eqref{eq-inver-1} agrees with
\begin{equation}\label{eq-inver-2}
(1+\alpha)^\frac12C_0^{\frac{1}{1+\alpha}}(\alpha-1)^{\frac{1-\alpha}{2(1+\alpha)}}\beta^{\frac{-1}{1+\alpha}}(r,\vartheta)\left(\frac{\alpha-1}{1+\alpha}+{f}_{\overline{\beta}}(\beta(r,\vartheta),\vartheta-\beta(r,\vartheta))\right)^\frac12-r=0.
\end{equation}
Define now 
$$
\beta_0(r)=C_0(\alpha-1)r^{-1-\alpha},
$$
agreeing with the change of coordinates for the trivial solution, which was given in Proposition \ref{prop-trivial}. Hence, decompose $\beta$ as
$$
\beta=\beta_0(1+\tilde{\beta}),
$$
and hence we can write \eqref{eq-inver-2} as $G(\tilde{\beta},f)(r,\vartheta)=0$ for any $(r,\vartheta)\in (0,+\infty)\times[0,2\pi]$, where
\begin{equation}\label{eq-inver-3}
G(\tilde{\beta},f)(r,\vartheta):=\frac{(1+\alpha)^\frac12}{(\alpha-1)^\frac12}(1+\tilde{\beta}(r,\vartheta))^{\frac{-1}{1+\alpha}}\left(\frac{\alpha-1}{1+\alpha}+{f}_{\overline{\beta}}(\beta(r,\vartheta),\vartheta-\beta(r,\vartheta))\right)^\frac12-1.
\end{equation}

We claim that there exists $\varepsilon<1$ such that $G:\times B_{Z^{\sigma,\alpha}}(0,\varepsilon)\times B_{X_1^{\sigma,\delta}(0,\varepsilon)}\rightarrow Z^{\sigma,\alpha}$ is well-defined and $C^1$. Hence, we would like to find nontrivial $(\tilde{\beta},f)$ such that $G(\tilde{\beta},f)(r)=0$ for any $r\in[0,+\infty)$. We have the trivial solution given by
$$
G(0,0)(r,\vartheta)=0, \quad \forall (r,\vartheta)\in[0,+\infty)\times[0,2\pi].
$$
We wish to implement the Implicit Function theorem there and thus we need to compute the linearized operator:
\begin{align*}
\partial_{\tilde{\beta}}G(0,0)h(r,\vartheta)=-\frac{1}{1+\alpha}h(r,\vartheta),
\end{align*}
which is clearly an isomorphism from $Z^{\sigma,\alpha}$ into $Z^{\sigma,\alpha}$.

Finally, it remains to check the well-definition of $G$. For that, let us write $G$ as follows
\begin{align*}
G(\tilde{\beta},f)(r)=&\frac{(1+\alpha)^\frac12}{(\alpha-1)^\frac12}\frac{(1+\tilde{\beta}(r,\vartheta))^{\frac{-2}{1+\alpha}}\left(\frac{\alpha-1}{1+\alpha}+{f}_{\overline{\beta}}(\beta(r,\vartheta),\vartheta-\beta(r,\vartheta))\right)-\frac{(\alpha-1)}{(1+\alpha)}}{(1+\tilde{\beta}(r,\vartheta))^{\frac{-1}{1+\alpha}}\left(\frac{\alpha-1}{1+\alpha}+{f}_{\overline{\beta}}(\beta(r,\vartheta),\vartheta-\beta(r,\vartheta))\right)^\frac12+\frac{(\alpha-1)^\frac12}{(1+\alpha)^\frac12}}\\
=&\frac{(1+\alpha)^\frac12}{(\alpha-1)^\frac12}\frac{\left[(1+\tilde{\beta}(r,\vartheta))^{\frac{-2}{1+\alpha}}-1\right]\left(\frac{\alpha-1}{1+\alpha}+{f}_{\overline{\beta}}(\beta(r,\vartheta),\vartheta-\beta(r,\vartheta))\right)}{(1+\tilde{\beta}(r,\vartheta))^{\frac{-1}{1+\alpha}}\left(\frac{\alpha-1}{1+\alpha}+{f}_{\overline{\beta}}(\beta(r,\vartheta),\vartheta-\beta(r,\vartheta))\right)^\frac12+\frac{(\alpha-1)^\frac12}{(1+\alpha)^\frac12}}\\
&+  \frac{(1+\alpha)^\frac12}{(\alpha-1)^\frac12} \frac{{f}_{\overline{\beta}}(\beta(r,\vartheta),\vartheta-\beta(r,\vartheta))}{(1+\tilde{\beta}(r,\vartheta))^{\frac{-1}{1+\alpha}}\left(\frac{\alpha-1}{1+\alpha}+{f}_{\overline{\beta}}(\beta(r,\vartheta),\vartheta-\beta(r,\vartheta))\right)^\frac12+\frac{(\alpha-1)^\frac12}{(1+\alpha)^\frac12}}.
\end{align*}
Keep in mind that $f\in X_1^{\sigma,\delta}$. First, note that $f_{\overline{\beta}}\in L^\infty_\sigma$. Assuming that $\tilde{\beta}\in L^\infty$, we have that 
$$
\beta(r,\vartheta) \sim_{0,\infty} r^{-1-\alpha},
$$
and hence
\begin{equation}\label{f-asimp-r}
f_{\overline{\beta}}(\beta(r,\vartheta),\vartheta-\beta(r,\vartheta))\sim_{0,\infty} \frac{r^{\sigma(1+\alpha)}}{(1+r)^{2\sigma(1+\alpha)}}.
\end{equation}
Hence, it is natural to assume that $\tilde{\beta}\in L^\infty_{\sigma(1+\alpha)}$. Moreover, for $\varepsilon$ small and $\tilde{\beta}\in B_{ L^\infty_{\sigma(1+\alpha)}}(0,\varepsilon)$ we have that all the denominator in the previous expression of $G$ are bounded. Hence, from \eqref{f-asimp-r} we find that $G(\tilde{\beta},f)\in L^\infty_{\sigma(1+\alpha)}$.

Since in the expression of $G$ we have $\partial_{\overline{\beta}}f(\beta(r,\vartheta),\vartheta-\beta(r,\vartheta))$, it suggests us that we can only take one derivative in $r$. Indeed
$$
\partial_r (\partial_{\overline{\beta}}f(\beta(r,\vartheta),\vartheta-\beta(r,\vartheta)))=(\partial_\varphi \partial_{\overline{\beta}}f)(\beta(r,\vartheta),\vartheta-\beta(r,\vartheta))\partial_r \beta(r,\vartheta).
$$
Since $\beta(\partial_\varphi \partial_{\overline{\beta}}f)\in L^\infty_\sigma$ we find that
$$
r\partial_r (\partial_{\overline{\beta}}f(\beta(r,\vartheta),\vartheta-\beta(r,\vartheta)))\sim_{0,\infty} \frac{r^{\sigma(1+\alpha)}}{(1+r)^{2\sigma(1+\alpha)}}.
$$
That motivates us to include $r\partial_r \tilde{\beta}\in L^\infty_{\sigma(1+\alpha)}$ in the space $Z^{\sigma,\alpha}$.  With such space, we have the well-definition of $G$ and then we conclude the proof by means of the Implicit Function theorem.

By \eqref{phi-r-theta} we find that $\phi\in Z_0^{\sigma,\alpha}$, where 
$$
Z_0^{\sigma,\alpha}:=\left\{\tilde{\beta}\in C_b([0,\infty)\times[0,2\pi]), \quad \tilde{\beta},r\partial_r \tilde{\beta}\in L^\infty_{0,\sigma(1+\alpha)}\right\}.
$$
\end{proof}

\subsection{Solution in the physical variables}
From Theorem \ref{theo-1} together with Remark \ref{rem-th1} we have constructed solutions $(\Psi,\Theta)$ to \eqref{eq-self-sim-2}, which is nothing but the vorticity equation in the adapted coordinates, and \eqref{equation} which refers to the elliptic relation between $\Psi$ and $\Theta$. In particular, we find $(\Psi,\Theta)$ as
\begin{align*}
\Psi(\beta,\phi)=&-C_0^{\frac{2}{1+\alpha}}(\alpha-1)^{\frac{1-\alpha}{1+\alpha}}\beta^{\frac{\alpha-1}{1+\alpha}}(1+\tilde{f}(\tilde{\Omega})(\beta,\phi)),\\
\Theta(\beta,\phi)=&(1+\alpha)^{-\frac{1}{2\mu}}\left[C_0^{\frac{2}{1+\alpha}}(\alpha-1)^{\frac{1-\alpha}{1+\alpha}}\right]^{-\frac{1+\alpha-\gamma}{2}}\beta^{\frac{1+\alpha-\gamma}{1+\alpha}}\left(\frac{\alpha-1}{1+\alpha}-\tilde{f}_{\overline{\varphi}}(\tilde{\Omega})(\beta,\phi)\right)^{-\frac{1+\alpha-\gamma}{2}}\tilde{\Omega}(\phi),
\end{align*}
where $\tilde{\Omega}\in B_{X_2^p}(1,\varepsilon)$ and $f(\tilde{\Omega})\in B_{X_1^{\sigma,\delta}}(0,\varepsilon)$.

Let us now come back to $\hat{\theta}$ by inverting the adapted coordinates using Proposition \ref{prop-invert}. Denote $\tilde{\theta}(r,\vartheta)=\hat{\theta}(re^{i\vartheta})$ and then
\begin{align*}
\tilde{\theta}(r,\vartheta)=&\Theta(\beta(r,\vartheta),\phi(r,\vartheta))\\
=&\Theta\Big(C_0(\alpha-1)r^{-1-\alpha}(1+\tilde{\beta}(r,\vartheta)),\vartheta-C_0(\alpha-1)r^{-1-\alpha}(1+\tilde{\beta}(r,\vartheta))\Big)\\
=&(1+\alpha)^{-\frac{1}{2\mu}}(\alpha-1)^{\frac{1+\alpha-\gamma}{2}}r^{-(1+\alpha-\gamma)}\\
&\times \left(\frac{\alpha-1}{1+\alpha}-\tilde{f}_{\overline{\varphi}}(\tilde{\Omega})\Big(C_0(\alpha-1)r^{-1-\alpha}(1+\tilde{\beta}(r,\vartheta)),\vartheta-C_0(\alpha-1)r^{-1-\alpha}(1+\tilde{\beta}(r,\vartheta))\Big)\right)^{-\frac{1+\alpha-\gamma}{2}}\\
&\times\tilde{\Omega}\Big(\vartheta-C_0(\alpha-1)r^{-1-\alpha}(1+\tilde{\beta}(r,\vartheta))\Big),
\end{align*}
with $\tilde{\beta}\in Z^{\sigma,\alpha}$. From \eqref{theta-self-similar} we have that
$$
\theta(t,x)=\frac{1}{t^\frac{1+\alpha-\gamma}{1+\alpha}}\hat{\theta}\left(\frac{x}{t^{\frac{1}{1+\alpha}}}\right)=\frac{1}{t^\frac{1+\alpha-\gamma}{1+\alpha}}\tilde{\theta}\left(\frac{r}{t^{\frac{1}{1+\alpha}}},\vartheta\right),
$$
where $x=re^{i\vartheta}$. Hence, using that $f\in X_1^{\sigma,\delta}$ we get that the initial condition agrees with
$$
\theta_0(x)=r^{-(1+\alpha-\gamma)}\tilde{\Omega}(\vartheta),
$$
where $\tilde{\Omega}\in B_{X_2^p}(1,\varepsilon)$. 

Moreover, $\theta_0$ is locally integrable:
$$
||\theta_0||_{L_{\textnormal{loc}}^1(\R^2)}=||\tilde{\Omega}||_{L^1([0,2\pi])} \int_0^R r^{\gamma-\alpha}dr=\frac{1}{1-\alpha+\gamma}||\tilde{\Omega}||_{L^1([0,2\pi])} R^{1-\alpha+\gamma},
$$
for any $R<\infty$, where $1-\alpha+\gamma>0$. Note that $\tilde{\Omega}\in L^1$ since $\tilde{\Omega}\in L^p$, for some $p>\frac{1}{1-\gamma}$.

\subsection{Main result}
Our main result reads as follows.
\begin{theo}
Let $\gamma\in(0,1)$ and $\alpha\in(1,1+\gamma)$. Then, there exists $\varepsilon>0$ such that for any $\frac{2\pi}{m}$-periodic $\tilde{\Omega}\in B_{L^p([0,2\pi])}(1,\varepsilon)$ and $p>\frac{1}{1-\gamma}$, the initial condition
$$
\theta_0(r e^{i\vartheta})=r^{-(1+\alpha-\gamma)}\tilde{\Omega}(\vartheta),
$$
describes a self-similar solution of the type \eqref{theta-self-similar} for $\textnormal{(SQG)}_\gamma$, for any $m\geq 1$.
\end{theo}

\appendix

\section{Radial solutions}\label{ap-radial}
In this section, we review the expression of the stream function $\psi$ associated to a radial vorticity $\theta$.

\begin{pro}
If $f=(-\Delta)^\beta g$, with $\beta\in(-1,0)$ and $g(x)=g(|x|)$ radial, then
$$
f(x)=2\pi C_\beta\int_0^\infty G_0(r,|x|)g(r)rdr,
$$
where
$$
G_0(r,|x|)=\left\{
\begin{array}{ll}
\frac{1}{\Gamma(-\beta)}\left(\frac{2}{r}\right)^{2+2\beta}\Gamma(1+\beta){}_2F_1\left(1+\beta,1+\beta,1,\frac{|x|^2}{r^2}\right),& |x|<r,\\
\frac{1}{\Gamma(-\beta)}\left(\frac{2}{|x|}\right)^{2+2\beta}\Gamma(1+\beta){}_2F_1\left(1+\beta,1+\beta,1,\frac{r^2}{|x|^2}\right),& |x|>r.\\
\end{array}
\right.
$$
\end{pro}
\begin{proof}
We want to calculate the explicit kernels of
$$ f = (-\Delta)^{\beta}g$$
whenever $-1 < \beta < 0$ and for general $g$ of the type
$$ g(r,\theta) = g_k(r) \cos(k\theta), \quad g(r,\theta) = g_k(r) \sin(k\theta). $$
Note that we are indeed interested in radial functions $g$ corresponding to $k=0$, but let us compute it in general. To do so, we will calculate the Fourier transform:

$$f(x) = \int \hat{g}(\xi) |\xi|^{2\beta} e^{i x \cdot \xi} d\xi $$

We start calculating $\hat{g}(\xi)$. In the first case:
\begin{align*}
\hat{g}(\xi) = \int_{\R^2} e^{-i x \cdot \xi} g(x) dx
 = \int_0^{\infty} g_k(r)r  \int_0^{2\pi} \cos(k \theta) e^{-ir(\cos(\theta) \xi_1  + \sin(\theta)\xi_2)} d\theta dr
\end{align*}
Writing 

$$\cos(\psi) = \frac{\xi_1}{\sqrt{\xi_1^2 + \xi_2^2}}, \quad \sin(\psi) = \frac{\xi_2}{\sqrt{\xi_1^2 + \xi_2^2}}, \quad z = \sqrt{\xi_1^2+\xi_2^2}r, $$
the inner integral becomes

\begin{align*}
\int_0^{2\pi} e^{-iz \cos(\theta - \psi)}\cos(k\theta) d\theta
=& \cos(k\psi) \int_0^{2\pi} e^{-i z \cos(u)}\cos(ku) du\\
=& 2 \pi T_k\left(\frac{\xi_1}{\sqrt{\xi_1^2+\xi_2^2}}\right) (-i)^k J_k(\sqrt{\xi_1^2+\xi_2^2}r),
\end{align*}
where $J_k$ is the Bessel function of order $k$, $T_k$ is the $k$-th Chebychev polynomial of the first kind. In the second case:

\begin{align*}
\hat{g}(\xi) = \int_{\R^2} e^{-i x \cdot \xi} g(x) dx
 = \int_0^{\infty} g_k(r)r  \int_0^{2\pi} \sin(k \theta) e^{-ir(\cos(\theta) \xi_1  + \sin(\theta)\xi_2)} d\theta dr
\end{align*}
Writing 
$$\cos(\psi) = \frac{\xi_1}{\sqrt{\xi_1^2 + \xi_2^2}}, \quad \sin(\psi) = \frac{\xi_2}{\sqrt{\xi_1^2 + \xi_2^2}},\quad  z = \sqrt{\xi_1^2+\xi_2^2}r $$
the inner integral becomes
\begin{align*}
\int_0^{2\pi} e^{-iz \cos(\theta - \psi)}\sin(k\theta) d\theta
=& \sin(k\psi) \int_0^{2\pi} e^{-i z \cos(u)}\sin(ku) du\\
=& 2 \pi \frac{\xi_2}{\sqrt{\xi_1^2+\xi_2^2}} U_{k-1}\left(\frac{\xi_1}{\sqrt{\xi_1^2+\xi_2^2}}\right) (-i)^k J_k(\sqrt{\xi_1^2+\xi_2^2}r),
\end{align*}
where $U_k$ is the $k$-th Chebychev polynomial of the second kind.

We now move on to the computation of the outer integral. Switching the integrals in $r$ and in $\xi$, and integrating first in $\xi$, in the first case:
\begin{align*}
& \int_0^{\infty} (-i)^k 2\pi g_k(r)r  \int_{\R^2} T_k\left(\frac{\xi_1}{\sqrt{\xi_1^2+\xi_2^2}}\right) J_k(\sqrt{\xi_1^2+\xi_2^2}r) |\xi|^{2\beta} e^{i x \cdot \xi} d\xi dr \\
& = \int_0^\infty  (-i)^k 2\pi g_k(r)r  \int_0^{\infty}  J_k(|\xi|r)|\xi|^{2\beta+1}\int_0^{2\pi} \cos(k\psi) e^{i |\xi|(x_1 \cos(\psi) + x_2 \sin(\psi))} d\psi d|\xi| dr\\
& = (-i)^{2k} (2\pi)^2 \cos(k\theta) \int_0^{\infty}  g_k(r)r  \int_0^{\infty}  J_k(|\xi|r)|\xi|^{2\beta+1}  J_k(|x||\xi|) d|\xi| dr
\end{align*}
Using that $-1 < \beta < 0$, we get
\begin{align*}
 & \int_0^\infty  J_k(|\xi|r)|\xi|^{2\beta+1}  J_k(|x||\xi|) d|\xi| \equiv G_{k}(r,|x|) \\
& = 
\left\{
\begin{array}{cc}
\frac{1}{2\Gamma(-\beta)}\left(\frac{2}{r}\right)^{2+2\beta} \left(\frac{|x|}{r}\right)^k \Gamma(1+\beta+k) \Gamma(k+1) _2 F_1(1+\beta,1+\beta+k,1+k, \frac{|x|^2}{r^2})  & \text{ if } |x| < r  \\
\frac{1}{2\Gamma(-\beta)}\left(\frac{2}{x}\right)^{2+2\beta} \left(\frac{r}{|x|}\right)^k \Gamma(1+\beta+k) \Gamma(k+1) _2 F_1(1+\beta,1+\beta+k,1+k, \frac{r^2}{|x|^2})  & \text{ if } |x| > r  \\
\end{array}
\right.
\end{align*}
where $_2 F_1$ is the Gauss Hypergeometric function. This yields:
\begin{align*}
 f(x) = (-i)^{2k} (2\pi)^2 \cos(k\theta)\int  g_k(r) G_k(r,|x|) r dr.
\end{align*}
The computation in the second case is similar, and we obtain:
\begin{align*}
 f(x) = (-i)^{2k} (2\pi)^2 \sin(k\theta)\int  g_k(r) G_k(r,|x|) r dr.
\end{align*}
Now, taking $k=0$ we arrive to the announced result.
\end{proof}
In the following, we apply the previous proposition to our trivial solution.
\begin{cor}\label{cor-radial}
If $\theta(x)=|x|^{-\frac{1}{\mu}}$, then $\psi(x)=-(-\Delta)^{-1+\frac{\gamma}{2}}\theta(x)=-C_0|x|^{2-\gamma-\frac{1}{\mu}}$, where
\begin{align}\label{C0}
\nonumber C_0=&\frac{ \pi 2^\gamma \Gamma\left(\frac{\gamma}{2}\right)\mu}{\Gamma\left(1-\frac{\gamma}{2}\right)}\left\{\frac{1}{2\mu-1}{}_3F_2\left(\left\{\frac{\gamma}{2},\frac{\gamma}{2},1-\frac{1}{2\mu}\right\},\left\{1, 2-\frac{1}{2\mu}\right\},1\right)\right.\\
&\left.+\frac{1}{1+\mu(-2+\gamma)}{}_3F_2\left(\left\{\frac{\gamma}{2},\frac{\gamma}{2},-1+\frac{\gamma}{2}+\frac{1}{2\mu}\right\},\left\{1, \frac{\gamma}{2}+\frac{1}{2\mu}\right\},1\right)\right\}.
\end{align}
\end{cor}
\begin{proof}
Use the previous proposition with $\beta=-1+\frac{\gamma}{2}$ and $g(x)=|x|^{-\frac{1}{\mu}}$.
\end{proof}

The constant $C_0$ will be important in the kernel study of the linearized operator. For that reason, let us give its expression in terms of the integral operator $T_k$ defined in \eqref{T-k}.
\begin{pro}\label{prop-C0-2}
The constant $C_0$ can be also written as
\begin{align*}
C_0=&\frac{C_\gamma}{2\pi}\int_0^\infty\int_0^{2\pi}\frac{z^{\gamma-\alpha}}{\left\{1+z^2-2z\cos(\eta)\right\}^\frac{\gamma}{2}}\\
=&\frac{C_\gamma}{2\pi}M[T_0](1-\alpha+\gamma).
\end{align*}
\end{pro}
\begin{proof}
From the definition of the $(-\Delta)^{-1+\frac{\gamma}{2}}$ operator one has
$$
C_0=r^{-2+\gamma+1/\mu}\frac{C_\gamma}{2\pi}\frac{C_\gamma}{2\pi}\int_0^\infty\int_0^{2\pi}\frac{z^{\gamma-\alpha}}{\left\{r^2+z^2-2rz\cos(\eta)\right\}^\frac{\gamma}{2}}.
$$
By performing the change of variable $z\mapsto rz$ and using the expression of $\mu$ in \eqref{mu} and $T_k$ in \eqref{T-k} we arrive to the announced identity.
\end{proof}


\section{Special functions}\label{ap-special}
This section aims to give an overview of some special functions such as the Gauss Hypergeometric function, the Gamma function or the Digamma function. At the end, we will also introduce the Mellin transform and some of its properties.

Define for any real numbers $a,b\in \mathbb{R},\, c\in \mathbb{R}\backslash(-\mathbb{N})$ the Gauss hypergeometric function $z\mapsto F(a,b;c;z)$ on the open unit disc $\mathbb{D}$ by the power series
\begin{equation}\label{GaussF}
F(a,b;c;z)=\sum_{n=0}^{\infty}\frac{(a)_n(b)_n}{(c)_n}\frac{z^n}{n!}, \quad \forall z\in \mathbb{D},
\end{equation}
where the Pochhammer's  symbol $(x)_n$ is defined by
$$
(x)_n = \begin{cases}   1,   & n = 0, \\
 x(x+1) \cdots (x+n-1), & n \geq1,
\end{cases}
$$
and verifies
\begin{equation*}
(x)_n=x\,(1+x)_{n-1},\quad (x)_{n+1}=(x+n)\,(x)_n.
\end{equation*}
Moreover, let us recall the integral representation of the hypergeometric function, see for instance  \cite[p. 47]{Rainville:special-functions}. Assume that  $ \textnormal{Re}(c) > \textnormal{Re}(b) > 0,$ then 
\begin{equation}\label{hypergeom-integral}
\hspace{1cm}F(a,b;c;z)=\frac{\Gamma(c)}{\Gamma(b)\Gamma(c-b)}\int_0^1 x^{b-1} (1-x)^{c-b-1}(1-zx)^{-a}~ dx,\quad \forall{z\in \C\backslash[1,+\infty)}.
\end{equation}

The function $\Gamma: \C\backslash\{-\N\} \to \C$ refers to the gamma function, which is the analytic continuation to the negative half plane of the usual gamma function defined on the positive half-plane $\{\textnormal{Re} z > 0\}$, and given by
$$
\Gamma(z)=\int_0^{+\infty}\tau^{z-1}e^{-\tau} d\tau,
$$
and satisfies the relation
$
\Gamma(z+1)=z\,\Gamma(z), \ \forall z\in \C \backslash(-\N).
$
Moreover, the Digamma function is defined through the logarithmic of the Gamma function as:
$$
\psi^0(z)=\frac{\Gamma'(z)}{\Gamma(z)},
$$
and has the following integral representation:
\begin{equation}\label{digamma-int}
\psi^0(z)=\int_0^1\frac{1-t^{z-1}}{1-t}dt-\gamma_0,
\end{equation}
for $\textnormal{Re}(z)>0$ and where $\gamma_0$ stands for the Euler constant.

In the following lemma, we give a representation of the Gauss Hypergeometric function in terms of some special trigonometric integrals, the proof can be found in \cite{Garcia-Hmidi-Mateu:time-periodic-3d-qg}.
\begin{lem}\label{lem-integral}
Let $k\in\N$, $\gamma\geq0$ and $A>1$, then
$$
\bigintsss_0^{2\pi}\frac{\cos(k\eta)}{(A-\cos(\eta))^{\frac{\gamma}{2}}}d\eta=\frac{2\pi}{(1+A)^{\frac{\gamma}{2}+k}}\frac{\left(\frac{\gamma}{2}\right)_k 2^k\left(\frac12\right)_k}{(2k)!}{}_2F_1\left(k+\frac{\gamma}{2}, k+\frac12; 2k+1; \frac{2}{1+A}\right ).
$$

\end{lem}

Finally, let us introduce the Mellin transform of $f$ defined as 
$$
M[f](s)=\int_0^\infty x^{s-1}f(x)dx.
$$
In the following lemma, we give some properties of the Mellin transform, we refer to \cite{Paris-Kaminski:book-mellin} for the proof and for more information about this transform.
\begin{lem}\label{lem-mellin}
The Mellin transform of $f$ has the following properties.
\begin{itemize}
\item If 
$$
(f\wedge g)(x)=\int_0^\infty f(y)g\left(\frac{x}{y}\right)\frac{dy}{y},
$$
then 
\begin{equation}\label{mellin-1}
M[f\wedge g](s)=M[f]M[g].
\end{equation}
\item 
\begin{equation}\label{mellin-2}
M[x^\alpha f](s)=M[f](s+\alpha).
\end{equation}
\item
\begin{equation}\label{mellin-3}
M[f'](s)=-(s-1)M[f](s-1).
\end{equation}
\item
\begin{equation}\label{mellin-4}
M[x^\alpha f](s)=M[f](s+\alpha).
\end{equation}
\item 
\begin{equation}\label{mellin-5}
M[f(x^a)]=|a|^{-1}M[f](s/a).
\end{equation}
\end{itemize}
\end{lem}

\bibliography{references}

\def\cprime{$'$}
\begin{thebibliography}{10}

\bibitem{Ao-Davila-delPino-Musso-Wei:travelling-rotating-solutions-gsqg}
Weiwei Ao, Juan D\'{a}vila, Manuel del Pino, Monica Musso, and Juncheng Wei.
\newblock Travelling and rotating solutions to the generalized inviscid surface
  quasi-geostrophic equation.
\newblock {\em Trans. Amer. Math. Soc.}, 374(9):6665--6689, 2021.

\bibitem{Azzam-Bedrossian:bmo-uniqueness-active-scalar-equations}
Jonas Azzam and Jacob Bedrossian.
\newblock Bounded mean oscillation and the uniqueness of active scalar
  equations.
\newblock {\em Trans. Amer. Math. Soc.}, 367(5):3095--3118, 2015.

\bibitem{Breschi-Fontelos:self-similar-fragmentation}
Giancarlo Breschi and Marco~A. Fontelos.
\newblock A note on the self-similar solutions to the spontaneous fragmentation
  equation.
\newblock {\em Proc. A.}, 473(2201):20160740, 13, 2017.

\bibitem{Bressan-Murray:self-similar-euler}
Alberto Bressan and Ryan Murray.
\newblock On self-similar solutions to the incompressible {E}uler equations.
\newblock {\em J. Differential Equations}, 269(6):5142--5203, 2020.

\bibitem{Bressan-Shen:a-posteriori-error-self-similar-euler}
Alberto Bressan and Wen Shen.
\newblock A posteriori error estimates for self-similar solutions to the
  {E}uler equations.
\newblock {\em Discrete Contin. Dyn. Syst.}, 41(1):113--130, 2021.

\bibitem{Buckmaster-Shkoller-Vicol:nonuniqueness-sqg}
Tristan Buckmaster, Steve Shkoller, and Vlad Vicol.
\newblock Nonuniqueness of weak solutions to the {SQG} equation.
\newblock {\em Comm. Pure Appl. Math.}, 72(9):1809--1874, 2019.

\bibitem{Burbea:motions-vortex-patches}
Jacob Burbea.
\newblock Motions of vortex patches.
\newblock {\em Lett. Math. Phys.}, 6(1):1--16, 1982.

\bibitem{Cannone-Xue:self-similar-solutions-sqg}
Marco Cannone and Liutang Xue.
\newblock Remarks on self-similar solutions for the surface quasi-geostrophic
  equation and its generalization.
\newblock {\em Proc. Amer. Math. Soc.}, 143(6):2613--2622, 2015.

\bibitem{Castro-Cordoba:infinite-energy-sqg}
A.~Castro and D.~C{\'o}rdoba.
\newblock Infinite energy solutions of the surface quasi-geostrophic equation.
\newblock {\em Adv. Math.}, 225(4):1820--1829, 2010.

\bibitem{Castro-Cordoba-GomezSerrano:existence-regularity-vstates-gsqg}
Angel Castro, Diego C{\'o}rdoba, and Javier G{\'o}mez-Serrano.
\newblock Existence and regularity of rotating global solutions for the
  generalized surface quasi-geostrophic equations.
\newblock {\em Duke Math. J.}, 165(5):935--984, 2016.

\bibitem{Castro-Cordoba-GomezSerrano:analytic-vstates-ellipses}
Angel Castro, Diego C{\'o}rdoba, and Javier G{\'o}mez-Serrano.
\newblock Uniformly rotating analytic global patch solutions for active
  scalars.
\newblock {\em Annals of PDE}, 2(1):1--34, 2016.

\bibitem{Castro-Cordoba-GomezSerrano:uniformly-rotating-smooth-euler}
Angel Castro, Diego C\'{o}rdoba, and Javier G\'{o}mez-Serrano.
\newblock Uniformly rotating smooth solutions for the incompressible 2{D}
  {E}uler equations.
\newblock {\em Arch. Ration. Mech. Anal.}, 231(2):719--785, 2019.

\bibitem{Castro-Cordoba-GomezSerrano:global-smooth-solutions-sqg}
Angel Castro, Diego C{\'o}rdoba, and Javier G{\'o}mez-Serrano.
\newblock Global smooth solutions for the inviscid {S}{Q}{G} equation.
\newblock {\em Memoirs of the AMS}, 266(1292):89 pages, 2020.

\bibitem{Chae:nonexistence-selfsimilar-euler-3d-sqg}
Dongho Chae.
\newblock Nonexistence of self-similar singularities for the 3{D}
  incompressible {E}uler equations.
\newblock {\em Comm. Math. Phys.}, 273(1):203--215, 2007.

\bibitem{Chae-Constantin-Cordoba-Gancedo-Wu:gsqg-singular-velocities}
Dongho Chae, Peter Constantin, Diego C{\'o}rdoba, Francisco Gancedo, and
  Jiahong Wu.
\newblock Generalized surface quasi-geostrophic equations with singular
  velocities.
\newblock {\em Comm. Pure Appl. Math.}, 65(8):1037--1066, 2012.

\bibitem{Cheng-Kwon-Li:non-uniqueness-steady-state-sqg}
Xinyu Cheng, Hyunju Kwon, and Dong Li.
\newblock Non-uniqueness of steady-state weak solutions to the surface
  quasi-geostrophic equations.
\newblock {\em Comm. Math. Phys.}, 388(3):1281--1295, 2021.

\bibitem{Cieslak-Kokocki-Ozanski:well-posedness-logarithmic-spiral-vortex-sheets}
Tomasz Cie\'slak, Piotr Kokocki, and Wojciech~S. Oza\'nski.
\newblock Well-posedness of logarithmic spiral vortex sheets.
\newblock {\em Arxiv preprint arXiv:2110.07543}, 2021.

\bibitem{Cieslak-Kokocki-Ozanski:nonsymmetric-logarithmic-spiral-vortex-sheets}
Tomasz Cie\'slak, Piotr Kokocki, and Wojciech~S. Oza\'nski.
\newblock Existence of nonsymmetric logarithmic spiral vortex sheet solutions
  to the 2d {E}uler equations.
\newblock {\em Arxiv preprint arXiv:2207.06056}, 2022.

\bibitem{Constantin-Lai-Sharma-Tseng-Wu:new-numerics-sqg}
Peter Constantin, Ming-Chih Lai, Ramjee Sharma, Yu-Hou Tseng, and Jiahong Wu.
\newblock New numerical results for the surface quasi-geostrophic equation.
\newblock {\em J. Sci. Comput.}, 50(1):1--28, 2012.

\bibitem{Constantin-Majda-Tabak:formation-fronts-qg}
Peter Constantin, Andrew~J. Majda, and Esteban Tabak.
\newblock Formation of strong fronts in the {$2$}-{D} quasigeostrophic thermal
  active scalar.
\newblock {\em Nonlinearity}, 7(6):1495--1533, 1994.

\bibitem{Constantin-Nie-Schorghofer:nonsingular-sqg-flow}
Peter Constantin, Qing Nie, and Norbert Sch{\"o}rghofer.
\newblock Nonsingular surface quasi-geostrophic flow.
\newblock {\em Phys. Lett. A}, 241(3):168--172, 1998.

\bibitem{Cordoba:nonexistence-hyperbolic-blowup-qg}
Diego Cordoba.
\newblock Nonexistence of simple hyperbolic blow-up for the quasi-geostrophic
  equation.
\newblock {\em Ann. of Math. (2)}, 148(3):1135--1152, 1998.

\bibitem{Cordoba-Fefferman:growth-solutions-qg-2d-euler}
Diego Cordoba and Charles Fefferman.
\newblock Growth of solutions for {QG} and 2{D} {E}uler equations.
\newblock {\em J. Amer. Math. Soc.}, 15(3):665--670, 2002.

\bibitem{Cordoba-Fontelos-Mancho-Rodrigo:evidence-singularities-contour-dynamics}
Diego C{\'o}rdoba, Marco~A. Fontelos, Ana~M. Mancho, and Jose~L. Rodrigo.
\newblock Evidence of singularities for a family of contour dynamics equations.
\newblock {\em Proc. Natl. Acad. Sci. USA}, 102(17):5949--5952, 2005.

\bibitem{Cordoba-GomezSerrano-Ionescu:global-generalized-sqg-patch}
Diego C\'{o}rdoba, Javier G\'{o}mez-Serrano, and Alexandru~D. Ionescu.
\newblock Global {S}olutions for the {G}eneralized {SQG} {P}atch {E}quation.
\newblock {\em Arch. Ration. Mech. Anal.}, 233(3):1211--1251, 2019.

\bibitem{Davila-DelPino-Musso-Wei:travelling-helices-vortex-filament}
Juan D\'{a}vila, Manuel del Pino, Monica Musso, and Juncheng Wei.
\newblock Travelling helices and the vortex filament conjecture in the
  incompressible {E}uler equations.
\newblock {\em Calc. Var. Partial Differential Equations}, 61(4):Paper No. 119,
  30, 2022.

\bibitem{delaHoz-Hassainia-Hmidi:doubly-connected-vstates-gsqg}
Francisco de~la Hoz, Zineb Hassainia, and Taoufik Hmidi.
\newblock Doubly connected {V}-states for the generalized surface
  quasi-geostrophic equations.
\newblock {\em Arch. Ration. Mech. Anal.}, 220(3):1209--1281, 2016.

\bibitem{Deng-Hou-Li-Yu:non-blowup-2d-sqg}
J.~Deng, T.~Y. Hou, R.~Li, and X.~Yu.
\newblock Level set dynamics and the non-blowup of the 2{D} quasi-geostrophic
  equation.
\newblock {\em Methods Appl. Anal.}, 13(2):157--180, 2006.

\bibitem{Dritschel:repeated-filamentation-2d-vorticity}
David~G. Dritschel.
\newblock The repeated filamentation of two-dimensional vorticity interfaces.
\newblock {\em J. Fluid Mech.}, 194:511--547, 1988.

\bibitem{Elgindi-Jeong:singular-vortex-patches-i}
Tarek~M. Elgindi and In-Jee Jeong.
\newblock On singular vortex patches, {I}: Well-posedness issues.
\newblock {\em Mem. Amer. Math. Soc.}, 2019.
\newblock To appear.

\bibitem{Elgindi-Jeong:singular-vortex-patches-ii}
Tarek~M. Elgindi and In-Jee Jeong.
\newblock On singular vortex patches, {II}: long-time dynamics.
\newblock {\em Trans. Amer. Math. Soc.}, 373(9):6757--6775, 2020.

\bibitem{Elgindi-Jeong:symmetries-fluids}
Tarek~M. Elgindi and In-Jee Jeong.
\newblock Symmetries and critical phenomena in fluids.
\newblock {\em Comm. Pure Appl. Math.}, 73(2):257--316, 2020.

\bibitem{Elling:algebraic-spiral-solutions-2d-euler}
Volker Elling.
\newblock Algebraic spiral solutions of the 2d incompressible {E}uler
  equations.
\newblock {\em Bull. Braz. Math. Soc. (N.S.)}, 47(1):323--334, 2016.

\bibitem{Elling:self-similar-euler-mixed-sign-vorticity}
Volker Elling.
\newblock Self-similar 2d {E}uler solutions with mixed-sign vorticity.
\newblock {\em Comm. Math. Phys.}, 348(1):27--68, 2016.

\bibitem{Escobedo-Velazquez:fundamental-solution-coagulation}
Miguel Escobedo and J.~J.~L. Vel\'{a}zquez.
\newblock On the fundamental solution of a linearized homogeneous coagulation
  equation.
\newblock {\em Comm. Math. Phys.}, 297(3):759--816, 2010.

\bibitem{Fontelos-delaHoz:singularities-water-waves}
M.~A. Fontelos and F.~de~la Hoz.
\newblock Singularities in water waves and the {R}ayleigh-{T}aylor problem.
\newblock {\em J. Fluid Mech.}, 651:211--239, 2010.

\bibitem{Friedlander-Shvydkoy:unstable-spectrum-sqg}
Susan Friedlander and Roman Shvydkoy.
\newblock The unstable spectrum of the surface quasi-geostropic equation.
\newblock {\em J. Math. Fluid Mech.}, 7(suppl. 1):S81--S93, 2005.

\bibitem{Gancedo-Strain:absence-splash-muskat-SQG}
F.~{Gancedo} and R.~M. {Strain}.
\newblock Absence of splash singularities for surface quasi-geostrophic sharp
  fronts and the {M}uskat problem.
\newblock {\em Proceedings of the National Academy of Sciences},
  111(2):635--639, 2014.

\bibitem{Gancedo:existence-alpha-patch-sobolev}
Francisco Gancedo.
\newblock Existence for the {$\alpha$}-patch model and the {QG} sharp front in
  {S}obolev spaces.
\newblock {\em Adv. Math.}, 217(6):2569--2598, 2008.

\bibitem{Gancedo-Patel:local-existence-blowup-gsqg}
Francisco Gancedo and Neel Patel.
\newblock On the local existence and blow-up for generalized {S}{Q}{G} patches.
\newblock {\em Arxiv preprint arXiv:1811.00530}, 2018.

\bibitem{Garcia:Karman-vortex-street}
Claudia Garc\'{\i}a.
\newblock K\'{a}rm\'{a}n vortex street in incompressible fluid models.
\newblock {\em Nonlinearity}, 33(4):1625--1676, 2020.

\bibitem{Garcia:vortex-patch-choreography}
Claudia Garc\'{\i}a.
\newblock Vortex patches choreography for active scalar equations.
\newblock {\em J. Nonlinear Sci.}, 31(5):Paper No. 75, 31, 2021.

\bibitem{Garcia-Hmidi-Mateu:time-periodic-doubly-3d-qg}
Claudia Garc\'{\i}a, Taoufik Hmidi, and Joan Mateu.
\newblock Time periodic doubly connected solutions for 3{D} quasi-geostrophic
  model.
\newblock {\em Arxiv preprint arXiv:2206:10197}, 2022.

\bibitem{Garcia-Hmidi-Mateu:time-periodic-3d-qg}
Claudia Garc\'{\i}a, Taoufik Hmidi, and Joan Mateu.
\newblock Time periodic solutions for 3{D} quasi-geostrophic model.
\newblock {\em Comm. Math. Phys.}, 390(2):617--756, 2022.

\bibitem{Garcia-Hmidi-Soler:non-uniform-vstates-euler}
Claudia Garc\'{\i}a, Taoufik Hmidi, and Juan Soler.
\newblock Non uniform rotating vortices and periodic orbits for the
  two-dimensional {E}uler equations.
\newblock {\em Arch. Ration. Mech. Anal.}, 238(2):929--1085, 2020.

\bibitem{GomezSerrano:stationary-patches}
Javier G\'{o}mez-Serrano.
\newblock On the existence of stationary patches.
\newblock {\em Adv. Math.}, 343:110--140, 2019.

\bibitem{Gravejat-Smets:travelling-waves-smooth-sqg}
Philippe Gravejat and Didier Smets.
\newblock {Smooth Travelling-Wave Solutions to the Inviscid Surface
  Quasi-Geostrophic Equation}.
\newblock {\em International Mathematics Research Notices}, 2019(6):1744--1757,
  2017.

\bibitem{Hassainia-Hmidi:v-states-generalized-sqg}
Zineb Hassainia and Taoufik Hmidi.
\newblock On the {V}-states for the generalized quasi-geostrophic equations.
\newblock {\em Comm. Math. Phys.}, 337(1):321--377, 2015.

\bibitem{Hassainia-Hmidi-Masmoudi:kam-gsqg}
Zineb Hassainia, Taoufik Hmidi, and Nader Masmoudi.
\newblock {K}{A}{M} theory for active scalar equations.
\newblock {\em Arxiv preprint arXiv:2110.08615}, 2021.

\bibitem{Held-Pierrehumbert-Garner-Swanson:sqg-dynamics}
Isaac~M. Held, Raymond~T. Pierrehumbert, Stephen~T. Garner, and Kyle~L.
  Swanson.
\newblock Surface quasi-geostrophic dynamics.
\newblock {\em J. Fluid Mech.}, 282:1--20, 1995.

\bibitem{Hmidi-Mateu:existence-corotating-counter-rotating}
Taoufik Hmidi and Joan Mateu.
\newblock Existence of corotating and counter-rotating vortex pairs for active
  scalar equations.
\newblock {\em Comm. Math. Phys.}, 350(2):699--747, 2017.

\bibitem{Hmidi-Mateu-Verdera:rotating-vortex-patch}
Taoufik Hmidi, Joan Mateu, and Joan Verdera.
\newblock Boundary regularity of rotating vortex patches.
\newblock {\em Archive for Rational Mechanics and Analysis}, 209(1):171--208,
  2013.

\bibitem{Hunter-Shu-Zhang:global-gsqg}
John~K. Hunter, Jingyang Shu, and Qingtian Zhang.
\newblock Global solutions for a family of {G}{S}{Q}{G} front equations.
\newblock {\em Arxiv preprint arXiv:2005.09154}, 2020.

\bibitem{Hunter-Shu-Zhang:global-model-front-sqg}
John~K. Hunter, Jingyang Shu, and Qingtian Zhang.
\newblock Global solutions of a surface quasigeostrophic front equation.
\newblock {\em Pure Appl. Anal.}, 3(3):403--472, 2021.

\bibitem{Isett-Ma:non-uniqueness-sqg}
Philip Isett and Andrew Ma.
\newblock A direct approach to nonuniqueness and failure of compactness for the
  {SQG} equation.
\newblock {\em Nonlinearity}, 34(5):3122--3162, 2021.

\bibitem{Isett-Vicol:holder-continuous-active-scalar}
Philip Isett and Vlad Vicol.
\newblock H{\"o}lder continuous solutions of active scalar equations.
\newblock {\em Annals of PDE}, 1(1):1--77, 2015.

\bibitem{Kaden:spiral}
Heinrich Kaden.
\newblock Aufwicklung einer unstabilen unstetigkeitsfl{\"a}che.
\newblock {\em Ingenieur-Archiv}, 2(2):140--168, 1931.

\bibitem{Kiselev-Ryzhik-Yao-Zlatos:singularity-alpha-patch-boundary}
A.~Kiselev, L.~Ryzhik, Y.~Yao, and A.~Zlato\v{s}.
\newblock Finite time singularity formation for the modified {S}{Q}{G} patch
  equation.
\newblock {\em Ann. of Math. (2)}, 184:909--948, 2016.

\bibitem{Kiselev-Nazarov:simple-energy-pump-sqg}
Alexander Kiselev and Fedor Nazarov.
\newblock A simple energy pump for the surface quasi-geostrophic equation.
\newblock In Helge Holden and Kenneth~H. Karlsen, editors, {\em Nonlinear
  Partial Differential Equations}, volume~7 of {\em Abel Symposia}, pages
  175--179. Springer Berlin Heidelberg, 2012.

\bibitem{Majda-Bertozzi:vorticity-incompressible-flow}
Andrew~J. Majda and Andrea~L. Bertozzi.
\newblock {\em Vorticity and incompressible flow}, volume~27 of {\em Cambridge
  Texts in Applied Mathematics}.
\newblock Cambridge University Press, Cambridge, 2002.

\bibitem{Mancho:numerical-studies-self-similar-alpha-patch}
Ana~M Mancho.
\newblock Numerical studies on the self-similar collapse of the alpha-patches
  problem.
\newblock {\em Communications in Nonlinear Science and Numerical Simulation},
  26(1-3):152--166, 2015.

\bibitem{Marchand:existence-regularity-weak-solutions-sqg}
Fabien Marchand.
\newblock Existence and regularity of weak solutions to the quasi-geostrophic
  equations in the spaces {$L^p$} or {$\dot H^{-1/2}$}.
\newblock {\em Comm. Math. Phys.}, 277(1):45--67, 2008.

\bibitem{Nahmod-Pavlovic-Staffilani-Totz:global-invariant-measures-gsqg}
Andrea~R. Nahmod, Nata\v{s}a Pavlovi\'c, Gigliola Staffilani, and Nathan Totz.
\newblock Global flows with invariant measures for the inviscid modified {SQG}
  equations.
\newblock {\em Stoch. Partial Differ. Equ. Anal. Comput.}, 6(2):184--210, 2018.

\bibitem{Ohkitani-Yamada:inviscid-limit-sqg}
Koji Ohkitani and Michio Yamada.
\newblock Inviscid and inviscid-limit behavior of a surface quasigeostrophic
  flow.
\newblock {\em Phys. Fluids}, 9(4):876--882, 1997.

\bibitem{Pacard-Riviere:book-vortices}
Frank Pacard and Tristan Rivi\`ere.
\newblock {\em Linear and nonlinear aspects of vortices}, volume~39 of {\em
  Progress in Nonlinear Differential Equations and their Applications}.
\newblock Birkh\"{a}user Boston, Inc., Boston, MA, 2000.
\newblock The Ginzburg-Landau model.

\bibitem{Paris-Kaminski:book-mellin}
R.~B. Paris and D.~Kaminski.
\newblock {\em Asymptotics and {M}ellin-{B}arnes integrals}, volume~85 of {\em
  Encyclopedia of Mathematics and its Applications}.
\newblock Cambridge University Press, Cambridge, 2001.

\bibitem{Pedlosky:geophysical}
Joseph Pedlosky.
\newblock Geophysical fluid dynamics.
\newblock {\em New York and Berlin, Springer-Verlag}, 1, 1982.

\bibitem{Pullin:spiral-vortex-sheet}
D.~I. Pullin.
\newblock On similarity flows containing two-branched vortex sheets.
\newblock In {\em Mathematical aspects of vortex dynamics ({L}eesburg, {VA},
  1988)}, pages 97--106. SIAM, Philadelphia, PA, 1989.

\bibitem{Rainville:special-functions}
Earl~D. Rainville.
\newblock {\em Special functions}.
\newblock The Macmillan Co., New York, 1960.

\bibitem{Renault:relative-equlibria-holes-sqg}
Coralie Renault.
\newblock Relative equilibria with holes for the surface quasi-geostrophic
  equations.
\newblock {\em J. Differential Equations}, 263(1):567--614, 2017.

\bibitem{Resnick:phd-thesis-sqg-chicago}
Serge~G Resnick.
\newblock {\em Dynamical problems in non-linear advective partial differential
  equations}.
\newblock PhD thesis, University of Chicago, Department of Mathematics, 1995.

\bibitem{Rodrigo:evolution-sharp-fronts-qg}
Jos{\'e}~Luis Rodrigo.
\newblock On the evolution of sharp fronts for the quasi-geostrophic equation.
\newblock {\em Comm. Pure Appl. Math.}, 58(6):821--866, 2005.

\bibitem{Scott-Dritschel:self-similar-sqg}
R.~K. Scott and D.~G. Dritschel.
\newblock Numerical simulation of a self-similar cascade of filament
  instabilities in the surface quasigeostrophic system.
\newblock {\em Phys. Rev. Lett.}, 112:144505, 2014.

\bibitem{Scott-Dritschel:self-similar-sqg-patch}
R.~K. Scott and D.~G. Dritschel.
\newblock Scale-invariant singularity of the surface quasigeostrophic patch.
\newblock {\em Journal of Fluid Mechanics}, 863:R2, 2019.

\bibitem{Scott:scenario-singularity-quasigeostrophic}
Richard~K. Scott.
\newblock A scenario for finite-time singularity in the quasigeostrophic model.
\newblock {\em Journal of Fluid Mechanics}, 687:492--502, 11 2011.

\end{thebibliography}
\bibliographystyle{plain}

\end{document}